\documentclass[a4paper,11pt]{article}
\usepackage[textwidth=16cm,textheight=24cm,headheight=14pt]{geometry}
\usepackage{amsmath,amsthm,amssymb,amsfonts,bm,bbm,xcolor,graphicx,fancyhdr,lastpage,mathtools,comment}
\usepackage[normalem]{ulem} 
\usepackage[shortlabels]{enumitem}
\usepackage{cmap,lmodern} \usepackage[T1]{fontenc} %make pdfs searchable
\definecolor{navy}{RGB}{20,0,105}
\usepackage[pdfauthor   = {Us}, pdftitle    = {}, pdfsubject  = {}, pdfkeywords = {}, bookmarks=true, bookmarksopen=true, pagebackref=false, hyperindex=true, colorlinks=true, allcolors=navy]{hyperref} 
\usepackage[capitalise]{cleveref}		
\renewcommand{\namecref}{\lcnamecref}

\usepackage[titletoc,title]{appendix}
\usepackage[backend=bibtex,giveninits=true,maxbibnames=9]{biblatex}
\usepackage{url}

\usepackage{breakurl}
\newcommand{\dif}{\mathop{}\!\mathrm{d}} 
       \newcommand{\dx}{\dif x}     \newcommand{\dPi}{\dif \Pi}     
 
 \newcommand{\pd}[2]{\frac{\partial #1}{\partial #2}}
   \newcommand{\NN}{\mathbb{N}}  \newcommand{\RR}{\mathbb{R}} \newcommand{\CC}{\mathbb{C}}  \newcommand{\II}{\mathbbm{1}} 
\newcommand{\Aa}{\mathcal{A}}  \newcommand{\Cc}{\mathcal{C}}  \newcommand{\Ee}{\mathcal{E}} \newcommand{\Ff}{\mathcal{F}}  \newcommand{\Hh}{\mathcal{H}}    \newcommand{\Ll}{\mathcal{L}}            \newcommand{\Xx}{\mathcal{X}}  
\DeclarePairedDelimiter{\norm}{\lVert}{\rVert}
\DeclarePairedDelimiter{\abs}{\lvert}{\rvert}
\DeclarePairedDelimiter{\braces}{ \{ }{ \} }
\DeclarePairedDelimiter{\brackets}{(}{)}
\DeclarePairedDelimiter{\sqbrackets}{[}{]}

\DeclarePairedDelimiter{\ip}{\langle}{\rangle}

\DeclarePairedDelimiter{\floor}{\lfloor}{\rfloor}
\DeclareMathOperator{\Var}{Var}
\DeclareMathOperator{\Cov}{Cov}

\DeclareMathOperator{\support}{supp}
\DeclareMathOperator{\Span}{span}
 
\newcommand{\eps}{\varepsilon}
\newcommand{\iidsim}{\overset{iid}{\sim}}
\newcommand{\WW}{\mathbb{W}}

\newcommand{\HH}{\mathbb{H}}
\newcommand{\loc}{\textnormal{loc}}
\DeclareMathOperator{\Grad}{\nabla}
\DeclareMathOperator{\Div}{\Grad \cdot}

\DeclareMathOperator{\Laplacian}{\mathop{{}\Delta}\nolimits}

\DeclareMathOperator{\tr}{tr}

\theoremstyle{definition}  \newtheorem*{definition*}{Definition}  \newtheorem*{definitions*}{Definitions}  \newtheorem{assumption}{Assumption}
\theoremstyle{remark}  \newtheorem*{remark*}{Remark}  \newtheorem*{remarks*}{Remarks} \newtheorem{remark}{Remark}
\theoremstyle{plain} \newtheorem{theorem}{Theorem} \newtheorem*{theorem*}{Theorem} \newtheorem{lemma}[theorem]{Lemma} \newtheorem{corollary}[theorem]{Corollary} \newtheorem*{lemma*}{Lemma} 
\crefname{appsec}{Appendix}{Appendices}
\crefname{assumption}{Assumption}{Assumptions}
\crefname{equation}{}{}
\crefname{enumi}{}{}
\newlist{lemenum}{enumerate}{1} 
\setlist[lemenum]{label=\alph*., ref=\arabic{theorem}\alph*}
\crefalias{lemenumi}{lemma} 
\newlist{thmenum}{enumerate}{1} 
\setlist[thmenum]{label=\Alph*., ref=\arabic{theorem}\Alph*}
\crefalias{thmenumi}{theorem} 

\usepackage{ifthen}
\AtEveryBibitem{ \clearlist{address} \clearfield{date}  \clearfield{isbn} \clearfield{issn} \clearlist{location} \clearfield{month} %\clearfield{series} 
	\clearfield{doi} \clearfield{url}  %  \ifthenelse{\ifentrytype{book}\OR\infentrytype{incollection}}{}{} % %needs ifthen package
	\renewbibmacro{in:}{%
		\ifentrytype{article}{\clearfield{pages}}{\printtext{\bibstring{in}\intitlepunct}}}
	\ifentrytype{misc}{}{\clearfield{eprint}}
}
\title{On statistical Calder\'on problems}
\author{Kweku Abraham and Richard Nickl}
\date{\textsc{University of Cambridge}}
\begin{document}
	\maketitle
	
	\begin{abstract}
		For $D$ a bounded domain in $\mathbb R^d, d \ge 2,$ with smooth boundary $\partial D$, the non-linear inverse problem of recovering the unknown conductivity $\gamma$ determining solutions $u=u_{\gamma, f}$ of the partial differential equation 
		\begin{equation*}
		\begin{split}
		\Div(\gamma \Grad u)&=0 \quad \text{ in }D, \\
		u&=f \quad \text { on } \partial D,
		\end{split}
		\end{equation*}
		from noisy observations $Y$ of the Dirichlet-to-Neumann map \[f \mapsto \Lambda_\gamma(f) =  {\gamma  \pd{u_{\gamma,f}}{\nu}}\Big|_{\partial D},\] with $\partial/\partial \nu$ denoting the outward normal derivative, is considered. The data $Y$ consists of $\Lambda_\gamma$ corrupted by additive Gaussian noise at noise level $\eps>0$, and a statistical algorithm $\hat \gamma(Y)$ is constructed which is shown to recover $\gamma$ in supremum-norm loss at a statistical convergence rate of the order $\log(1/\eps)^{-\delta}$ as $\eps \to 0$. It is further shown that this convergence rate is optimal, up to the precise value of the exponent $\delta>0$, in an information theoretic sense. The estimator $\hat \gamma(Y)$ has a Bayesian interpretation in terms of the posterior mean of a suitable Gaussian process prior and can be computed by MCMC methods.
		
		\smallskip
		
	\end{abstract}
	\tableofcontents
	
	\section{Introduction}\label{sec:TheCalderonProblem}
	Let $D \subset \mathbb R^d, d \ge 2,$ be a bounded domain, which we understand here to be a connected open set with smooth boundary $\partial D$. For $\gamma: D \to (0,\infty)$ a \textit{conductivity coefficient}, consider solutions $u$ to the \emph{Dirichlet problem} 
	\begin{equation}\label{eqn:DirichletProblem}
	\begin{split}
	\Div(\gamma \Grad u)&=0 \quad \text{ in }D, \\
	u&=f \quad \text { on } \partial D,
	\end{split}
	\end{equation}
	where $\Grad$ denotes the usual gradient operator and where $f: \partial D \to \CC$ prescribes some boundary values. The parameter spaces considered in the sequel are of the form  
	\begin{align} 
	\label{eqn:DefinitionOfTheta}  \Gamma_{m, D'}&=\braces[\Big]{\gamma\in C(D) : \inf_{x\in D} \gamma(x) \geq m,~\gamma=1 \text{ on } D\setminus D'},\\
	\label{eqn:DefinitionOfThetaMalpha}
	\Gamma^\alpha_{m,D'}(M)&=\braces{\gamma \in \Gamma_{m, D'}: \norm{\gamma}_{H^\alpha (D)}\leq M}, \quad M>0, 
	\end{align}
	where $m\in (0,1)$ is a fixed constant, $D'$ is a domain compactly supported in $D$ (that is, its closure $\bar{D'}$ is contained in $D$), and $\alpha \ge 0$
	measures the regularity of $\gamma$ in the Sobolev scale. The Sobolev spaces $H^\alpha(D), H^\alpha(\partial D)$ of complex-valued functions (and variations thereof) are defined in detail in \cref{sec:LaplaceBeltramiEigenfunctionsAndSobolevSpaces}; the standard $L^2(D), L^2(\partial D)$ Lebesgue spaces arise as the case $\alpha=0$, with inner products $\ip{\cdot, \cdot}_{L^2(D)}, \ip{\cdot, \cdot}_{L^2(\partial D)}$, respectively, and $C(D)$ denotes the space of bounded continuous \textit{real-valued} functions on $D$, equipped with the sup-norm $\norm{\cdot}_{\infty}$. Except where otherwise stated, all integrals are taken with respect to Lebesgue and surface measures on $D$ and $\partial D$ respectively.
	
	\smallskip
	
	The elliptic partial differential equation (PDE) in (\ref{eqn:DirichletProblem}) has, for $\gamma\in\Gamma_{m, D'}$ and $f\in H^{s+1}(\partial D)/\CC,$ $s \in \RR$, a unique weak solution $u_{\gamma, f}$ in  the space
	\[\Hh_s:=\brackets[\big]{H^{\min\braces{1,s+3/2}}(D)\cap H^1_\loc (D)}/\CC;\]  that is (for $\overline v$ denoting the complex conjugate of $v$), the equations 
	\begin{equation}\label{eqn:DirichletProblemWeakFormulation}
	\begin{split}\int_D \gamma \Grad u \cdot \Grad \overline v&=0 \quad \forall v\in H_0^1(D),\\
	u&=f \quad \text{ on } \partial D,
	\end{split}
	\end{equation} hold simultaneously (for $u\in \Hh_s$) if and only if $u=u_{\gamma,f}$, where the boundary values of $u$ are defined in a trace sense. Here and below $/\CC$ means that we identify functions $f, f+c$ which are equal up to a scalar $c \in \CC$. See \cref{lem:DirichletProblemHasUniqueSolution}  in \cref{sec:MappingPropertiesOfLambda_gamma} and its proof for details.
	
	Given a solution $u_{\gamma,f}$ to the Dirichlet problem, one can measure the \emph{Neumann (boundary) data}
	\[{\gamma\pd{u_{\gamma,f}}{\nu}}\Big|_{\partial D}\equiv {\pd{u_{\gamma,f}}{\nu}}\Big|_{\partial D},~~\gamma \in \Gamma_{m, D'},\] where $\pd{}{\nu}$ denotes the outward normal derivative on $\partial D$ (again to be understood in a trace sense). It can be shown (see \cref{lem:LambdaGammaTakesOneDerivative}) 
	that for any $s\in\RR$ and any $f\in H^{s+1}(\partial D)/\CC$, the Neumann data lies in the space
	\begin{equation} \label{eqn:DefintionOfHsDiamond}
	H^s_\diamond(\partial D):= \braces{g \in H^{s}(\partial D) : \ip{g,1}_{L^2(\partial D)}=0}.
	\end{equation}Thus, we may define the so-called \emph{Dirichlet-to-Neumann map}, 
	\begin{equation} \label{forward}
	\begin{split}
	\Lambda_\gamma : \quad  H^{s+1}(\partial D)/\CC &\to H^{s}_\diamond (\partial D), \\
	f &\mapsto {\gamma  \pd{u_{\gamma,f}}{\nu}}\Big|_{\partial D},
	\end{split}
	\end{equation}
	which associates to each prescribed boundary value $f$ the Neumann data of the solution of the PDE \cref{eqn:DirichletProblem}. The choice to quotient the domain of $\Lambda_\gamma$ by $\CC$ is natural as the Neumann data is invariant with respect to addition of scalars.

	The \emph{Calder{\'o}n problem} \cite{Calderon1980} is a well studied inverse problem that addresses the task of recovering interior conductivities $\gamma$ from knowledge of the boundary data $\Lambda_\gamma$. Note that while $\Lambda_\gamma$ itself is a linear operator between Hilbert spaces, the `forward map' $\gamma \to \Lambda_\gamma$ is non-linear. Landmark injectivity results by Sylvester and Uhlmann ($d \ge 3$) and by Nachman ($d=2$) show, however, that recovery is in principle possible. 
	
	\begin{theorem*}[Sylvester \& Uhlmann, \cite{Sylvester1987}; Nachman \cite{N96}] 
		If $\Lambda_{\gamma_1} = \Lambda_{\gamma_2}$ then $\gamma_1=\gamma_2$.
	\end{theorem*}
	
	Nachman \cite{Nachman1988} and Novikov \cite{N88} studied elaborate inversion algorithms that allow recovery of $\gamma$ if exact knowledge of the entire operator $\Lambda_\gamma$ is available. Moreover Alessandrini \cite{Alessandrini1988} (and later Novikov and Santacesaria \cite{NS10} for $d=2$) gave `stability estimates' providing quantitative continuity bounds for the inverse map, and Mandache \cite{Mandache2001} gave an `instability estimate' showing that these bounds are nearly sharp. 
		
	The Calder\'on problem has since been vigorously studied and an excellent survey can be found in Uhlmann \cite{Uhlmann2009} and also in the lecture notes by Salo \cite{S08}. Its importance partly stems from its applications to electrical impedance tomography (EIT) -- described in more detail in the next section -- where \emph{discrete} boundary measurements of the operator $\Lambda_\gamma$ are performed to infer the interior conductivity $\gamma$. Any such data comes with error, and arguably the most natural mathematical description of such approximate measurements is by a \textit{statistical} noise model. As the superposition of many independent errors is well described by a normal distribution (via the central limit theorem), it is further natural to postulate that this noise follows a Gaussian law. In algorithmic practice this has already been widely acknowledged in the general setting of inverse problems, where statistical, and in particular Bayesian, inversion approaches have flourished in the last decade since the influential work of Stuart \cite{Stuart2010}. In the context of EIT we refer to the articles \cite{KKSV00, BM12, Kolehmainen2013, Roininen2014, Gehre2014, Dunlop2016} and the many references therein. Currently, little theory giving statistical guarantees for the performance of such Bayesian de-noising methodology is available, particularly for non-linear problems. Some recent progress has been made in non-linear settings (see \cite{V13, NS17, Nickl2017b, NS17b, Nickl2018, Monard2019, GN19}) but no results are available at present for the Calder\'on problem described above, and the purpose of the present paper is to at least partially fill this gap.  
	
	We will introduce three natural noise models for such statistical Cald\'eron problems, all asymptotically closely related, in the next section. We prove our main theorems initially in one of these models, and show in \cref{sec:lecamdef} that the results are in fact valid in the other models too. The preferred model for the theoretical development is \cref{eqn:Model:HilbertSpaceMeasurements}, wherein one observes $\Lambda_\gamma$ corrupted by a Gaussian white noise in an appropriate space of Hilbert--Schmidt operators. The noise is described by the scalar quantity $\eps>0$ governing its magnitude and a parameter $r\in\RR$ determining its `spectral heteroscedasticity'. If we denote by $P^\gamma_\eps=P^\gamma_{\eps,r}$ the resulting probability law of the noisy observations $Y$ of $\Lambda_\gamma$, then our main results can be summarised in the following two theorems.

	\begin{theorem}\label{thm:ExistenceOfEstimator}
		Let $\alpha > 3+d$ be an integer, let $m_0 \in (0,1)$, $M > 1$ be given, and let $D_0$ be a domain in $\mathbb R^d$ such that $\bar {D}_0$ is contained in $D$. 
		
		There exists a measurable function $\hat{\gamma}=\hat \gamma_\eps(Y)$ of the observations $Y \sim P^\gamma_\eps$ such that 
		\[\sup_{\gamma\in \Gamma^\alpha_{ m_0, D_0}(M)} P^\gamma_\eps(\norm{\hat{\gamma}-\gamma}_{\infty}> C\log(1/\eps)^{-\delta})\to 0,~\text{as } \eps\to 0,\]
		where $\delta>0$ depends only on $d$ and $\alpha$, and $C$ % is uniform in the parameter space, i.e.\ it 
		depends only on $\alpha,$ $M$, $m_0$, $D$, $D_0$ and $r$. 
	\end{theorem}
	
	The estimator $\hat \gamma$ in the previous theorem has a natural Bayesian interpretation in terms of the posterior mean of a suitable Gaussian process based prior for $\gamma$. Such priors are most effective for recovery of sufficiently smooth $\gamma$, and the precise bound on $\alpha$ is chosen here for convenience (see \cref{rem:RougherPriors} for further discussion). The derivation and implementation of $\hat \gamma$ are described in \cref{sec:BayesianInterpretation}, where we give the more concrete \cref{thm:ContractionForTheta}, which implies \cref{thm:ExistenceOfEstimator}. We note that $\hat \gamma$ can be calculated without knowledge of the bound $M$ for $\norm{\gamma}_{H^\alpha(D)}$.
	
	The slow (logarithmic) convergence rate is not surprising in view of the folklore that the Calder\'on problem is a severely ill-posed inverse problem (cf.~also \cite{Mandache2001}). The following result makes this folklore information-theoretically precise -- it shows that the convergence rate obtained by the estimator $\hat{\gamma}$ is optimal in the statistical minimax sense, at least up to the precise value of the exponent $\delta$, for the prototypical case where $D_0, D$ are nested balls in $\mathbb R^d$. We denote by $\norm{\cdot}$ the standard Euclidean norm on $\mathbb R^d$. 
	
	\begin{theorem}
		\label{thm:MinimaxLowerBound}
		Let $D_0=\braces{x\in\RR^d : \norm{x}<1/2} \subset D=\braces{x\in\RR^d : \norm{x}<1}$, let $\alpha$ be an integer greater than $2$, and let $m_0 \in(0,1)$ be arbitrary. 
		
		For any $\delta'>\alpha(2d-1)/d$ and all $M$ large enough there exists $c=c(\delta', \alpha, d,m_0,r,M)$ such that
		\[ \inf_{\tilde{\gamma}} \sup_{\gamma \in \Gamma^\alpha_{ m_0, D_0}(M)} P^\gamma_\eps(\norm{\tilde{\gamma}-\gamma}_{\infty}> c\log(1/\eps)^{-\delta'})>1/4\] for all $\eps$ small enough, where the infimum extends over all measurable functions $\tilde{\gamma}=\tilde \gamma(Y)$ of the data $Y \sim P^\gamma_\eps$. 
		
	\end{theorem}
	
	The particular value $1/4$ in the lower bound is chosen for convenience. Determining the exact exponent $\delta$ in the minimax convergence rate is a delicate PDE question related directly to the stability estimates of \cite{Alessandrini1988, NS10}, and beyond the scope of the present paper. For $d \ge 3$, one could use an explicit bound of Novikov \cite{N11} to show that a choice of $\delta$ satisfying $\delta/\alpha\to 1$ as $\alpha\to\infty$ is permitted in \cref{thm:ExistenceOfEstimator}. This scales proportionally (in the regularity index $\alpha$) to the exponent $\delta' = \alpha(2d-1)/d$ from the lower bound \cref{thm:MinimaxLowerBound}.

	\medskip
	
	This paper is structured as follows. In  \cref{sec:NoiseModel} we introduce the measurement model we consider in our theorems, and discuss its relationship to physical measurement models arising in medical imaging practice. In  \cref{sec:BayesianInterpretation} we give the construction of the Bayesian algorithm $\hat \gamma$ that solves our noisy version of the Calder\'on problem. \Cref{sec:FinalRemarks} contains some concluding remarks concerning the main theorems. All proofs and related background material are relegated to later sections. For convenience, the notation used is informally gathered in \cref{sec:Notation}.

	\section{Noise model and electrical impedance tomography}\label{sec:NoiseModel}
	We now introduce three scenarios for noisy observations of the operator $\Lambda_\gamma$ from (\ref{forward}), and discuss their relationship at the end of this subsection. Define \[\tilde{\Lambda}_\gamma= \Lambda_\gamma-\Lambda_1\] where the fixed (deterministic and known) operator $\Lambda_1$ is the Dirichlet-to-Neumann map for the standard Laplace equation, that is, eq.~\cref{eqn:DirichletProblem} with $\gamma=1$ identically on $D$. We then equivalently consider measuring a noisy version of $\tilde{\Lambda}_\gamma$. 
		
	Real-world data involving the Calder{\'o}n problem arises for example in medical imaging, namely in \emph{electrical impedance tomography}, see \cite{KKSV00, BM12, Kolehmainen2013, Roininen2014, Gehre2014, Dunlop2016}  and references therein. Electrodes are attached to a patient (or some other physical medium), and are used both to apply voltages and to record the resulting currents. If we assume the applied voltages are uniform across the surface of any given electrode, and the electrodes measure the average current across their surface, we are led to the observation model 	
	\begin{equation} \label{eqn:Model:HaarMeasurements} Y_{p,q}= \ip{\tilde{\Lambda}_\gamma [\psi_p],\psi_q}_{L^2(\partial D)} +\eps g_{p,q}, \quad p,q\leq P, \quad g_{p,q}\iidsim N(0,1),~ \eps>0,\end{equation} where the $\psi_p$ are, up to scaling factors, indicator functions $\II_{I_p}$ of some disjoint measurable subsets $(I_p)_{p\leq P}$ of $\partial D$ representing the locations of the electrodes. Throughout $N(0,1)$ denotes the standard normal distribution. In principle the noise level $\eps>0$ could vary with $p$ and $q$, but choosing scaling factors $c_p$ so that the $\psi_p=c_p \II_{I_p}$ are $L^2(\partial D)$-orthonormal we expect to be able to realise the above homoscedastic noise model. Also note that while the inner product $\ip{\cdot,\cdot}_{L^2(\partial D)}$ is defined with respect to complex scalars, $\tilde{\Lambda}_\gamma[\psi_p]$ takes real values since $\psi_p$ does (see before \cref{lem:tildeLambdaIsHilbertSchmidt}), and hence  it is natural to consider real-valued noise $g_{p,q}$ and data $Y_{p,q}$.
	
	\smallskip
	
	An alternative noise model considers spectral measurements. Denote by $(\phi_k=\phi_k^{(0)}:k\in\NN\cup \braces{0})$ an orthonormal basis of $L^2(\partial D)$ consisting of real-valued eigenfunctions of the Laplace--Beltrami operator on the compact manifold $\partial D$, described in more detail in \cref{sec:LaplaceBeltramiEigenfunctionsAndSobolevSpaces}. [If $D$ is a disc in $\mathbb R^2$ these comprise the usual trigonometric basis, while for $D$ a ball in $\mathbb R^3$, they are the spherical harmonics.] By discarding the constant function $\phi_0$ we obtain a basis of the spaces $L^2(\partial D)/\CC$ and $L^2_\diamond(\partial D)=H^0_\diamond(\partial D)$. Moreover, appropriate rescaling of these basis functions also provides orthonormal bases $(\phi_k^{(r)}:{k\in\NN})$ of all $H^r(\partial D)/\CC$ and $H^r_\diamond(\partial D)$ spaces, $r\in\RR$. % -- see \cref{sec:LaplaceBeltramiEigenfunctionsAndSobolevSpaces} for details. 
	For some $r\in\RR$, we then consider the noisy matrix measurement model 
	\begin{equation} \label{eqn:Model:LaplaceBeltramiMeasurements} Y_{j,k}= \ip{\tilde{\Lambda}_\gamma [\phi_j^{(r)}],\phi_k^{(0)}}_{L^2(\partial D)} +\eps g_{j,k}, \quad j\leq J,k\leq K, \quad g_{j,k}\iidsim N(0,1), ~\eps>0,\end{equation}
where again it is natural to consider real-valued noise $g_{j,k}$ only. The parameter $r$ can in principle be chosen by the experimenter and reflects how the signal-to-noise ratio varies with frequency: as $r$ increases, the signal at high frequencies (i.e.\ at larger values of $j$) decreases compared to the signal at low frequencies. Likely the most realistic choices are $r=0$ (which will allow for comparison of the models \cref{eqn:Model:HaarMeasurements,eqn:Model:LaplaceBeltramiMeasurements}), and $r=1$, in which case the signal-to-noise ratio is the same across all frequencies: since $\Lambda_\gamma$ maps $H^{1}(\partial D)/\CC$ to $L^2(\partial D)$ isomorphically (\cref{lem:LambdaGammaTakesOneDerivative}), the signal magnitude $\norm{\Lambda_\gamma [\phi_j^{(1)}]}_{L^2(\partial D)}$ is of order $1$ for all $j$. A similar reasoning ($\norm{\phi_k^{(0)}}_{L^2(\partial D)}=1$ for all $k$) underpins the choice of the $L^2(\partial D)$-inner product in \cref{eqn:Model:LaplaceBeltramiMeasurements}.

If we formally take the limit $J,K\to \infty$ in \cref{eqn:Model:LaplaceBeltramiMeasurements}, we obtain a model of Gaussian white noise on a space of Hilbert--Schmidt operators as follows. For $j,k\in\NN$, let  $b_{jk}^{(r)}: H^r(\partial D) \to L^2(\partial D)$ denote the tensor product operator
	\begin{equation}\label{eqn:bklDefinition}
	 b_{jk}^{(r)}(f)= \phi_j^{(r)}\otimes \phi_k^{(0)} (f) := \ip{f,\phi_j^{(r)}}_{H^r(\partial D)} \phi_k^{(0)}, \quad f\in H^r(\partial D), \end{equation}
	 and define the space of linear operators	\begin{equation}\label{eqn:HHrDefinition}
	 \HH_r :=\Big\{T:H^r(\partial D) \to L^2(\partial D),~ T=\sum_{j,k=1}^\infty t_{jk} b_{jk}^{(r)} : t_{jk}\in \RR, \sum_{j,k=1}^\infty t_{jk}^2<\infty \Big\} %(t_{jk})\in \ell_2(\RR)}
	 	.\end{equation}  The elements of $\HH_r$ are the `Hilbert--Schmidt' operators between (the real-valued subsets of) the Hilbert spaces $H^r(\partial D)$ and $L^2(\partial D)$, see \cite{Aubin2011} Chapter 12. Moreover, $\HH_r$ is itself a (real) Hilbert space for the inner product
	 \[\ip{S,T}_{\HH_r}=\sum_{j,k=1}^\infty s_{jk}t_{jk}\equiv \sum_{j,k=1}^\infty \ip{S\phi_j^{(r)},\phi_k^{(0)}}_{L^2(\partial D)} \ip{T\phi_j^{(r)},\phi_k^{(0)}}_{L^2(\partial D)} .\] 
	We then consider observing a realisation of the Gaussian process 
	\begin{equation}\label{eqn:ContinuousModelDistributionalSense}
	 \begin{split}
	 \brackets[\big]{Y(T)&=\ip{\tilde{\Lambda}_\gamma,T}_{\HH_r} + \eps \WW(T) 	:T\in \HH_r},~\eps>0,  \text{ where }\\ \WW(T) &\equiv \ip{\WW,T}_{\HH_r}:=\sum_{j,k=1}^\infty g_{jk} \ip{T\phi_j^{(r)},\phi_k^{(0)}}_{L^2(\partial D)},% = \sum_{j,k=1}^\infty g_{jk}t_{jk}, 
	 \text{ for }g_{jk}\iidsim N(0,1).
	 \end{split}
	 \end{equation}
This makes sense rigorously only if $\tilde{\Lambda}_\gamma \in \HH_r$, and it is proved  in \cref{sec:MappingPropertiesOfLambda_gamma}, \cref{lem:tildeLambdaIsHilbertSchmidt}, that this is indeed the case for any $\gamma \in \Gamma_{m, D'}$ and any $r\in\RR$.  

The process $\WW$ so defined is a Gaussian white noise (\textit{isonormal process}; see, e.g., p.19 in \cite{Gine2016}) indexed by the (real) Hilbert space $\HH_r$. %\magenta{justify briefly?}  
A closed form description of the data in \cref{eqn:ContinuousModelDistributionalSense} is therefore
	\begin{equation} \label{eqn:Model:HilbertSpaceMeasurements}
	Y = \tilde{\Lambda}_\gamma +\eps \WW, \quad \eps>0.
	\end{equation}
We write $P^\gamma_{\eps,r}$ for the law of $Y$ in this last model. We often suppress the parameter $r$, and write $P^\gamma_\eps$ for the probability law and $E^\gamma_\eps$ for the corresponding expectation operator.

\medskip

The continuous model \cref{eqn:Model:HilbertSpaceMeasurements} is more convenient for the application of PDE techniques and facilitates a clearer exposition in the proofs to follow. We prove our main results \cref{thm:ExistenceOfEstimator,thm:MinimaxLowerBound} in that model initially. We will show that the models \cref{eqn:Model:HilbertSpaceMeasurements} and \cref{eqn:Model:HaarMeasurements} are asymptotically closely related to each other in a rigorous `Le Cam' sense, and that as a consequence, \cref{thm:ExistenceOfEstimator,thm:MinimaxLowerBound} also hold in the `electrode model' \cref{eqn:Model:HaarMeasurements}: see \cref{thm:MinimaxityInAllNoiseModels} for a precise statement. The intuitive idea can be summarised as follows: Model \cref{eqn:Model:HilbertSpaceMeasurements} with $r=0$ contains all the information available in model \cref{eqn:Model:HaarMeasurements} simply by evaluating $Y(T)$ with $T=\psi_p \otimes \psi_q$ for each $p,q\le P$ in \cref{eqn:ContinuousModelDistributionalSense}; conversely, as $P \to \infty$, one can approximate Laplace--Beltrami eigenfunctions via linear combinations of indicator functions (under appropriate conditions on the sets $I_p,p\leq P$), and in doing so, given data from model \cref{eqn:Model:HaarMeasurements} we approximately recover data from model \cref{eqn:Model:LaplaceBeltramiMeasurements} and ultimately then also from \cref{eqn:Model:HilbertSpaceMeasurements}. We refer to \cref{sec:EquivalenceOfNoiseModels} for rigorous details, where also the (simpler) equivalence of the models  \cref{eqn:Model:LaplaceBeltramiMeasurements} and  \cref{eqn:Model:HilbertSpaceMeasurements} is established. 

	\section{The Bayesian approach to the noisy Calder\'on problem}\label{sec:BayesianInterpretation}
	We now construct the estimator $\hat{\gamma}$ featuring in \cref{thm:ExistenceOfEstimator}. Following the Bayesian approach to inverse problems advocated by A. Stuart \cite{Stuart2010}, we will construct $\hat \gamma$ in terms of the posterior mean arising from a certain Gaussian process prior. In the context of the EIT inverse problem a Bayesian approach was proposed already in \cite{KKSV00}, and conceptually related work appears in fact much earlier in Diaconis \cite{Diaconis1988} who further traces some of the key ideas back to H. Poincar\'e -- see Chapter XV, \textsection216, in \cite{P1896} for what is possibly the first proposal of an infinite-dimensional Gaussian series prior in a numerical analysis context.

	\smallskip
	
	To this end we need to first establish the existence of a posterior distribution in our measurement setting.  In the Gaussian white noise model \cref{eqn:Model:HilbertSpaceMeasurements}, the log-likelihood function can be derived from the Cameron--Martin theorem in a suitable Hilbert space: precisely, the law $P^\gamma_\eps$ of $Y$  is dominated by the law $P^1_\eps$ of $\eps\WW$, with log-likelihood function
	\begin{equation}\label{eqn:LogLikelihood} \ell(\gamma)\equiv \log p^\gamma_\eps(Y) := \log \frac{dP^\gamma_\eps}{dP^1_\eps}(Y)  = \frac{1}{\eps^2} \ip{Y,\tilde{\Lambda}_\gamma}_{\HH_r}-\frac{1}{2\eps^2}\norm{\tilde{\Lambda}_\gamma}_{\HH_r}^2,~~\gamma \in \Gamma_{m, D'}.
	\end{equation} 
	See Section 7.4 in \cite{Nickl2017b} for a detailed derivation, which requires Borel-measurability (ensured by \cref{lem:ForwardStability} below) of the map $\gamma \mapsto \tilde{\Lambda}_\gamma$ from the (Polish) space $\Gamma_{m,D'}$ equipped with the $\norm{\cdot}_\infty$-topology into the Hilbert space $\HH_r$.
	
	Then for any prior (Borel) probability measure $\Pi$ on  $\Gamma_{m,D'}$, the posterior distribution given observations $Y$ is given by
	\begin{equation}\label{eqn:PosteriorDefinition}
	\Pi(B\mid Y) = \frac{\int_B p^\gamma_\eps (Y)\dPi(\gamma)}{\int_{\Gamma_{m, D'}} p^\gamma_\eps (Y) \dPi(\gamma )}, \quad B \subset \Gamma_{m, D'} \text{ Borel-measurable},
	\end{equation}
	see again Section 7.4 in \cite{Nickl2017b} (and also \cite{Ghosal2017}, eq (1.1)). We denote by $E^\Pi[\cdot]$ the expectation operator according to the prior, and by $E^\Pi[\mathrel{\;\cdot\;} \mid Y]$ the expectation according to the posterior. 
	
	\smallskip 
	
	What precedes is stated for a prior defined on the conductivities $\gamma$. The priors introduced below will be of the form $\gamma=\Phi \circ \theta$ for a suitable link function $\Phi$, where $\theta$ takes values in a \emph{linear} space, so that a Gaussian process prior can be assigned to $\theta$. Composition with the map $\Phi$ described in \cref{sec:PriorConstruction} gives rise to a measurable bijection between the parameter spaces for $\theta$ and for $\gamma$, and the comments above therefore equivalently yield the existence of the posterior distribution for $\theta$.

	\subsection{Prior construction}\label{sec:PriorConstruction}
	
	We define the prior on $\theta$ in terms of a base prior $\Pi'$. For the base prior we assume the following -- we refer, e.g., to \cite[Sections 2.1 and 2.6]{Gine2016} for the basic definitions of Gaussian measures and processes and their reproducing kernel Hilbert spaces (RKHS). 
	\begin{assumption}\label{assumption:pco}
		Let $\Pi'$ be a centred Gaussian Borel probability measure on the Banach space $C_u(D)$ of uniformly continuous real-valued functions on $D$, and let $\alpha, \beta$ be integers satisfying $\alpha>\beta>2+d/2$. Assume $\Pi'(H^\beta(D))=1$ and that the RKHS $(\mathcal H, \norm{\cdot}_{\mathcal H})$ of $\Pi'$ is continuously embedded into the Sobolev space $H^\alpha(D)$. 
	\end{assumption}
	Natural candidates for such priors are restrictions to $D$ of Gaussian processes whose covariances are given by Whittle--Mat\'ern kernels, see \cite{Ghosal2017}, p.313 and p.575 -- in these cases one can satisfy the assumption for any $2+d/2<\beta<\alpha-d/2$ by taking $\mathcal H$ to \textit{coincide} with the Sobolev space $H^\alpha(D)$. The restriction to integer-valued $\alpha, \beta$ is convenient to simplify some proofs.
	
	\smallskip
	
	In the proofs that follow we will require that the true $\gamma_0$ is in the `interior' of the support of the induced prior on $\gamma$, so recalling that \cref{thm:ExistenceOfEstimator} is stated uniformly over $\Gamma^\alpha_{m_0,D_0}(M)$, we choose $0<m_1<m_0<1$ and a domain $D_1$ such that $\bar {D}_0\subset D_1, \bar {D}_1\subset D$. Then let $\zeta: D\to [0,1]$ be a smooth cutoff function, identically one on $D_0$ and compactly supported in $D_1$. For a %random function $\theta' \sim \Pi'$ and 
	link function $\Phi :\RR\to (m_1,\infty)$ to be specified, we define the prior $\Pi=\Pi_\eps$ as the (Borel) law in $C_u(D)$ of the random function
	\begin{equation} \label{eqn:tampering}
	\gamma = \Phi \circ \theta,~~ \theta(x) = \theta_\eps(x) =  \eps^{d/(\alpha+d)}\zeta(x) \theta'(x),~x \in D, ~ \theta' \sim \Pi',
	\end{equation}
	where, in a slight abuse of notation, we use the notation $\Pi$ for the prior laws both of $\theta$ and of the induced conductivity $\gamma=\Phi \circ \theta$. The link function $\Phi$ will be required to be \textit{regular} in the sense of \cite{Nickl2018}, that is to say, $\Phi$ is a smooth bijective function satisfying $\Phi(0)=1$, $\Phi'>0$ on $\RR$, and $\norm{\Phi^{(j)}}_{\infty}<\infty$ for all integers $j \ge 1$, with inverse function denoted by $\Phi^{-1}:(m_1,\infty) \to \RR$. We refer to \cite{Nickl2018} Example 8 where a regular link function is exhibited, and to \cite{Nickl2018} Lemma 29 for basic properties of such functions. 
	In particular we note that there are constants $C=C(\Phi)$, $c=c(\Phi,m_0,m)$, $C'=C'(\Phi,\alpha)$ and $c'=c'(\Phi,\alpha,m_0,m)$ such that for any bounded functions $\theta,\theta_0$, any integer $\alpha\geq d/2$ and any $\gamma_0,\gamma\in \Gamma_{m,D_1},~m\in(m_1,m_0)$, \begin{align}\label{eqn:NormPhiCircThetaEqualsNormTheta} \norm{\Phi \circ \theta-\Phi \circ \theta_0}_{\infty} &\leq C\norm{\theta-\theta_0}_{\infty}, \\
	\label{eqn:NormPhiInverseCircGammaLinfty}
	\norm{\Phi^{-1} \circ \gamma-\Phi^{-1} \circ \gamma_0}_{\infty}&\leq c\norm{\gamma-\gamma_0}_{\infty}
	\\
	\label{eqn:NormPhiCircThetaBound} \norm{\Phi\circ \theta}_{H^\alpha(D)}&\leq C'(1+\norm{\theta}_{H^\alpha(D)}^\alpha),
	\\
	\label{eqn:NormPhiInverseCircGammaBound}\norm{\Phi^{-1}\circ \gamma_0}_{H^\alpha(D)}&\leq c'(1+\norm{\gamma_0}_{H^\alpha(D)}^\alpha).
	\end{align}  
	The first inequality %, with constant $C=\norm{\Phi'}_{L^\infty(\RR)}$, 
	is an immediate consequence of the mean value theorem and the third is given in \cite{Nickl2018} Lemma 29. The second and fourth inequalites follow from the same arguments, applied to the function $\Phi^{-1}$ (this can be seen to be regular on the domain $[m,\infty)$ for $m>m_1$ by considering explicit formulas for its derivatives).

	\subsection{Posterior contraction result}\label{sec:PosteriorContractionResults}
	For the following result we define
	\begin{equation}\label{eqn:XiEpsDelta} \xi_{\eps,\delta}=\log(1/\eps)^{-\delta}, \quad \eps,\delta>0.\end{equation}

	\begin{theorem}\label{thm:ContractionForTheta}
		For some $m_0\in(0,1)$, $D_0$ compactly contained in $D$, and $M>0$, suppose that the true conductivity $\gamma_0$ belongs to the set
		\begin{equation}\label{set}
		\Gamma_{m_0,D_0} \cap \braces{\Phi\circ \theta : \theta\in \Hh, \norm{\theta}_{\Hh}\leq M},
		\end{equation}
		and define $\theta_0=\Phi^{-1}\circ \gamma_0$.
		For $\Pi'$ satisfying \cref{assumption:pco} let $\Pi$ be the prior arising from \cref{eqn:tampering}, and denote by $\Pi(\mathrel{\;\cdot\;}\mid Y)$ the posterior distribution for $\theta$ arising from observations $Y$ in the model \cref{eqn:Model:HilbertSpaceMeasurements}. 
		Then there exist constants $C=C(M,m_0,m_1,D,D_0,D_1,\Phi,\zeta,r,\alpha,\beta)>0$ and $\delta=\delta(d,\beta)>0$ such that
		\begin{equation}\label{eqn:ContractionForTheta} \Pi\brackets[\big]{\norm{\theta-\theta_0}_{\infty}>C\xi_{\eps,\delta} \mid Y} \to^{P_\eps^{\gamma_0}} 0\text{ as }\eps \to 0.\end{equation}  Moreover, if $E^\Pi[\theta \mid Y]$ denotes the (Bochner) mean of $\Pi(\mathrel{\;\cdot\;}\mid Y)$, then for any $K>C$, 
		\begin{equation}\label{eqn:PosteriorMeanConsistent}
		\sup_{\gamma_0}  P^{\gamma_0}_\eps\brackets[\big]{\norm{ E^{\Pi}\sqbrackets{\theta \mid Y}-\theta_0}_{\infty}>K\xi_{\eps,\delta}} \to 0~\text{as } \eps \to 0,
		\end{equation}
		where the supremum extends over all $\gamma_0$ in the set (\ref{set}).
	\end{theorem}
	
	\Cref{thm:ContractionForTheta}, whose proof is given in \cref{sec:Proofs:PosteriorContraction}, immediately implies \cref{thm:ExistenceOfEstimator}: Indeed, given an integer $\alpha>3+d$, let $\Pi$ be a prior from \cref{eqn:tampering} whose base prior $\Pi'$ satisfies \cref{assumption:pco} with RKHS $\Hh=H^{\alpha}(D)$. [For instance, take a Whittle-Mat\'ern prior, noting that a choice of integer $\beta>2+d/2$ is then admissible.] Let $\hat{\theta}= E^\Pi[\theta \mid Y]$ be the associated posterior mean and define $\hat{\gamma}=\Phi\circ\hat{\theta}$. Then \cref{eqn:NormPhiCircThetaEqualsNormTheta,eqn:PosteriorMeanConsistent} will imply \cref{thm:ExistenceOfEstimator}, so it suffices to show that the conditions of \cref{thm:ExistenceOfEstimator} imply those of \cref{thm:ContractionForTheta}, in particular that for any $\gamma_0\in \Gamma^\alpha_{ m_0, D_0}(M) 
	$, there exists an $M'=M'(\alpha, M, m_0, D, D_0)$
	such that $\norm{\theta_0}_{H^\alpha(D)}\leq M'.$ But this is immediate from \cref{eqn:NormPhiInverseCircGammaBound}. 
	
	%We remark that the minimax estimation lower bound \cref{thm:MinimaxLowerBound} also lower bounds the posterior contraction rate, so that the contraction rate obtained in \cref{eqn:ContractionForGamma} is optimal, at least up to the exponent $\delta$. This is because a posterior contraction rate of $\xi$ implies the existence of an estimator achieving rate $\xi$ (for example, the centre of the smallest posterior ball of mass at least 1/2 -- see Theorem 8.7 in \cite{Ghosal2017}). 
	
	\section{Concluding remarks}\label{sec:FinalRemarks}
	\begin{remark} [Computational aspects]
		The posterior mean $E^\Pi\sqbrackets{\theta \mid Y}$ can be calculated via MCMC or expectation-propagation methods (naturally in the discretisations \cref{eqn:Model:HaarMeasurements} or \cref{eqn:Model:LaplaceBeltramiMeasurements} of our continuous model \cref{eqn:Model:HilbertSpaceMeasurements}). This allows one to bypass potential difficulties encountered by optimisation based methods in non-linear settings such as the present one. For instance, the MAP estimates studied in \cite{BM12, Nickl2018} may not provably recover global optima in the EIT setting, since the non-linearity of the map $\theta \mapsto \Lambda_{\Phi(\theta)}\equiv \Lambda_\gamma$ implies that the associated least squares criterion is not necessarily convex. Variational methods such as those proposed in  \cite{HKQ18} for EIT also typically require a convex relaxation to be efficiently computable, see \cite{KV87}.
		
		The pCN algorithm \cite{CRSW13} allows one to sample from posterior distributions in general inverse problems as long as the forward map $\theta \mapsto \Lambda_{\Phi(\theta)}$ can be evaluated, which in our setting has the basic cost of (numerically) solving the standard elliptic PDE \cref{eqn:DirichletProblem}. Even in the absence of log-concavity of the posterior measure one can give sampling guarantees for this algorithm, see \cite{HSV14}, so that the approximate computation of $E^\Pi[\theta \mid Y]$ by the sample average $(1/M) \sum_m \theta_m$ of the pCN Markov chain is provably possible at any given noise level $\eps$. Related work on MCMC-based approaches in the setting of electrical impedance tomography can be found in \cite{KKSV00, Roininen2014, Dunlop2016}, wherein also many further references can be found. Instead of MCMC methods one can also resort to variational Bayes methods -- see for example \cite{Gehre2014}, where computation of the posterior mean is addressed specifically for the EIT problem relevant in the present paper. 
	\end{remark}
	\begin{remark}[Estimation of $\gamma$ at the boundary]
		In \cref{thm:ExistenceOfEstimator} we assume that the true conductivity $\gamma_0$ equals $1$ on the complement of some known set $D_0 \subset D$. In the proofs we work mostly with the \emph{differences} $\Lambda_\gamma-\Lambda_{\gamma_0}$ and the results can be expected to extend to $\gamma_0$ being known (and positive) on $D\setminus D_0$. Estimation of the value of $\gamma_0$ at $\partial D$ is an `easier' (less ill-posed) problem and is addressed, in a statistical setting, in \cite{CG17}.
	\end{remark}
	
	\begin{remark}[Smoothness of $\gamma$] \label{rem:RougherPriors}
		In \cref{thm:ExistenceOfEstimator} we assume $\alpha>3+d$, and hence that the true conductivity is sufficiently smooth in a Sobolev sense. In terms of the proofs, this requirement primarily arises from the stability estimates employed (\cite{Alessandrini1988} requires a minimal Sobolev smoothness of $\gamma$ of degree $\beta>2+d/2$), and further from analytical support properties of Gaussian process priors (for a Whittle--Mat\'ern process with RKHS $H^\alpha(D)$ to be supported in $H^\beta(D)$, one needs $\alpha>\beta+d/2$), necessitating $\alpha>2+d$. To avoid technicalities with non-integer Sobolev spaces, we strengthen our hypothesis to $\alpha>3+d$.
		
		An interesting direction for further research would seek to replace the smoothness assumption on $\gamma_0$ with different structural assumptions. For example, if we consider a class of functions that are piecewise constant on a known finite collection of subsets of $D$, better than logarithmic stability results are available in \cite{AV05}, and faster convergence rates may then be obtained. Note that this would require different methods, e.g., prior measures that can model discontinuous functions, which is beyond the scope of the present paper. 
	\end{remark}

	\section{Proofs}\label{sec:Proofs}

The proof of \cref{thm:ContractionForTheta} follows the template devised in \cite{Monard2019} for a very different inverse problem. We first show that the posterior distribution induced on the operators $\Lambda_\gamma$ contracts around $\Lambda_{\gamma_0}$ (see \cref{thm:ContractionForLambda}), by combining tools from Bayesian nonparametric statistics \cite{vdVvZ08, Ghosal2017} with analytical properties of the `regression' operators $\Lambda_\gamma$ (specifically of their low-rank approximations). Dealing with the unboundedness of Gaussian priors for $\theta$ and with the non-linearity of the composite forward map $\theta \mapsto \Lambda_{\Phi(\theta)}$ poses a main challenge in the proof. As in \cite{Monard2019} this challenge is overcome using the rescaling in \cref{eqn:tampering} of the base prior $\Pi'$. For such priors the posterior distribution concentrates on sufficiently regular conductivities $\gamma$ that the stability estimates of \cite{Alessandrini1988, NS10} -- which we show to hold also for the information-theoretically relevant norms here -- can be applied, resulting in contraction of the posterior distributions of $\gamma$ and $\theta$ about $\gamma_0$ and $\theta_0$ respectively. Finally, to deduce consistency of the posterior mean, we adapt a quantitative uniform integrability argument from \cite{Monard2019} to the present situation. The proof of Theorem \ref{thm:MinimaxLowerBound} follows from the `instability' estimate in \cite{Mandache2001} and common information-theoretic lower bound techniques for statistical estimators as in \cite{T09, Gine2016}.
	
\begin{remark*}
	The model \cref{eqn:Model:HilbertSpaceMeasurements} naturally considers $\Lambda_\gamma$ as an element of the \textit{real} Hilbert space $\HH_r$ consisting of Hilbert-Schmidt operators between the sets of \textit{real-valued} functions of $H^r(\partial D)/\CC$ and $L^2_\diamond(\partial D)$, respectively. Various results in the literature -- such as the stability and instability results of \cite{Alessandrini1988} and \cite{Mandache2001} to be used below -- regard $\Lambda_\gamma$ as an element of the space $\HH_{r,\mathbb C}$ of Hilbert-Schmidt operators between \textit{complex-valued} function spaces $H^r(\partial D)/\CC, L^2_\diamond(\partial D)$. This difference is inconsequential since $\HH_r$ embeds into $\HH_{r, \mathbb C}$ via the unique extension $T_\mathbb C(f+ig)=T(f) + i T(g)$ for $T \in \HH_r$, and the Hilbert--Schmidt norm of this larger space, when restricted to $\HH_r$, coincides with the $\HH_r$ norm. In fact, since we regard $b_{jk}^{(r)} \in \HH_r$ in \cref{eqn:bklDefinition} as a map from $H^r(\partial D)$ to $L^2_\diamond(\partial D)$, strictly speaking we have already embedded $\HH_r$ into $\HH_{r,\CC}$ in this way. %; this also makes the norm equality clear via Parseval's theorem, since the $b_{jk}^{(r)}$ form an orthonormal basis of $\HH_{r, \mathbb C}$.
\end{remark*}
	
%\begin{remark*}
%		Our noise model \cref{eqn:Model:HilbertSpaceMeasurements} naturally considers $\Lambda_\gamma$ as an element of a \textit{real} Hilbert space $\HH_r$ consisting of Hilbert-Schmidt operators between the subspaces of \textit{real-valued} functions of $H^r(\partial D)/\CC$ and $L^2_\diamond(\partial D)$, respectively. Various results in the literature -- such as the stability and instability results of \cite{Alessandrini1988} and \cite{Mandache2001} to be used below -- regard $\Lambda_\gamma$ as a map in the space $\HH_{r,\mathbb C}\equiv \Ll_2(H^r(\partial D)/\CC, L^2_\diamond(\partial D))$ of Hilbert-Schmidt operators between complex-valued function spaces. This difference is inconsequential since we can view $\HH_r$ a subspace of $\HH_{r, \mathbb C}$ via the unique extension $T_\mathbb C(f+ig)=T(f) + i T(g)$ for $T \in \HH_r$, and the Hilbert--Schmidt norm of this larger space, when restricted to $\HH_r$, coincides $\|T_\mathbb C\|_{\HH_{r, \mathbb C}}=\|T\|_{\HH_r}$ (for instance using Parseval's theorem and that the $b_{jk}^{(r)} \in \HH_r$ from \cref{eqn:bklDefinition} form an orthonormal basis of $\HH_{r, \mathbb C}$ as well).
%	\end{remark*}

	\subsection{Low rank approximation of $\tilde {\Lambda}_\gamma$}
	
	A key idea used in various proofs that follow is that we can project the operator $\tilde{\Lambda}_\gamma$ onto a finite-dimensional subspace and incur only a small error. We recall from \cref{eqn:bklDefinition} the orthonormal basis $(b_{jk}^{(r)})_{j,k\in\NN}$ of $\HH_r$ consisting of tensor product operators 
	\begin{equation*}
	b_{jk}^{(r)}(f)= \phi_j^{(r)} \otimes \phi_k^{(0)} (f) := \ip{f,\phi_j^{(r)}}_{H^r(\partial D)} \phi_k^{(0)}, \quad f\in H^r(\partial D)/\CC,~j,k\in\NN,\end{equation*} where the Laplace--Beltrami eigenfunctions $(\phi_j^{(r)})_{j\in\NN}$  were introduced before \cref{eqn:Model:LaplaceBeltramiMeasurements}. An operator $U\in\HH_r$ has coefficients $\ip{U,b_{jk}^{(r)}}_{\HH_r}=\ip{U\phi_j^{(r)},\phi_k^{(0)}}_{L^2(\partial D)}$ with respect to this basis, and we define the projection map $\pi_{JK}$ by 
	\begin{equation}\label{eqn:def:PiJK}
	\pi_{JK} U = \sum_{j\leq J} \sum_{k\leq K} \ip{U\phi_j^{(r)},\phi_k^{(0)}}_{L^2(\partial D)}b_{jk}^{(r)}.
	\end{equation}
	
	\begin{lemma}\label{lem:ProjectionErrorDecaysAsMinJK^-nu}
		For constants $m\in (0,1),M>1$ and some domain $D'$ compactly contained in $D$, let $\gamma\in\Gamma_{m, D'}$ be bounded by $M$ on $D$. For any $\nu>0$ there is a constant $C=C(\nu,D,D',r)>0$ such that \[\norm{\tilde{\Lambda}_\gamma-\pi_{JK} \tilde{\Lambda}_\gamma}_{\HH_r} \leq C\tfrac{M}{m}\min\brackets{J,K}^{-\nu}.\] 
	\end{lemma}
	\begin{proof}
		Apply \cref{lem:ProjectionsApproximateHSOperatorsWell} from \cref{sec:ComparisonResultsForHilbertSchmidtOperators} with $s=0$, $p=r-\nu(d-1)$ and $q=\nu(d-1)$, and note $\norm{\tilde{\Lambda}_\gamma}_{H^{p-(d-1)}(D) \to H^q(D)}\leq C\tfrac{M}{m}$ for such a constant $C$ by \cref{lem:tildeLambdaInfinitelySmoothing}. \qedhere
	\end{proof}

	The proofs of the stability results for the Calder{\'o}n problem in the next section involve the $H^{1/2}(\partial D)/\CC \to H^{-1/2}(\partial D)$ operator norm, which we denote by $\norm{\cdot}_*$. To connect this norm to the information-theoretically relevant $\HH_r$-norm, the following consequence of \cref{lem:ProjectionErrorDecaysAsMinJK^-nu} will be useful.

	\begin{lemma}\label{lem:HHrNormEquivalentTo*Norm}
		For $m\in (0,1)$, $M_0,M_1>1$ and $D'$ a domain compactly contained in $D$, let $\gamma_0,\gamma \in \Gamma_{m,D'}$ be bounded on $D$ by $M_0$ and $M_1$ respectively. Then there are constants $C_1$ and $C_2$ depending only on $r$, $D$ and $D_0$ such that if  $\norm{\Lambda_\gamma-\Lambda_{\gamma_0}}_*\leq 1$ then
		\begin{equation}\label{eqn:HHrNormEquivalentTo*Norm} \norm{\Lambda_\gamma-\Lambda_{\gamma_0}}_{\HH_r} \leq C_1 \brackets[\big]{\tfrac{M_1+M_0}{m}\norm{\Lambda_\gamma-\Lambda_{\gamma_0}}_{*}}^{1/2} ,
		\end{equation}
		and if $\norm{\Lambda_\gamma-\Lambda_{\gamma_0}}_{\HH_r}\leq 1$ then
		\begin{equation}\label{deep}
		\norm{\Lambda_\gamma-\Lambda_{\gamma_0}}_*\leq C_2\brackets[\big]{ \tfrac{M_1+M_0}{m} \norm{\Lambda_\gamma-\Lambda_{\gamma_0}}_{\HH_r}}^{1/2}.\end{equation}
	\end{lemma}
	\begin{proof}
		For $J>0$ and $\nu>0$ to be chosen, by \cref{lem:ProjectionErrorDecaysAsMinJK^-nu} we have 
		\[ \norm{\tilde{\Lambda}_\gamma-\pi_{JJ} \tilde{\Lambda}_\gamma}_{\HH_r} \leq C\tfrac{M_1}{m}J^{-\nu},\] for a constant $C=C(\nu,D,D',r)$, and a corresponding bound %with $M_0$ in place of $M_1$ 
		holds for $\norm{\tilde{\Lambda}_{\gamma_0}-\pi_{JJ}\tilde{\Lambda}_{\gamma_0}}_{\HH_r}$.
		An application of \cref{lem:EquivalenceOfProjectedNorms} with $s=0$, $p=d-1/2,$ and $q=-1/2$, also yields (with $x_+=\max(x,0)$, and $\Ll_2(H^{d-1/2},H^{-1/2})$ a space of Hilbert--Schmidt operators as defined in \cref{sec:ComparisonResultsForHilbertSchmidtOperators})
		\begin{align*} \norm{\pi_{JJ} \tilde{\Lambda}_\gamma - \pi_{JJ} \tilde{\Lambda}_{\gamma_0}}_{\HH_r} & \leq C' (1+J^{1/(d-1)})^{1/2+(d-1/2-r)_+}\norm{\pi_{JJ} (\tilde{\Lambda}_\gamma - \tilde{\Lambda}_{\gamma_0})}_{\Ll_2(H^{d-1/2},H^{-1/2})} \\ & \leq C' (1+J^{1/(d-1)})^{(d+\abs{r})}\norm{\tilde{\Lambda}_\gamma - \tilde{\Lambda}_{\gamma_0}}_{\Ll_2(H^{d-1/2},H^{-1/2})} \\
		&\leq c'(1+J^{1/(d-1)})^{(d+\abs{r})} \norm{\tilde{\Lambda}_\gamma - \tilde{\Lambda}_{\gamma_0}}_*,\end{align*} for constants $C', c'$ depending on $D,r$,
		where we use \cref{lem:ProjectionsApproximateHSOperatorsWell} to obtain the final inequality. Since $\Lambda_\gamma-\Lambda_{\gamma_0}=\tilde{\Lambda}_\gamma - \tilde{\Lambda}_{\gamma_0}$, we deduce, for a constant $C''$ that
		\begin{equation*} 
		\begin{split} \norm{\Lambda_\gamma-\Lambda_{\gamma_0}}_{\HH_r} &\leq \norm{\tilde{\Lambda}_\gamma-\pi_{JJ} \tilde{\Lambda}_\gamma}_{\HH_r}+\norm{\tilde{\Lambda}_{\gamma_0}-\pi_{JJ}\tilde{\Lambda}_{\gamma_0}}_{\HH_r}+\norm{\pi_{JJ}\tilde{\Lambda}_\gamma-\pi_{JJ} \tilde{\Lambda}_{\gamma_0}}_{\HH_r} \\ 
		& \leq C''\brackets[\big]{\brackets[\big]{\tfrac{M_1+M_0}{m}}J^{-\nu} + J^{(d+\abs{r})/(d-1)} \norm{\tilde{\Lambda}_\gamma - \tilde{\Lambda}_{\gamma_0}}_* }. \end{split}
		\end{equation*}
		Since $\norm{\tilde{\Lambda}_\gamma - \tilde{\Lambda}_{\gamma_0}}_*\leq1$, we can choose an integer $J$ to balance the two terms up to a constant (take $J=\floor{\big(\frac{m}{M_0+M_1}\norm{\tilde{\Lambda}_\gamma-\tilde{\Lambda}_{\gamma_0}}_*\big)^{-(d-1)/(\nu(d-1)+d+\abs{r})}}$). This yields, for a constant $c''$,
		\[\norm{\Lambda_\gamma-\Lambda_{\gamma_0}}_{\HH_r} \leq c''\brackets[\big]{\tfrac{M_1+M_0}{m}}\brackets[\Big]{\brackets[\big]{\tfrac{M_1+M_0}{m}}^{-1} \norm{\Lambda_\gamma-\Lambda_{\gamma_0}}_{*}}^{\nu(d-1)/(\nu(d-1)+d+\abs{r})}.\] 
		Choosing $\nu= (d+\abs{r})/(d-1)$ yields \cref{eqn:HHrNormEquivalentTo*Norm}.
		
		For (\ref{deep}), given that $\norm{\Lambda_\gamma-\Lambda_{\gamma_0}}_*\leq \norm{\Lambda_\gamma-\Lambda_{\gamma_0}}_{\Ll_2(H^{1/2}(\partial D)/\CC,H^{-1/2}(\partial D))}$ (using the general fact that the Hilbert--Schmidt norm upper bounds the operator norm of a linear operator between separable Hilbert spaces), and the observation that the proof of \cref{lem:ProjectionErrorDecaysAsMinJK^-nu} can be adapted to apply with the $\Ll_2(H^{1/2}(\partial D)/\CC,H^{-1/2}(\partial D))$ norm in place of the $\HH_r$ norm, an almost identical argument to the above yields
		\begin{equation*}\norm{\Lambda_\gamma-\Lambda_{\gamma_0}}_{*} \leq C\brackets[\big]{ \brackets[\big]{\tfrac{M_1+M_0}{m}} J^{-\nu} + J^{(r-1/2)_+/(d-1)} \norm{\Lambda_\gamma-\Lambda_{\gamma_0}}_{\HH_r}}.\end{equation*} 
		Choosing  $J$ to balance the terms yields \[\norm{\Lambda_\gamma-\Lambda_{\gamma_0}}_{*} \leq C'\brackets[\big]{\tfrac{M_1+M_0}{m}}\brackets[\Big]{\brackets[\big]{\tfrac{M_1+M_0}{m}}^{-1} \norm{\Lambda_\gamma-\Lambda_{\gamma_0}}_{\HH_r}}^{\nu(d-1)/(\nu(d-1)+(r-1/2)_+)},\] and the result follows from noting that the exponent is at least 1/2 for $\nu>(r-1/2)_+/(d-1)$. \qedhere
	\end{proof}

	\subsection{Forward and inverse continuity results}\label{sec:Proofs:Stability}
	We now prove the following estimates for the maps $\gamma\mapsto \Lambda_\gamma, \Lambda_\gamma \to \gamma$. 
	
	\begin{lemma} \label{lem:ForwardStability}
		For $m\in(0,1),$ $M_0,M_1>1$ and $D'$ a domain compactly contained in $D$, let $\gamma,\gamma_0\in \Gamma_{m,D'}$ be bounded on $D$ by $M_1$ and $M_0$ respectively. Then there exist constants $C=C(r,D,D')$, $\tau=\tau(D)$ such that
		\[\norm{\Lambda_\gamma-\Lambda_{\gamma_0}}_{\HH_r}\leq C\tfrac{M_0M_1}{m^2}\norm{\gamma-\gamma_0}_{\infty}^{1/2},\] whenever $\norm{\gamma-\gamma_0}_{\infty}\leq \tau\frac{m^2}{M_0M_1}$. 
	\end{lemma}
	
	\begin{lemma}\label{lem:InverseStability}
		For some $\beta>2+d/2$, some $m\in(0,1),$ $M>0$ and some domain $D'$ compactly contained in $D$, suppose $\gamma,\gamma_0\in \Gamma^\beta_{m,D'}(M)=\braces{\gamma \in \Gamma_{m, D'}: \norm{\gamma}_{H^\beta (D)}\leq M}$.
		Then there exist constants $C$ and $\tau$ depending only on $M$, $D$, $D'$, $m$, $\beta$ and $r$ such that, for a constant $\delta=\delta(d,\beta)>0$,  \[\norm{\gamma-\gamma_0}_{\infty}\leq C \abs{\log \norm{\Lambda_\gamma -\Lambda_{\gamma_0}}_{\HH_r}}^{-\delta},\] whenever $\norm{\Lambda_\gamma -\Lambda_{\gamma_0}}_{\HH_r}\leq \tau$. 
	\end{lemma}
	The explicit form of the dependence of the constant in \cref{lem:ForwardStability} on $M_1$ and $M_0$ will be convenient in the proof of \cref{lem:LinftySmallBalls}.

	\begin{proof}[Proof of \cref{lem:ForwardStability}]
		We initially show, for some $C=C(D)$, that
		\begin{equation}\label{eqn:OperatorNormTildeLambdaGamma}
		\norm{\Lambda_\gamma-\Lambda_{\gamma_0}}_*\leq CM_0\tfrac{m+ M_1}{m^2} \norm{\gamma-\gamma_0}_{\infty}.
		\end{equation}
		The result then follows from \cref{lem:HHrNormEquivalentTo*Norm}, noting that $(M_0+M_1)M_0(m+M_1)/m^3\leq 4M_0^2M_1^2/m^4$.
		
		For $\gamma,\gamma_0$ as given, let $f\in H^{1/2}(\partial D)/\CC$, and recall we write $u_{\gamma,f}$ for the unique solution in $\Hh_{-1/2} \equiv H^1(D)/\CC$ to the Dirichlet problem on $D$ with conductivity $\gamma$ and boundary data $f$, whose existence is guaranteed by \cref{lem:DirichletProblemHasUniqueSolution}, and similarly for $u_{\gamma_0,f}$. The equivalence class of functions $u_{\gamma,f}-u_{\gamma_0,f}$ has a representative $w\in H_0^1(D)$, which is easily seen to (weakly) solve the PDE
		\begin{equation}\label{eqn:DifferenceOfSolutionsPDE}
		\begin{aligned}
		\Div (\gamma \Grad w) %&= -\Div (\gamma \Grad u_{\gamma_0}) 
		&= \Div \brackets[\big]{(\gamma_0-\gamma)\Grad u_{\gamma_0,f}} \quad &&\text{ in }D, \\
		w&=0 \quad &&\text { on } \partial D.
		\end{aligned}
		\end{equation}
		We have the dual representation (see remark (ii) in \cref{sec:LaplaceBeltramiEigenfunctionsAndSobolevSpaces})  \[\norm[\Big]{\pd{w}{\nu}}_{H^{-1/2}(\partial D)}=\sup\braces[\Big]{ \abs[\big]{\ip[\Big]{ \pd{w}{\nu},v}_{L^2(\partial D)}} : v\in H^{1/2}(\partial D), \norm{v}_{H^{1/2}(\partial D)}=1}.\] For $v\in H^{1/2}(\partial D)$, by a standard trace theorem (e.g.\ Chapter I, Theorem 9.4 of \cite{Lions1972}) there exists $V\in H^1(D)$ such that $V|_{\partial D}=v$ and $\norm{V}_{H^1(D)}\leq C\norm{v}_{H^{1/2}(\partial D)}$ for a constant $C=C(D)$. Repeatedly applying the divergence theorem  (recalling that $\gamma=\gamma_0=1$ on $\partial D$) and the Cauchy--Schwarz inequality, and using \cref{eqn:DifferenceOfSolutionsPDE}, we deduce
		\begin{align*}
		\abs[\Big]{\int_{\partial D} \overline{v} \pd{w}{\nu}  } %&= \abs[\Big]{\int_D \Div (V\gamma \Grad w) \dx} \\
		&=\abs[\Big]{\int_D \overline{V}\Div (\gamma \Grad w)  + \int_D \gamma \Grad \overline{V} \cdot \Grad w } \\
		&\leq \abs[\Big]{\int_D \overline{V} \Div \brackets[\big]{(\gamma_0-\gamma)\Grad u_{\gamma_0,f}} }+\norm{\gamma}_{\infty}\norm{V}_{H^1(D)}\norm{\Grad w}_{L^2(D)} \\
		& \leq \abs[\Big]{\int_D (\gamma_0-\gamma) \Grad \overline{V} \cdot \Grad u_{\gamma_0,f} }+C\norm{\gamma}_{\infty}\norm{v}_{H^{1/2}(\partial D)}\norm{\Grad w}_{L^2(D)}\\
		&\leq C\norm{v}_{H^{1/2}(\partial D)} \brackets[\big]{\norm{\gamma_0-\gamma}_{\infty} \norm{\Grad u_{\gamma_0,f}}_{L^2(D)} + M_1\norm{\Grad w}_{L^2(D)}},
		\end{align*}
		hence \begin{equation}\label{eqn:NormPdWNuH-1/2} \norm[\Big]{\pd{w}{\nu}}_{H^{-1/2}(\partial D)}\leq C\brackets[\big]{\norm{\gamma_0-\gamma}_{\infty} \norm{\Grad u_{\gamma_0,f}}_{L^2(D)}+ M_1 \norm{\Grad w}_{L^2(D)}}.\end{equation} 
		A weak solution $w$ to \cref{eqn:DifferenceOfSolutionsPDE} by definition satisfies, for any $v\in H_0^1(D)$,
		\[\int_D \gamma \Grad w\cdot \Grad \overline{v}=\int_D (\gamma_0-\gamma)\Grad u_{\gamma_0,f}\cdot \Grad \overline{v}. \]
		In particular this applies with $v=w$, hence, since $\gamma\geq m$ on $D$, we apply the Cauchy--Schwarz inequality to deduce
		\[ m \norm{\Grad w}_{L^2(D)}^2 \leq \norm{\gamma_0-\gamma}_{\infty} \norm{\Grad u_{\gamma_0,f}}_{L^2(D)} \norm{\Grad w}_{L^2(D)}, \] which, returning to \cref{eqn:NormPdWNuH-1/2}, shows that
		\[ \norm[\Big]{\pd{w}{\nu}}_{H^{-1/2}(\partial D)}\leq C\norm{\Grad u_{\gamma_0,f}}_{L^2(D)} \norm{\gamma_0-\gamma}_{\infty} \brackets[\big]{1+\tfrac{M_1}{m}}.\]
		Applying the trace result Theorem 9.4 in \cite{Lions1972} Chapter I now to each representative of the equivalence class $f\in H^{1/2}(D)/\CC$ and optimising, there exists $F\in H^1(D)/\CC$ such that $F|_{\partial D}=f$ and $\norm{F}_{H^1(D)/\CC}\leq C\norm{f}_{H^{1/2}(\partial D)/\CC}$ for a constant $C=C(D)$. By definition of a weak solution to \cref{eqn:DirichletProblem} we have
		\[\int_D \gamma_0 \Grad u_{\gamma_0,f} \cdot \Grad \overline{(u_{\gamma_0,f}-F)}  = 0,\]
		and again applying the Cauchy--Schwarz inequality we deduce 
		\begin{equation}\label{eqn:u0H1Bound}
		\norm{\Grad u_{\gamma_0,f}}_{L^2(D)}\leq C\tfrac{M_0}{m}\norm{f}_{H^{1/2}(\partial D)/\CC}.
		\end{equation} 
		Overall we have shown \[\norm[\big]{(\Lambda_\gamma-\Lambda_{\gamma_0}) f}_{H^{-1/2}(\partial D)}\equiv \norm[\Big]{\pd{w}{\nu}}_{H^{-1/2}(\partial D)} \leq C\brackets[\big]{1+\tfrac{M_1}{m}}\tfrac{M_0}{m} \norm{\gamma-\gamma_0}_{\infty} \norm{f}_{H^{1/2}(\partial D)/\CC}.\] Taking the supremum over all $f$ with $H^{1/2}(\partial D)/\CC$ norm equal to 1, \cref{eqn:OperatorNormTildeLambdaGamma} follows. \qedhere
	\end{proof}
	
	\begin{proof}[Proof of \cref{lem:InverseStability}]
		When $d \ge 3$, Theorem 1 in Alessandrini \cite{Alessandrini1988} states that there exist constants $\delta=\delta(d,\beta)$ and $C=C(M,m,D, D', \beta)$  such that there is a (monotone) function $\omega$ satisfying 
		\begin{equation} \label{eqn:ModulusOfContinuityInverseMap} \norm{\gamma-\gamma_0}_{\infty} \leq C \omega(\norm{\Lambda_\gamma-\Lambda_{\gamma_0}}_*),\quad  \omega(t)\leq \log(1/t)^{-\delta}\text{ for }t<e^{-1}.
		\end{equation} 
		Likewise, when $d=2$, we can use Theorem 1.1 in Novikov \& Santacesaria \cite{NS10} in conjunction with the usual reduction of the Calder\'on problem to a suitable Schr\"odinger equation (e.g., see Lemma 4.8 and the proof of Theorem 4.1 in \cite{S08}) to show that \cref{eqn:ModulusOfContinuityInverseMap} still holds in this case. [In \cite{NS10}, an a priori bound for $\norm{\gamma}_{C^2}$ is assumed, which is derived here from the bound $M$ on $\norm{\gamma}_{H^\beta}$ using a Sobolev embedding.]
		
		To proceed, we appeal to \cref{lem:HHrNormEquivalentTo*Norm}, noting that $M$ upper bounds  $\norm{\gamma}_{\infty}$ and $\norm{\gamma_0}_{\infty}$  (up to a multiplicative constant) by a Sobolev embedding. Thus, for a constant $C'$ depending on $M,$ $m$, $D,$ $D'$, $\beta$ and $r$ we have
		\begin{align*}
		\omega(\norm{\Lambda_\gamma-\Lambda_{\gamma_0}}_*)&\leq \omega (C'\norm{\Lambda_\gamma-\Lambda_{\gamma_0}}_{\HH_r}^{1/2}) \\
		&\leq (\tfrac{1}{4} \log\brackets{\norm{\Lambda_\gamma-\Lambda_{\gamma_0}}_{\HH_r}^{-1}})^{-\delta},
		\end{align*}
		provided $\norm{\Lambda_\gamma-\Lambda_{\gamma_0}}_{\HH_r}< \min(e^{-2}(C')^{-2},(C')^{-4},1).$ The result follows. \qedhere
	\end{proof}

	\subsection{Tests and prior support properties}\label{sec:Proofs:TestsAndSmallBalls}
	In this section we prove two main auxiliary results. First, we prove the existence of certain statistical test functions required in the contraction theorem given in \cref{sec:Proofs:PosteriorContraction}. Instead of using robust `Hellinger-distance'-based testing as in \cite{vdVvZ08, Monard2019}, it is more convenient in the present setting (in part to deal with necessary boundedness restrictions on $\gamma$) to deduce the existence of tests from the existence of certain estimators with sufficiently good concentration properties (following ideas  in \cite{GN11}).  
	
	Recall $\Gamma_{m_1,D_1}$ is a superset of $\Gamma_{m_0,D_0}$ on which the prior on $\gamma=\Phi\circ \theta$ concentrates its mass. 
	
	\begin{lemma}\label{lem:ExistenceOfTests}
		Let $\gamma_0\in\Gamma_{m_0,D_0}$ be bounded by $M_0$ on $D$. Let $\eta_\eps>0$ satisfy $\eta_\eps\eps^{-(1-a)}\to \infty$ as $\eps\to 0$ for some $0<a<1$. For any $\kappa>0$ and $M_1>1 $, there exists a constant $C=C(\kappa,m_1,D, D_1,M_1,M_0)$ and tests (i.e. $\braces{0,1}$-valued measurable functions of the data) $\psi=\psi_\eps(Y)$ with $Y \sim P_\eps^\gamma$ from \cref{eqn:Model:HilbertSpaceMeasurements} such that for all $\varepsilon$ small enough
		\[\max \brackets{ E^{\gamma_0}_\eps \psi, \sup\braces{E^{\gamma}_\eps \sqbrackets{1-\psi} : \gamma\in \Gamma_{m_1,D_1},\norm{\gamma}_{\infty}\leq M_1, \norm{\Lambda_\gamma-\Lambda_{\gamma_0}}_{\HH_r}\geq C \eta_\eps} } \leq e^{-\kappa(\eta_\eps/\eps)^2}.\]
	\end{lemma}
	
	\begin{proof}
		We prove the existence of an estimator $\hat{\Lambda}$ satisfying that for any $\kappa>0,$ there exists a constant $C'=C'(\kappa,m_1,D_1, M_0, M_1,D)$ such that for all $\eps$ small enough,
		\begin{equation}\label{eqn:ConcentrationOfAnEstimator} 
		\sup\braces{ P^\gamma_\eps (\norm{\hat{\Lambda}-\tilde{\Lambda}_\gamma}_{\HH_r}>C'\eta_\eps) : \gamma \in \Gamma_{m_1,D_1},\norm{\gamma}_{\infty}\leq \max(M_0,M_1)}\leq e^{-\kappa(\eta_\eps/\eps)^2}.
		\end{equation} Then, setting $C=2C'$ and  $\psi_\eps(Y)=\II\braces{\norm{\hat{\Lambda}-\tilde{\Lambda}_{\gamma_0}}_{\HH_r}>\tfrac{1}{2}C\eta_\eps},$ the result follows from an application of the triangle inequality (see, e.g., the proof of Proposition 6.2.2 in \cite{Gine2016}).
		
		Define an estimator $\hat{\Lambda}$ by $\hat{\Lambda}=\sum_{j,k\leq J} \hat{\Lambda}_{jk}b_{jk}^{(r)}$, where $J=J_\eps=\floor{\eta_\eps/\eps}$ and  
		\begin{equation}\label{eqn:hatThetajk} \hat{\Lambda}_{jk}=\ip{Y,b_{jk}^{(r)}}_{\HH_r}=\ip{\tilde{\Lambda}_\gamma \phi_j^{(r)},\phi_k^{(0)}}_{L^2(\partial D)} + \eps g_{jk},~Y \sim P^\gamma_\eps,
		\end{equation} 
		where we note $g_{jk}=\ip{\WW,b_{jk}^{(r)}}_{\HH_r}\iidsim N(0,1)$. %\magenta{Note that $\hat{\Lambda}_{jk}\in\RR$ so that $\hat{\Lambda}\in \HH_r$ even with the new reality constraint} 
		Then we have the bias-variance decomposition
		\begin{equation}\label{eqn:BiasVarianceDecompositionLambdaHat}
		P^\gamma_\eps (\norm{\hat{\Lambda}-\tilde{\Lambda}_\gamma}_{\HH_r}>C'\eta_\eps) \leq \II\braces{ \norm{\tilde{\Lambda}_\gamma-\pi_{JJ}\tilde{\Lambda}_\gamma}_{\HH_r}>\tfrac{1}{2}C'\eta_\eps} + P^\gamma_\eps(\norm{\hat{\Lambda}-\pi_{JJ} \tilde{\Lambda}_\gamma}_{\HH_r}>\tfrac{1}{2}C'\eta_\eps).
		\end{equation}
		Recall, by \cref{lem:ProjectionErrorDecaysAsMinJK^-nu}, for any $\nu>0$ there is a constant $C_1=C_1(\nu,r,M_0,M_1,m_1,D_1,D)$ such that
		\begin{equation} \label{eqn:BiasIsSmall}
		\norm{\tilde{\Lambda}_\gamma-\pi_{JJ}\tilde{\Lambda}_\gamma}_{\HH_r} \leq C_1 J^{-\nu},
		\end{equation}
		hence the indicator in \cref{eqn:BiasVarianceDecompositionLambdaHat} is bounded by $\II\braces{C_1 J^{-\nu}>\tfrac{1}{2}C'\eta_\eps}$.
		Choosing $\nu>(1-a)/a$, %so that $\nu/(1+\nu)>1-a,$ 
		one finds that the assumption $\eta_\eps\eps^{-(1-a)}\to \infty$ ensures this term vanishes for $\eps$ small enough.
		
		For the variance term in \cref{eqn:BiasVarianceDecompositionLambdaHat}, observe that by Parseval's identity \begin{equation*}
		\norm{\hat{\Lambda}-\pi_{JJ} \tilde{\Lambda}_\gamma}_{\HH_r}^2 = \sum_{j,k\leq J} (\hat{\Lambda}_{jk}-\ip{\tilde{\Lambda}_\gamma\phi_j^{(r)},\phi_k^{(0)}}_{L^2(\partial D)})^2 = \eps^2\sum_{j,k \le J} g_{jk}^2.
		\end{equation*}
		One now applies a standard tail inequality (e.g., Theorem 3.1.9 in \cite{Gine2016}) to the effect that
		\begin{equation}\label{eqn:chi2kConcentration}
		\Pr\Big(\sum_{j,k \le J}g_{jk}^2\geq J^2+2J\sqrt{x}+2x\Big)\leq e^{-x}.
		\end{equation}
		For a constant $\kappa>0$, taking $x=\kappa(\eta_\eps/\eps)^2$, and for our choice $J =\floor{\eta_\eps/\eps}$, we see that for $C'$ large enough depending only on $\kappa$ we have
		\[
		P^\gamma_\eps (\norm{\hat{\Lambda}-\pi_{JJ} \tilde{\Lambda}_\gamma}_{\HH_r}^2>\frac{1}{2}C'\eta_\eps^2)\leq e^{-\kappa(\eta_\eps/\eps)^2},
		\] hence the result. \qedhere
	\end{proof}

	To proceed define $K(p,q)=E_{X\sim p}\log \frac{p}{q}(X)$ to be the Kullback--Leibler divergence between distributions with densities $p$ and $q$, and recall the definition of the probability densities $p_\eps^\gamma$ from \cref{eqn:LogLikelihood}. Also denote by $\Var_\gamma$ the variance operator associated to the probability measure $P^\gamma_\eps$. The following is then a standard result for a white noise model on a Hilbert space. 
	\begin{lemma}\label{lem:KLdistanceIsHHdistance}
		Let $\gamma_0,\gamma_1\in\Gamma_{m_1,D_1}$. Then $K(p^{\gamma_0}_\eps,p^{\gamma_1}_\eps)=\frac{1}{2}\eps^{-2}\norm{\Lambda_{\gamma_0}-\Lambda_{\gamma_1}}_{\HH_r}^2$, and $\Var_{\gamma_0}\brackets[\big]{\log\frac{p^{\gamma_0}_\eps}{p^{\gamma_1}_\eps}}=\eps^{-2}\norm{\Lambda_{\gamma_0}-\Lambda_{\gamma_1}}_{\HH_r}^2$.
	\end{lemma}
	\begin{proof}
		Using the explicit formula \cref{eqn:LogLikelihood} for the log-likelihoods, we see that under $\gamma_0$, \begin{align*}\ell(\gamma_0)-\ell(\gamma_1)&=\eps^{-2} \ip{Y,\tilde{\Lambda}_{\gamma_0}-\tilde{\Lambda}_{\gamma_1}}_{\HH_r} - \tfrac{1}{2}\eps^{-2} \norm{\tilde{\Lambda}_{\gamma_0}}_{\HH_r}^2+\tfrac{1}{2}\eps^{-2} \norm{\tilde{\Lambda}_{\gamma_1}}_{\HH_r}^2 \\
		&= \tfrac{1}{2}\eps^{-2} \norm{\tilde{\Lambda}_{\gamma_0}-\tilde{\Lambda}_{\gamma_1}}_{\HH_r}^2 + \eps^{-1} \ip{\WW,\tilde{\Lambda}_{\gamma_0}-\tilde{\Lambda}_{\gamma_1}}_{\HH_r},
		\end{align*}
		which is normally distributed with mean $\tfrac{1}{2}\eps^{-2}\norm{\tilde{\Lambda}_{\gamma_0}-\tilde{\Lambda}_{\gamma_1}}_{\HH_r}^2$ and variance $\eps^{-2}\norm{\tilde{\Lambda}_{\gamma_0}-\tilde{\Lambda}_{\gamma_1}}_{\HH_r}^2$. Noting that $\tilde{\Lambda}_{\gamma_0}-\tilde{\Lambda}_{\gamma_1}=\Lambda_{\gamma_0}-\Lambda_{\gamma_1}$, we deduce the result.\qedhere
	\end{proof}
	We define `balls' $B_{KL}^\eps(\eta)$ around the true parameter $\theta_0=\Phi^{-1}\circ \gamma_0$ by %(also writing $\gamma=\Phi\circ \theta$ for a generic $\theta$) 
	\begin{equation}\label{eqn:B_KLdefinition}
	B_{KL}^\eps(\eta) = \braces[\big]{\theta\in C_u(D) : K(p^{\gamma_0}_\eps,p^{\Phi\circ\theta}_\eps)\leq (\eta/\eps)^2, \Var_{\gamma_0}(\log(p^{\gamma_0}_\eps/p^{\Phi\circ \theta}_\eps))\leq (\eta/\eps)^2}.
	\end{equation}
	Then the following is an immediate consequence of \cref{lem:KLdistanceIsHHdistance}.
	\begin{corollary}\label{lem:HHNormBallsContainedInKLBalls}
		For any $\eta>0$, $\braces{\theta\in C_u(D) : \norm{\Lambda_{\Phi\circ \theta}-\Lambda_{\gamma_0}}_{\HH_r} \leq \eta}\subseteq B_{KL}^\eps(\eta)$.
	\end{corollary}
	
	With the preceding preparations, we can now prove the following support result for the prior $\Pi$ from \cref{thm:ContractionForTheta}, using a result of Li \& Linde \cite{Li1999}.
	\begin{lemma} \label{lem:LinftySmallBalls} Let $\eta_\eps=\eps^{\alpha/(\alpha+d)}$. Under the conditions of \cref{thm:ContractionForTheta}, there exists a constant  $\omega=\omega(\alpha,m_1,M,D,D_1,\Phi,r)>0$ such that $\Pi(B_{KL}^\eps(\eta_\eps))\geq e^{-\omega (\eta_\eps/\eps)^2}$ for all $\eps$ small enough, uniformly in $\gamma_0$ in the set \cref{set}. 
	\end{lemma}
	\begin{proof}
		By \cref{eqn:NormPhiCircThetaEqualsNormTheta} and a Sobolev embedding, there is a constant $M_0$ depending only on $\Phi$, $\alpha$ and $M$ such that $\norm{\gamma_0}_{\infty}\leq M_0$. By \cref{lem:ForwardStability}, we deduce, for a constant $C=C(r,D,D_1)$, 
		\begin{equation*}
		\norm{\Lambda_\gamma-\Lambda_{\gamma_0}}_{\HH_r} \leq C\frac{M_0\norm{\gamma}_\infty}{m_1^2}\norm{\gamma-\gamma_0}_\infty^{1/2} 
		\end{equation*} 
		provided $\norm{\gamma-\gamma_0}_{\infty}$ is small enough. It follows from this calculation and \cref{lem:HHNormBallsContainedInKLBalls} that for $\eta_\eps$ small enough and some constant $C'>0$ we have  \[\braces{\theta: \norm{\Phi\circ \theta -\gamma_0}_{\infty}\leq C'\eta_\eps^2}\subseteq \braces{\norm{\Lambda_{\Phi\circ \theta}-\Lambda_{\gamma_0}}_{\HH_r}\leq \eta_\eps}\subseteq B^\eps_{KL}(\eta_\eps).\] Appealing again to \cref{eqn:NormPhiCircThetaEqualsNormTheta}, we further deduce that \[\braces{\theta \in C_u(D) : \norm{\theta-\theta_0}_{\infty}\leq A\eta_\eps^2}\subseteq B_{KL}^\eps(\eta_\eps),\] for a constant $A=A(\alpha, M, m_1,D,D_1,r,\Phi)$, so that it remains to lower bound $\Pi \brackets[\big]{\norm{\theta-\theta_0}_{\infty}\leq A\eta_\eps^2}$. Note (recalling \cref{assumption:pco} and the definition of the prior \cref{eqn:tampering}) that $\Pi$ has RKHS $\Hh_\eps = \braces{ \zeta \theta' : \theta'\in \Hh}$, with norm $\norm{\cdot}_{\Hh_\eps}$ satisfying the bound $\norm{\theta}_{\Hh_\eps}\leq \eps^{-d/(\alpha+d)} \norm{\theta'}_{\Hh}=(\eta_\eps/\eps) \norm{\theta'}_{\Hh},$ for any $\theta'$ such that $\zeta \theta'=\theta$. Because $\theta_0=\Phi^{-1}\circ \gamma_0$ for some $\gamma_0\in \Gamma_{m_0,D_0}$, we see from the definitions of $\Phi$ and $\zeta$ that $\theta_0=\zeta \theta_0$, hence we deduce that $\norm{\theta_0}_{\Hh_\eps} \leq (\eta_\eps/\eps) \norm{\theta_0}_{\Hh}\leq M\eta_\eps/\eps$.  By Corollary 2.6.18 in \cite{Gine2016}, we then have
		\begin{align*}
		\Pi\brackets{\norm{\theta-\theta_0}_{\infty}\leq A\eta_\eps^2} & \geq e^{-\tfrac{1}{2}\norm{\theta_0}_{\Hh_\eps}^2} \Pi\brackets{\norm{\theta}_{\infty}\leq A\eta_\eps^2} \\
		& \geq e^{-\frac{1}{2}M^2(\eta_\eps/\eps)^2}\Pi' \brackets[\big]{\norm{\theta'}_{\infty}\leq A\tfrac{\eta_\eps^3}{\eps}}.
		\end{align*} 
		
		Next, since $\Hh$ embeds continuously into $H^\alpha(I_d)$ for some large enough cube $I_d$ (by a standard extension argument for Sobolev spaces), the unit ball $B_\mathcal H$ of $\mathcal H$ has covering numbers with respect to the supremum norm $N=N(B_\mathcal H,\norm{\cdot}_{\infty},\delta)$ (i.e.\ the smallest number of $\norm{\cdot}_{\infty}$ balls of radius $\delta$ needed to cover $B_\Hh$) satisfying
		\begin{equation}\label{eqn:CoveringNumbersForHalphaD} \log N(B_\mathcal H,\norm{\cdot}_{\infty},\delta)\leq K \delta^{-d/\alpha}\end{equation} for some constant $K=K(\alpha,D)$ (cf.~after Corollary 4.3.38 in \cite{Gine2016}). We can thus apply \cite{Li1999}, Theorem 1.2, to see \[\Pi' \brackets[\big]{\norm{\theta'}_{\infty}\leq A\tfrac{\eta_\eps^3}{\eps}}\geq e^{-A'\brackets{\frac{\eta_\eps^3}{\eps}}^{-s}},\] for some constant $A'=A'(A,K)$, where $s$ is such that $\frac{d}{\alpha}=\frac{2s}{2+s}$, i.e.\ $s=\frac{2d}{2\alpha-d}$. 
		
		Overall, we have shown $\Pi(B_{KL}^\eps(\eta_\eps))\geq e^{-\frac{1}{2}M^2(\eta_\eps/\eps)^2}e^{-A'\brackets{\eta_\eps^3/\eps}^{-2d/(2\alpha-d)}}$, where the constant $A'$ depends only on $D$, $\alpha$, $M$, $m_1$, $D_1$, $r$ and $\Phi$.  For $\eta_\eps=\eps^{\alpha/(\alpha+d)}$ we find $\brackets{\eta_\eps^3/\eps}^{-2d/(2\alpha-d)}= \brackets{\eta_\eps/\eps}^2$, and the result follows.\qedhere
	\end{proof}

	\subsection{Posterior asymptotics}\label{sec:Proofs:PosteriorContraction}
	
	\subsubsection{Posterior regularity and contraction about $\Lambda_{\gamma_0}$}
	
	The following two results follow ideas from Bayesian nonparametric statistics \cite{vdVvZ08, Ghosal2017} combined with the lemmas from the previous section. Together with the stability estimate \cref{lem:InverseStability}, these two estimates allow us to proceed as in \cite{Monard2019} to prove \cref{thm:ContractionForTheta}.
	\begin{lemma}\label{lem:C2ConcentrationOfPosterior} Let $\eta_\eps,\omega$ be as in \cref{lem:LinftySmallBalls}. Under the assumptions of \cref{thm:ContractionForTheta} there exists $M'>0$ such that
		
		\begin{equation}\label{eqn:C2ConcentrationOfPosteriorAtExponentialRate} \sup_{\gamma_0 \in \Gamma_{m_0,D_0}\cap\braces{\Phi\circ \theta : \theta\in\Hh,\norm{\theta}_{\Hh}\leq M}} P^{\gamma_0}_\eps \brackets[\big]{\Pi\brackets{\norm{\gamma}_{H^\beta(D)}>M' \mid Y} > e^{-(\omega+4)(\eta_\eps/\eps)^2}}\to 0,\end{equation} 
		as $\eps\to 0$. The bound \cref{eqn:C2ConcentrationOfPosteriorAtExponentialRate} also holds with the supremum norm $\norm{\cdot}_\infty$ in place of the $H^\beta(D)$ norm. 
	\end{lemma}
	\begin{theorem}\label{thm:ContractionForLambda} Let $\eta_\eps,\omega$ be as in \cref{lem:LinftySmallBalls}. Under the conditions of \cref{thm:ContractionForTheta}, there exists $C>0$ such that  %$C=C(m_0,m_1,M,D,D_1,\Phi,r)$.$\delta=\delta(d)>0$ 
		\begin{equation} \label{eqn:PosteriorConcentratesExponentiallyFast} \sup_{\gamma_0\in \Gamma_{m_0,D_0}
			\cap\braces{\Phi\circ \theta : \theta\in \Hh, \norm{\theta}_{\Hh}\leq M} }P^{\gamma_0}_\eps \brackets[\big]{ \Pi (\norm{\Lambda_\gamma-\Lambda_{\gamma_0}}_{\HH_r}> C\eta_\eps \mid Y ) > 2e^{-(\omega+4)(\eta/\eps)^2} } \to 0,
		\end{equation}
		as $\eps\to 0$.
	\end{theorem}

	To prove the preceding results, we note that the posterior from \cref{eqn:PosteriorDefinition} can be written as
	\begin{equation}\label{posto}
	\Pi(B \mid Y)= \frac{\int_B p^\gamma_\eps(Y)/p^{\gamma_0}_\eps(Y) \dPi(\gamma)}{\int_{\Gamma_{m_1,D_1}} p^\gamma_\eps(Y)/p^{\gamma_0}_\eps(Y) \dPi(\gamma)}, \quad B\subset \Gamma_{m_1,D_1} \text{ measurable}.
	\end{equation}
	
	The following lemma controls the size of the denominator in the last expression.
	
	\begin{lemma}\label{lem:ELBO}
		Let $\eta_\eps,\omega$ be as in \cref{lem:LinftySmallBalls}. Introduce the event \[L=L_{\omega,\gamma_0}=\braces[\Big]{\int_{\Gamma_{m_1,D_1}} \frac{p^\gamma_\eps}{p^{\gamma_0}_\eps}(Y)\dPi(\gamma)\geq e^{-(\omega+2)(\eta_\eps/\eps)^2}}.\] Then, under the assumptions of \cref{thm:ContractionForTheta}, we have
		\begin{equation}\sup_{\gamma_0 \in \Gamma_{m_0,D_0}\cap\braces{\Phi\circ \theta : \theta\in\Hh,\norm{\theta}_{\Hh}\leq M}} P^{\gamma_0}_\eps(L^c)\to 0, \quad \text{as}~~\eps\to 0.\end{equation}  
	\end{lemma}
	\begin{proof} Given \cref{lem:LinftySmallBalls}, the proof is a standard argument based on Chebyshev's inequality, Jensen's inequality and Fubini's theorem (for example, see \cite{Gine2016} Lemma 7.3.4 for a proof in the setting of white noise on $L^2([0,1])$ which adapts straightforwardly to the current Hilbert space setting, with probability measure $\nu$ there equal to the renormalised restriction of $\dPi(\gamma)$ to $B_{KL}^\eps(\eta_\eps)$). \qedhere
	\end{proof}

	\begin{proof}[Proof of \cref{lem:C2ConcentrationOfPosterior}]
		Define $L$ as in \cref{lem:ELBO}. By \cref{posto}, using Fubini's theorem and the fact that $E^{\gamma_0}_\eps \frac{p^\gamma_\eps}{p^{\gamma_0}_\eps}(Y) = 1$, we see that for every Borel set $B$
		\[ E^{\gamma_0}_\eps \sqbrackets{\II_L \Pi \brackets{B \mid Y}} \leq e^{(\omega+2)(\eta_\eps/\eps)^2} E^{\gamma_0}_\eps \int_B \frac{p^\gamma_\eps}{p^{\gamma_0}_\eps}(Y) \dPi(\gamma)= e^{(\omega+2)(\eta_\eps/\eps)^2} \Pi(B). \] Then, setting $B=\braces{\norm{\gamma}_{H^\beta(D)}>M'}$, an application of Markov's inequality yields
		\begin{equation*}
		P^{\gamma_0}_\eps \brackets[\big]{\Pi\brackets{\norm{\gamma}_{H^\beta(D)}> M' \mid Y} > e^{-(\omega+4)(\eta_\eps/\eps)^2}} \leq P^{\gamma_0}_\eps(L^c) + e^{(2\omega+6)(\eta_\eps/\eps)^2} \Pi(\norm{\gamma}_{H^\beta(D)}>M').
		\end{equation*}
		The first term on the right vanishes asymptotically, uniformly across $\gamma_0$ in the given set, by \cref{lem:ELBO}. For the second, we recall \cref{eqn:NormPhiCircThetaBound} and the definition \cref{eqn:tampering} of the prior. In conjunction with the facts that $\eps^{-d/(\alpha+d)}=\eta_\eps/\eps$ for our choice of $\eta_\eps$ and that $\norm{\zeta \theta'}_{H^\beta(D)}\leq C\norm{\zeta}_{H^\beta(D)}\norm{\theta'}_{H^\beta(D)}$ for some constant $C$ (since $\beta>d/2$), these allow us to deduce that
		\[ \Pi\brackets{\norm{\gamma}_{H^\beta(D)}>M'}\leq \Pi'\brackets[\Big]{\norm{\theta'}_{H^\beta(D)}>\tfrac{\eta_\eps}{\eps}(C\norm{\zeta}_{H^\beta(D)})^{-1}(M'/C'-1)^{1/\beta}}.\] 
		Since $\eta_\eps/\eps \to \infty$ and since $\Pi'(H^\beta(D))=1$ by hypothesis, we can apply a version of Fernique's theorem, more specifically Theorem 2.1.20  in \cite{Gine2016} (see also Example 2.1.6 therein) to deduce that for any $c>0$ there exists a $M'=M'(\beta,c, C',\zeta)$ such that the last probability does not exceed $e^{-c(\eta_\eps/\eps)^2}$. Taking $c>2\omega+6$ concludes the proof for the $H^\beta(D)$ norm, and the result for the supremum norm follows by a Sobolev embedding and adjusting the constant $M'$. \qedhere
	\end{proof}

	\begin{proof}[Proof of \cref{thm:ContractionForLambda}]
		We decompose in a standard way: writing \[S=\braces{\gamma \in \Gamma_{m_1,D_1} : \norm{\Lambda_\gamma-\Lambda_{\gamma_0}}_{\HH_r} > C\eta_\eps,~\norm{\gamma}_{\infty}\leq M'},\] for $M'$ the constant of \cref{lem:C2ConcentrationOfPosterior} and $C$ a large enough constant (to be chosen below), we have, for $\psi$ the test given by \cref{lem:ExistenceOfTests} and $L$ as in \cref{lem:ELBO},
		\begin{equation}\label{eqn:DecompositionOfPosteriorFarawayMass}
		\begin{aligned}
		\Pi \brackets[\big]{\norm{\Lambda_\gamma-\Lambda_{\gamma_0}}_{\HH_r} > C\eta_\eps \mid Y} &\leq \Pi \brackets[\big]{S \mid Y} + \Pi \brackets[\big]{\norm{\gamma}_{\infty} > M' \mid Y}, \\
		\Pi \brackets{S \mid Y} &\leq \II_{L^c}+\psi +\Pi\brackets{S \mid Y} \II_L [1-\psi].
		\end{aligned}
		\end{equation}
		Hence, denoting by $\Cc$ the event $\braces[\big]{ \Pi (\norm{\Lambda_\gamma-\Lambda_{\gamma_0}}_{\HH_r}> C\eta_\eps \mid Y ) > 2e^{-(\omega+4)(\eta_\eps/\eps)^2} }$, we have
		\begin{align*}&P^{\gamma_0}_\eps (\Cc)  \leq P^{\gamma_0}_\eps\brackets[\Big]{\Pi \brackets[\big]{\norm{\gamma}_{\infty} > M' \mid Y}>e^{-(\omega+4)(\eta_\eps/\eps)^2}}+ P^{\gamma_0}_\eps\brackets[\Big]{\Pi\brackets{S \mid Y}>e^{-(\omega+4)(\eta_\eps/\eps)^2}}, \\
		&P^{\gamma_0}_\eps\brackets[\Big]{\Pi\brackets{S \mid Y}>e^{-(\omega+4)(\eta_\eps/\eps)^2}} \leq P^{\gamma_0}_\eps(L^c)+E^{\gamma_0}_\eps \psi+ P^{\gamma_0}_\eps\brackets[\Big]{\Pi\brackets{S \mid Y} \II_L [1-\psi]>e^{-(\omega+4)(\eta_\eps/\eps)^2}}. \end{align*} In view of \cref{lem:C2ConcentrationOfPosterior,lem:ExistenceOfTests,lem:ELBO}, it suffices to show that there exists a constant $C$ such that $P^{\gamma_0}_\eps\brackets[\Big]{\Pi\brackets{S \mid Y} \II_L [1-\psi]>e^{-(\omega+4)(\eta_\eps/\eps)^2}}\to 0$, uniformly across $\gamma_0$ in the given set, and we note that $\norm{\gamma_0}_\infty$ is uniformly bounded in the set by a Sobolev embedding. Appealing to \cref{posto}, Fubini's theorem and again \cref{lem:ExistenceOfTests} we have for every $\kappa>0$ and for $C$ large enough (depending on $\kappa$),
		\begin{align*} 
		E^{\gamma_0}_\eps \sqbrackets[\Big]{\Pi \brackets{S \mid Y}\II_L (1-\psi)}  & \leq e^{(\omega+2)(\eta_\eps/\eps)^2} E^{\gamma_0}_\eps \int_{S} \frac{p^\gamma_\eps}{p^{\gamma_0}_\eps}(Y) (1-\psi)(Y) \dPi(\gamma) \\
		& \leq e^{(\omega+2)(\eta_\eps/\eps)^2} \int_{S} E_\eps^\gamma \sqbrackets{(1-\psi)(Y)}\dPi(\gamma) \\
		& \leq e^{(\omega+2-\kappa)(\eta_\eps/\eps)^2},
		\end{align*}
		hence by Markov's inequality, the probability in question is bounded by 
		$e^{(2\omega+6-\kappa)(\eta_\eps/\eps)^2}$. This tends to zero, uniformly in $\gamma_0$, if $C$ is large enough to permit $\kappa>2\omega+6$.  \qedhere
	\end{proof}
	
	\subsubsection{Proof of \cref{thm:ContractionForTheta}}
	
	Recall \Cref{lem:InverseStability} to the effect that
	\[\norm{\gamma-\gamma_0}_{\infty} \leq C' \abs{\log \norm{\Lambda_\gamma-\Lambda_{\gamma_0}}_{\HH_r}}^{-\delta},\] at least for $\norm{\Lambda_\gamma-\Lambda_{\gamma_0}}_{\HH_r}$ small enough, for some $\delta=\delta(d, \beta)>0$ and a constant $C'$ depending only on $M$, $D$, $D_1$, $\alpha$, $\beta$, $r$, and an upper bound for $\norm{\gamma}_{H^\beta(D)}$, where we have also used $\norm{\gamma_0}_{H^\beta(D)} \leq \norm{\gamma_0}_{H^\alpha(D)}$, eq.~\cref{eqn:NormPhiCircThetaBound} and the hypothesis on $\gamma_0$. Thus, for $M'>0$ the constant of \cref{lem:C2ConcentrationOfPosterior} and $C_1>0$ large enough that \cref{thm:ContractionForLambda} applies with constant $C_1$, we have \[\brackets{\norm{\gamma}_{H^\beta(D)} \leq M'}\wedge \brackets{\norm{\Lambda_\gamma-\Lambda_{\gamma_0}}_{\HH_r}\leq C_1\eta_\eps} \implies \norm{\gamma-\gamma_0}_{\infty}\leq C'\xi_{\eps,\delta}.\] In view of \cref{eqn:NormPhiInverseCircGammaLinfty}, for $\eps$ small enough (noting that since $C'\xi_{\eps,\delta}\to 0$ it is eventually smaller than $m_0-m_1$), we further deduce for a constant $C$ that
	\begin{equation} \label{eqn:PosteriorProbOfDistanceThetaControl}\Pi\brackets{\norm{\theta-\theta_0}_{\infty}>C\xi_{\eps,\delta} \mid Y} \leq \Pi \brackets{\norm{\Lambda_\gamma-\Lambda_{\gamma_0}}_{\HH_r}>C_1\eta_\eps \mid Y}+\Pi\brackets{\norm{\gamma}_{H^\beta(D)} >M' \mid Y}. \end{equation} Then \cref{thm:ContractionForLambda,lem:C2ConcentrationOfPosterior} imply the contraction rate \cref{eqn:ContractionForTheta}.
	
	To prove consistency of the posterior mean $E[\theta \mid Y]$, introduce, for $\omega,L$ as in \cref{lem:ELBO}, the event
	\[\Aa=L \cap \braces[\Big]{\Pi\brackets{\norm{\theta-\theta_0}_{\infty}>C\xi_{\eps,\delta} \mid Y} \leq 3e^{-(\omega+4)(\eta_\eps/\eps)^2}},\] and note that
	\begin{equation}\label{eqn:BoundingDistanceOfPosteriorMean1}
	P^{\gamma_0}_\eps \brackets[\big]{\norm{E^{\Pi}\sqbrackets{\theta\mid Y}-\theta_0}_{\infty}>K\xi_{\eps,\delta}} \leq P^{\gamma_0}_\eps(\Aa^c) + P^{\gamma_0}_\eps\brackets[\big]{\norm{E^{\Pi}\sqbrackets{\theta-\theta_0 \mid Y}}_{\infty} \II_\Aa > K\xi_{\eps,\delta}} .
	\end{equation}
	In view of \cref{eqn:PosteriorProbOfDistanceThetaControl}, observe that $\Aa^c$ is contained in
	\[ L^c \cup \braces[\Big]{\Pi \brackets{\norm{\Lambda_\gamma-\Lambda_{\gamma_0}}_{\HH_r}>C_1\eta_\eps \mid Y}>2e^{-(\omega+4)(\eta_\eps/\eps)^2}}\cup \braces[\Big]{\Pi\brackets{\norm{\gamma}_{H^\beta(D)} >M' \mid Y}>e^{-(\omega+4)(\eta_\eps/\eps)^2}}, \] hence, by \cref{lem:C2ConcentrationOfPosterior,thm:ContractionForLambda,lem:ELBO}, $P^{\gamma_0}_\eps(\Aa^c)\to 0$ as $\eps\to 0$, uniformly in $\gamma_0$.
	Now for the second term in \cref{eqn:BoundingDistanceOfPosteriorMean1}, by Jensen's inequality and the Cauchy--Schwarz inequality, we have
	\begin{align*}\norm{E^{\Pi}\sqbrackets{\theta-\theta_0 \mid Y}}_{\infty} \II_\Aa &\leq C\xi_{\eps,\delta} + E^{\Pi}\sqbrackets [\Big]{\norm{\theta-\theta_0}_{\infty}\II\braces{\norm{\theta-\theta_0}_{\infty}>C\xi_{\eps,\delta}} \mathrel{\big |} Y} \II_\Aa \\
	&\leq C\xi_{\eps,\delta} +\brackets[\big]{E^{\Pi} \sqbrackets{ \norm{\theta-\theta_0}_{\infty}^2 \mid Y}}^{1/2} \Pi\brackets{\norm{\theta-\theta_0}_{\infty}>C\xi_{\eps,\delta} \mid Y}^{1/2} \II_\Aa,
	\end{align*}
	and by Markov's inequality, it follows for $K>C$ that
	\begin{equation}\label{eqn:BoundingDistanceOfPosteriorMean2} \begin{split} &P^{\gamma_0}_\eps (\norm{E^{\Pi}\sqbrackets{\theta-\theta_0 \mid Y}}_{\infty} \II_\Aa > K\xi_{\eps,\delta} ) \leq \\ &\quad \tfrac{1}{(K-C)\xi_{\eps,\delta}} E^{\gamma_0}_\eps \sqbrackets[\Big]{ \brackets[\big]{E^{\Pi} \sqbrackets{ \norm{\theta-\theta_0}_{\infty}^2 \mid Y}}^{1/2} \brackets{\Pi\braces{\norm{\theta-\theta_0}_{\infty}>C\xi_{\eps,\delta} \mid Y}}^{1/2} \II_\Aa }.\end{split}\end{equation}
	Again applying the Cauchy--Schwarz inequality the last expected value is bounded by
	\begin{equation*} \quad E^{\gamma_0}_\eps \sqbrackets[\big]{ E^{\Pi} \sqbrackets{ \norm{\theta-\theta_0}_{\infty}^2 \mid Y}\II_\Aa}^{1/2} E^{\gamma_0}_\eps\sqbrackets[\big]{ \Pi\brackets{\norm{\theta-\theta_0}_{\infty}>C\xi_{\eps,\delta} \mid Y}\II_\Aa}^{1/2}. \end{equation*}
	From the definition of the event $\Aa$ it is immediate that
	\[E^{\gamma_0}_\eps\sqbrackets[\big]{ \Pi\brackets{\norm{\theta-\theta_0}_{\infty}>C\xi_{\eps,\delta} \mid Y} \II_\Aa} \leq 3e^{-(\omega+4)(\eta_\eps/\eps)^2},\]
	while, applying \cref{posto} and Fubini's theorem and since $E^{\gamma_0}_\eps\tfrac{p^{\Phi\circ \theta}_\eps}{p^{\gamma_0}_\eps}(Y)=1$, we have \begin{align*}
	E^{\gamma_0}_\eps \sqbrackets[\big]{ E^{\Pi} \sqbrackets{ \norm{\theta-\theta_0}_{\infty}^2 \mid Y}\II_\Aa} &\leq e^{(\omega+2)(\eta_\eps/\eps)^2} E^{\gamma_0}_\eps \sqbrackets[\Big]{  \int_{C_u(D) } \norm{\theta-\theta_0}_{\infty}^2 \tfrac{p^{\Phi\circ \theta}_\eps}{p^{\gamma_0}_\eps}(Y) \dPi(\theta)} \\
	& \leq e^{(\omega+2)(\eta_\eps/\eps)^2} E^{\Pi}\norm{\theta-\theta_0}_{\infty}^2.
	\end{align*}
	Plugging this back into \cref{eqn:BoundingDistanceOfPosteriorMean2}, we see
	\begin{equation*} P^{\gamma_0}_\eps (\norm{E^{\Pi}\sqbrackets{\theta-\theta_0 \mid Y}}_{\infty} \II_\Aa > K\xi_{\eps,\delta} ) \leq \tfrac{\sqrt{3}}{(K-C)\xi_{\eps,\delta}} \brackets[\Big]{e^{(\omega+2)(\eta_\eps/\eps)^2} E^{\Pi}\norm{\theta-\theta_0}_{\infty}^2 e^{-(\omega+4)(\eta_\eps/\eps)^2}}^{1/2}.\end{equation*}
	Note that, since $E^{\Pi}\norm{\theta-\theta_0}_{\infty}^2 \leq 2\brackets{ \norm{\theta_0}_{\infty}^2+E^{\Pi}\norm{\theta}_{\infty}^2}$ and $E^{\Pi'}\norm{\theta'}^2_{\infty}$ is finite (Exercise 2.1.2 and Theorem 2.1.20a in \cite{Gine2016}), a Sobolev embedding combined with the prior definition \cref{eqn:tampering} implies that $E^{\Pi}\norm{\theta-\theta_0}_{\infty}^2$ is bounded uniformly across the specified $\theta_0$'s. Since $e^{- (\eta_\eps/\eps)^2}/\xi_{\eps,\delta}\to 0$, we see, returning to \cref{eqn:BoundingDistanceOfPosteriorMean1}, that the result follows.

	\subsection{Proof of the lower bound \cref{thm:MinimaxLowerBound}} \label{sec:Proofs:LowerBound}
	Recall the shorthand \cref{eqn:XiEpsDelta} and also the definition of $K(p,q)$ from before \cref{lem:KLdistanceIsHHdistance}. It is enough to find $\gamma_0,\gamma_1\in\Gamma^\alpha_{m_0,D_0}(M)$ (both allowed to depend on $\eps$) such that, for some $\mu$ small enough (to be chosen), 
	\begin{enumerate}[i.]
		\item $\norm{\gamma_1-\gamma_0}_{\infty}\geq \xi_{\eps,\delta'}\equiv \log(1/\eps)^{-\delta'}$
		\item $K(p^{\gamma_1}_\eps,p^{\gamma_0}_\eps)\leq \mu . $
	\end{enumerate}
	Indeed, the standard proof of this reduction (for example as in \cite{Gine2016}, Theorem 6.3.2, or Chapter 2 in \cite{T09}) is as follows. Under condition (i), noting that $\psi=\II\braces{\norm{\hat{\gamma}-\gamma_1}_{\infty}< \norm{\hat{\gamma}-\gamma_0}_{\infty}}$ yields a test of $H_0: \gamma=\gamma_0$ against $H_1: \gamma=\gamma_1$, we see 
	\begin{equation*}\label{eqn:MinimaxRiskLowerBoundedBy2PointTests} \inf_{\hat{\gamma}} \sup_{\gamma\in\Gamma^\alpha_{m_0, D_0}(M)} P^\gamma_\eps\big(\norm{\hat{\gamma}-\gamma}_{\infty} \geq \tfrac{1}{2} \xi_{\eps,\delta'} \big) \ge \inf_\psi \max\brackets{P^{\gamma_0}_\eps(\psi\not = 0),P^{\gamma_1}_\eps(\psi\not= 1)}, \end{equation*} where the latter infimum is over all tests $\psi$. Introducing the event $A=\braces[\big]{\frac{p^{\gamma_0}_\eps}{p^{\gamma_1}_\eps}\geq 1/2}$, we see
	\begin{equation*}
	P^{\gamma_0}_\eps(\psi\not = 0) \geq  E^{\gamma_1}_\eps\sqbrackets[\big]{\tfrac{p^{\gamma_0}_\eps}{p^{\gamma_1}_\eps} \II_A \psi}
	\geq \tfrac{1}{2}\sqbrackets{ P^{\gamma_1}_\eps\brackets{\psi=1}- P^{\gamma_1}_\eps(A^c)}
	\end{equation*}
	Thus, writing $p_1=P^{\gamma_1}_\eps(\psi=1)$, we see
	\begin{align*} \max\brackets{P^{\gamma_0}_\eps(\psi\not=0),P^{\gamma_1}_\eps(\psi\not =1)} &\geq \max \brackets{ \tfrac{1}{2}(p_1-P^{\gamma_1}_\eps(A^c)),1-p_1} \\ &\geq \inf_{p\in[0,1]} \max \brackets{ \tfrac{1}{2}(p-P^{\gamma_1}_\eps(A^c)),1-p}. \end{align*}
	The infimum is attained when $\frac{1}{2}(p-P^{\gamma_1}_\eps(A^c))=1-p$ and takes the value $\frac{1}{3} P^{\gamma_1}_\eps(A)$ so that
	\begin{equation} \label{eqn:MinimaxRiskAnIntermediateLowerBound}
	\inf_{\hat{\gamma}} \sup_{\gamma\in\Gamma^\alpha_{m_0, D_0}(M)} P^\gamma_\eps\big(\norm{\hat{\gamma}-\gamma}_{\infty} \geq \tfrac{1}{2} \xi_{\eps,\delta'} \big) \geq \tfrac{1}{3} P^{\gamma_1}_\eps(A).
	\end{equation}
	Next observe
	\begin{align*}
	P^{\gamma_1}_\eps(A)&= P^{\gamma_1}_\eps\sqbrackets[\big]{\tfrac{p^{\gamma_1}_\eps}{p^{\gamma_0}_\eps} \leq  2}= 1- P^{\gamma_1}_\eps\sqbrackets[\big]{\log\brackets[\big]{\tfrac{p^{\gamma_1}_\eps}{p^{\gamma_0}_\eps}}
		>\log 2} \geq 1- P^{\gamma_1}_\eps\sqbrackets[\big]{\abs{\log(\tfrac{p^{\gamma_1}_\eps}{p^{\gamma_0}_\eps})}>\log 2} \\ &\geq 1-(\log{2})^{-1} E^{\gamma_1}_\eps \abs[\Big]{\log \brackets[\big]{ \tfrac{p^{\gamma_1}_\eps}{p^{\gamma_0}_\eps}}},
	\end{align*}
	where we have used Markov's inequality to attain the final expression.
	By the second Pinsker inequality (Proposition 6.1.7b in \cite{Gine2016}), using condition (ii) we can continue the chain of inequalities to see
	\[ P^{\gamma_1}_\eps(A) \geq 1-(\log 2)^{-1}\sqbrackets[\big]{K(p^{\gamma_1}_\eps,p^{\gamma_0}_\eps) + \sqrt{2K(p^{\gamma_1}_\eps,p^{\gamma_0}_\eps)} } \geq 1- (\log 2)^{-1} \brackets{\mu +\sqrt{2\mu}} . \] 
	Choosing $\mu$ small enough, we can thus lower bound \cref{eqn:MinimaxRiskAnIntermediateLowerBound} by
	\[\tfrac{1}{3} \brackets[\Big]{1- \frac{\mu+\sqrt{2\mu}}{\log 2}} >\frac{1}{4},\]  
	so that \cref{thm:MinimaxLowerBound} will follow.
	
	Now we prove the existence of $\gamma_0,\gamma_1\in\Gamma^\alpha_{m_0,D_0}(M)$ satisfying conditions (i) and (ii). We appeal to Corollary 1 in \cite{Mandache2001}, which says that for any integer $k\geq 2$, any $q\geq 0$, some $B>0$ and any $\xi\in(0,1)$ there exist $\gamma_0,\gamma_1$ such that  $\support(\gamma_j-1)\subseteq D_0$, $\gamma_j\geq 1$ on $D$ for $j=0,1$, and \begin{enumerate}[a.]
		\item $\norm{\gamma_1-\gamma_0}_{\infty} \geq \xi,$
		\item $\norm{\Lambda_{\gamma_1}-\Lambda_{\gamma_0}}_{H^{-q}(\partial D)/\CC\to H^q_\diamond(\partial D)} \leq \exp\brackets[\big]{-\xi^{-\frac{d}{(2d-1)k}}},$
		\item $\max\brackets{\norm{\gamma_1}_{C^k(D)},\norm{\gamma_0}_{C^k(D)}}\leq B$,
	\end{enumerate}
	where the norm in (b) is the operator norm and where $\norm{\cdot}_{C^k(D)}$ denotes the usual norm on the space $C^k(D)$ of functions with bounded continuous partial derivatives up to order $k$.   (Note that \cite{Mandache2001} states this with full space $H^{-q}(\partial D)$ in place of the quotient space, but since $\Lambda_{\gamma_j}$ maps constant functions to $0$ for $j=0,1$, the two operator norms coincide.) For $M$ sufficiently large, we may take $k=\alpha$ to deduce (noting that $C^\alpha(D)$ continuously embeds into $H^\alpha(D)$) that there exist $\gamma_0,\gamma_1\in \Gamma^\alpha_{m_0,D_0}(M)$ satisfying (a) and (b).
	Taking $\xi=\xi_{\eps,\delta'}$ we note that (i) holds by definition.
	
	For (ii), applying \cref{lem:ProjectionsApproximateHSOperatorsWell} with $s=0$, $p=\min(d-1,r)$ and $q=(d-1-r)_+\equiv \max(d-1-r,0)=d-1-p$ we see that, for a constant $C=C(d,r)$, 
	\[\norm{\Lambda_{\gamma_1} - \Lambda_{\gamma_0}}_{\HH_r}  \leq C\norm{\Lambda_{\gamma_1}-\Lambda_{\gamma_0}}_{H^{-q}(\partial D)/\CC\to H^q_\diamond(\partial D)}.\] 
	Thus, appealing to \cref{lem:KLdistanceIsHHdistance}, we can bound the KL-divergences $K(p^{\gamma_1}_\eps,p^{\gamma_0}_\eps)$ by %(half of)
	\[ \eps^{-2}\norm{\Lambda_{\gamma_1}-\Lambda_{\gamma_0}}_{{\HH_r}}^2 \leq C^2 \eps^{-2}\norm{\Lambda_{\gamma_1}-\Lambda_{\gamma_0}}_{H^{-q}\to H^{q}}^2\leq C^2 \exp\sqbrackets[\big]{2\log(1/\eps)-(\log(1/\eps))^{-\frac{\delta' d}{(2d-1)\alpha}}}.\] Since $\delta'>\alpha(2d-1)/d$ by assumption, the final expression tends to zero as $\eps\to 0$ and in particular is less than the $\mu$ of condition (ii) for $\eps$ small enough.
	
	\appendix

	\section{Laplace--Beltrami eigenfunctions and Sobolev spaces} \label{sec:LaplaceBeltramiEigenfunctionsAndSobolevSpaces}
	In this \namecref{sec:LaplaceBeltramiEigenfunctionsAndSobolevSpaces}, we define the Sobolev spaces $H^r(D),$ $H^r(\partial D)$ for a bounded domain $D\subset \RR^d$, with smooth boundary $\partial D$.
	\begin{definition*}[$H^r(D)$]
		We follow \cite{Lions1972} to define these Sobolev spaces: see Chapter I, Sections 1.1, 9.1 and 12.1 (pages 1, 40 and 70 respectively) for details. For $r\in\NN\cup\braces{0}$ we set \[H^r(D)=\braces{f\in L^2(D) : \norm{f}_{H^r(D)}^2=\sum_{\abs{\alpha}\leq r} \norm{D^\alpha f}_{L^2(D)}^2 <\infty},\] where for a multi-index  $\alpha=(\alpha_1,\dots,\alpha_d),$ $\alpha_j\in\NN\cup\braces{0}$ for $j\leq d$, of order $\abs{\alpha}=\sum_{j\leq d} \alpha_j$, \[ D^\alpha= \frac{\partial^{\abs\alpha}}{\partial x_1^{\alpha_1} \dots \partial x_d^{\alpha_d}},\] with partial derivatives defined in a weak sense. $L^2(D)$ is defined with respect to the Lebesgue measure on $D$. Note $H^r(D)$ is a Hilbert space, with inner product
		\[\ip{f,g}_{H^r(D)}=\sum_{\abs{\alpha}\leq r} \int D^\alpha f \cdot \overline{D^\alpha g}.\]
		For $r\in\RR$, $r\geq 0$ we then define $H^r(D)$ via interpolation. Finally, defining \[H^r_0(D)=\braces{f\in H^r(D) : \tfrac{\partial^j f}{\partial \nu^j}\big|_{\partial D}=0, 0 \le j <r-1/2},\] with the normal boundary derivatives defined in a trace sense, we define $H^{r}(D), r<0,$ as the topological dual space $(H_0^{\abs{r}}(D))^*$, equipped with the dual norm. [Cf.~also \cite{Lions1972} Chapter I, Sections 11.1, 11.4 and 12.1 on pages 55, 62 and 70.] 
	\end{definition*}
	
	For $C^\infty_c(D)$ the space of smooth functions compactly supported in $D$, $H^1_\loc(D)$ is defined as 
	\[ H^1_\loc(D)=\braces{f \text{ locally integrable }: f\phi \in H^1(D) \text{ for all $\phi\in C_c^\infty(D)$}}\] (see \cite{Lions1972}, Chapter II, Section 3.2 on p125), or, equivalently,
	\[H^1_\loc(D)=\braces{f \text{ locally integrable } : f|_U\in H^1(U) \text{ for all domains $U$ satisfying $\bar{U}\subset D$}}.\]
	
	\medskip 
	
	To define the Sobolev space $H^r(\partial D)$ for the compact boundary manifold $\partial D$, let $(\phi_k=\phi_k^{(0)}:k\in\NN\cup\braces{0})$ be an orthonormal basis of $L^2(\partial D)$ consisting of real-valued eigenfunctions of the Laplace--Beltrami operator $\Laplacian_{\partial D}$. The basic properties of such a basis are gathered, for example, in Chavel \cite{Chavel1984}, Chapter I. Let $\lambda_k\ge 0$ be the corresponding eigenvalues: 
	\[ -\Laplacian_{\partial D} \phi_{k} = \lambda_k \phi_k.\] (By convention these are assumed to have been sorted in increasing order.)
	\begin{definition*}[$H^r(\partial D)$]
		For $r\geq 0$, we define 
		\[ H^r(\partial D)= \braces{ f \in L^2(\partial D) \text{ s.t. } \sum_{k=0}^\infty (1+\lambda_k)^r \abs{\ip{f,\phi_k}_{L^2(\partial D)}}^2=: \norm{f}_{H^r(\partial D)}^2 <\infty},\] where the space $L^2(\partial D)$ is defined relative to the surface element on $\partial D$. For $r<0$, we define $H^r(\partial D)$ as the completion of $L^2(\partial D)$ with respect to the norm $\norm{\cdot}_{H^r(\partial D)}$ (defined as in the above display).
	\end{definition*}
	\begin{remarks*}
		\begin{enumerate}[i.]
			\item It is immediate that $\braces{\phi_k:k\in\NN\cup\braces{0})}$ is an orthogonal spanning set of $H^r(\partial D)$, and that setting $\phi_k^{(r)}=(1+\lambda_k)^{-r/2}\phi_k$ yields an orthonormal basis of $H^r(\partial D)$.
			\item This definition of $H^r(\partial D)$ coincides with other possible definitions. For example, for $r=1$ the calculation \[ \int_{\partial D} \Grad \phi_k \cdot \overline{\Grad \phi_l} = -\int_{\partial D} \phi_k \Laplacian_{\partial D} \phi_l = \lambda_l \int_{\partial D} \phi_k\phi_l=\lambda_l \delta_{kl},\] derived via the divergence theorem on a closed manifold (e.g.\ see \cite{Chavel1984} eq (35); note that the manifold $\partial D$ is compact) implies that our definition of $\norm{\cdot}_{H^1(\partial D)}$ is equivalent to the standard definition $\norm{f}_{H^1(\partial D)}^2=\norm{f}_{L^2(\partial D)}^2+\norm{\Grad f}_{L^2(\partial D)}^2,$ and inductively the same is true for $H^r(\partial D)$, $r\in\NN$.
			
			For the equivalence of this definition with some other definitions of negative or non-integer Sobolev spaces, see \cite{Lions1972} Chapter I Section 7.3 (p34-37). In particular note that $H^{-s}(\partial D)$ is the topological dual space of $H^s(\partial D)$ for any $s\in\RR$. 
			
			\item
			Note that $\phi_0$ is a constant function, hence the $H^r(\partial D)/\CC$ norm, defined by $\norm{\sqbrackets{f}}_{H^r(\partial D)/\CC}=\inf_{z\in\CC} \norm{f-z}$ for $\sqbrackets{f}$ the equivalence class over $\CC$ of a function $f\in H^r(\partial D)$, can also be characterised as \begin{equation} \label{eqn:QuotientNormCharacterisation} \norm{\sqbrackets{f}}_{H^r(\partial D)/\CC}^2=\sum_{k=1}^\infty (1+\lambda_k)^r \abs{\ip{f,\phi_k}_{L^2(\partial D)}}^2.\end{equation} Recall also $H^s_\diamond(\partial D)= \braces{g \in H^{s}(\partial D) : \ip{g,1}_{L^2(\partial D)}=0}$ (see \cref{eqn:DefintionOfHsDiamond}). Note that each $\sqbrackets{f}\in H^s(\partial D)/\CC$ has a representative $g\in H^s_\diamond(\partial D)$, and $\norm{f}_{H^s(\partial D)/\CC}=\norm{g}_{H^s(\partial D)}$. 
			We thus use the norm \cref{eqn:QuotientNormCharacterisation} on spaces $H^s(\partial D)/\CC$ and on $H^s_\diamond(\partial D)$ without further mention. We also typically write $f$ for the equivalence class $\sqbrackets{f}$ and only comment further on this where necessary.
		\end{enumerate}
	\end{remarks*}
	
	This `spectral' definition of $H^r(\partial D)$ is useful particularly because Weyl's law explicitly specifies the scaling of $\lambda_k$ as $k\to \infty$.
	
	\begin{lemma}[Weyl's law on a compact closed manifold, e.g.\ \cite{Chavel1984} eq.(49)] \label{lem:WeylsLaw}
		Suppose $M$ is a closed compact manifold of dimension $d$. Then
		\[N(\lambda)=(2\pi)^{-d}\lambda^{d/2}\omega_d\operatorname{Vol}(M)+o(\lambda^{d/2}),\] where $N(\lambda)$ is the number of eigenvalues (counted with multiplicity) no bigger than $\lambda$ and $\omega_d$ is the volume of a unit disc in $\RR^d$. 
	\end{lemma}
	\begin{corollary}\label{cor:ScalingOfEigValues}
		The eigenvalues of the Laplace--Beltrami operator $\Laplacian_{\partial D}$ satisfy $C_1 k^{2/(d-1)}\leq \lambda_k\leq C_2 k^{2/(d-1)}$ for constants $C_1,C_2$ depending only on $D$. Hence, the eigenfunctions satisfy 
		\begin{equation}\label{eqn:NormPhiKWeyls}
		C_3 (1+k^{\frac{1}{d-1}})^{s-r} \leq \norm{\phi_k^{(r)}}_{H^s(\partial D)}\leq C_4 (1+k^{\frac{1}{d-1}})^{s-r}, ~s,r \in \mathbb R,
		\end{equation} for constants $C_3$ and $C_4$ depending only on $\partial D$ and on the difference $s-r$. For $k>0$ the same expression holds with the quotient norm $\norm{\phi_k^{(r)}}_{H^s(\partial D)/\CC}$ in place of $\norm{\phi_k^{(r)}}_{H^s(\partial D)}$.
	\end{corollary}
	\begin{proof}
		We apply Weyl's law on the manifold $\partial D$, which has dimension $d-1$.
		Writing $N(\lambda^{-})$ for $\lim_{x\uparrow \lambda} N(x)$ and $N(\lambda^{+})$ for $\lim_{x\downarrow \lambda} N(x)$, we thus have
		\[ N(\lambda_k^{-})\leq k \leq N(\lambda_k^{+}).\]
		It follows that $C \lambda_k^{(d-1)/2}+o(\lambda_k^{(d-1)/2}) \leq k \leq C\lambda_k^{(d-1)/2}+o(\lambda_k^{(d-1)/2})$ for the constant $C=C(D)=(2\pi)^{-(d-1)}\omega_{d-1}\operatorname{Area}(\partial D)$ and hence we deduce the scaling of the eigenvalues. Then \cref{eqn:NormPhiKWeyls} follows from the definition of $H^r(\partial D)$ and the remarks thereafter. \qedhere
	\end{proof}
	
	\section{Comparison results for Hilbert--Schmidt operators}\label{sec:ComparisonResultsForHilbertSchmidtOperators}
	
	For separable Hilbert spaces $A$ and $B$, we use the notation $\Ll(A,B)$ for the space of bounded linear maps $A\to B$ equipped with the operator norm $\norm{\cdot}_{A\to B}$, and $\Ll_2(A,B)$ for the space of Hilbert--Schmidt operators $A\to B$ equipped with inner product $\ip{\cdot,\cdot}_{\Ll_2(A,B)}$ (e.g.\ see \cite[Chapter 12]{Aubin2011}). Define the orthonormal basis $(b_{jk}^{(p,q)})$ of $\Ll_2(H^p,H^q)$ by \[{b_{jk}^{(p,q)}(f)=\phi_j^{(p)}\otimes \phi_k^{(q)}(f)=\ip{f ,\phi_j^{(p)}}_{H^p}\phi_k^{(q)}, \quad j,k\in \NN,}\] (in this \namecref{sec:ComparisonResultsForHilbertSchmidtOperators} we omit explicit reference to the domain, writing $H^p$ for either $H^p(\partial D)/\CC$ or $H^p_\diamond(\partial D)$; as in remark (iii) in \cref{sec:LaplaceBeltramiEigenfunctionsAndSobolevSpaces}, both spaces can be identified with $\overline{\Span}\braces{\phi_k^{(p)} : k\geq 1}$, hence the omission should not cause confusion). The compatibility between our bases of $H^\alpha(\partial D)$ for different $\alpha\in\RR$ means that the subspaces spanned by $(b_{jk}^{(p,q)})_{j\leq J, k\leq K}$ coincide for all $p$ and $q$, and the $\Ll_2(H^p,H^q)$-orthogonal projections onto this subspace coincide with $\pi_{JK}$ (as defined in \cref{eqn:def:PiJK}). \Cref{cor:ScalingOfEigValues} implies the following lemmas controlling Hilbert--Schmidt norms for different domains and codomains in terms of each other, and in terms of operator norms.

	\begin{lemma} 
		\label{lem:EquivalenceOfProjectedNorms} Let $p,q,r,s\in\RR$ and let $T\in \Span\braces{b_{jk}^{(p,q)}:1\leq j\leq J,~1\leq k\leq K}$. Then there is a constant $C$ depending only on $D$ and on the differences $r-p,$ $s-q$ such that \[ \norm{T}_{\Ll_2(H^r,H^s)} \leq C (1+J^{1/(d-1)})^{(p-r)_+} (1+K^{1/(d-1)})^{(s-q)_+} \norm{T}_{\Ll_2(H^p,H^q)},\] where $x_+=\max(x,0)$ for $x\in\RR$.
	\end{lemma}
	
	\begin{proof}
		The coefficients $a_{jk}^{(r,s)}$ of $T$ with respect to the basis $(b_{jk}^{(r,s)})$ are given by  \[a_{jk}^{(r,s)}=\ip{T,b_{jk}^{(r,s)}}_{\Ll_2(H^r,H^s)}\equiv\sum_{p} \ip{T\phi_p^{(r)},b_{jk}^{(r,s)}(\phi_p^{(r)})}_{H^s}=\ip{T\phi_j^{(r)},\phi_k^{(s)}}_{H^s}\]  and we see from \cref{cor:ScalingOfEigValues} that \begin{equation}\label{eqn:ajkrs} \abs{a_{jk}^{(r,s)}} \leq C(1+j^{1/(d-1)})^{p-r}(1+k^{1/(d-1)})^{s-q}\abs{a_{jk}^{(p,q)}}\end{equation} for a constant $C$ depending only on $D$ and the differences $r-p,s-q$.
		
		Upper bounding $(1+j^{1/(d-1)})^{(p-r)}\leq (1+J^{1/(d-1)})^{(p-r)_+}$ for $j\leq J$, and similarly for $k$, we find that
		\begin{equation*}\norm{T}_{\Ll_2(H^r,H^s)}^2= \sum_{j\leq J,k\leq K} \abs{a_{jk}^{(r,s)}}^2 \leq C(1+J^{1/(d-1)})^{2(p-r)_+} (1+K^{1/(d-1)})^{2(s-q)_+} \norm{T}_{\Ll_2(H^p,H^q)}^2,\end{equation*}
		hence the result. \qedhere
	\end{proof}
	
	\begin{lemma} \label{lem:ProjectionsApproximateHSOperatorsWell} For $p,q,r,s\in \RR$ satisfying $p\leq r$ and $q\geq s$, let $T\in \Ll(H^{p-(d-1)},H^q)$. Then we have $T \in \Ll_2(H^r,H^s)$ with, for constant $C$ depending only on $D$ and the differences $r-p$, $q-s$,
		\begin{equation*}\label{eqn:OperatorNormControlsHSnorm(SobolevEmbedding?)}\norm{T}_{\Ll_2(H^r,H^s)} \leq C\norm{T}_{H^{p-(d-1)}\to H^{q}}, \end{equation*} 
		and moreover
		\begin{equation*}\label{eqn:ApproximationByProjectionBoundaryTerms}\norm{T-\pi_{JK}T}_{\Ll_2(H^r,H^s)} \leq C\norm{T}_{H^{p-(d-1)}\to H^{q}} \max\brackets[\big]{(1+J^{1/(d-1)})^{p-r},(1+K^{1/(d-1)})^{s-q}}. \end{equation*} 
	\end{lemma}
	
	\begin{proof}
		Firstly, as a consequence of \cref{cor:ScalingOfEigValues}, we have, for a constant $C=C(D)$, \begin{equation} \label{EIT-eqn:LambdaEjrBound}\norm{T \phi_j^{(p)}}_{H^q}^2 \leq \norm{T}_{H^{p-(d-1)}\to H^q}^2 \norm{\phi_j^{(p)}}_{H^{p-(d-1)}}^2\leq C \norm{T}_{H^{p-(d-1)}\to H^q}^2 (1+j^{1/(d-1)})^{-2(d-1)},\end{equation} 
		which is summable over $j$, hence by definition of the space of Hilbert--Schmidt operators, $T\in \Ll_2(H^p,H^q)$ and $\norm{T}_{\Ll_2(H^p,H^q)}\leq C'\norm{T}_{H^{p-(d-1)}\to H^q}$. That $T$ lies in $\Ll_2(H^r,H^s)$ and satisfies the the specified bound follows by monotonicity of $H^\alpha$ norms. 
		
		Since the $\Ll_2(H^r,H^s)$-orthogonal projection maps coincide for all $r$ and $s$, next defining $a_{jk}^{(r,s)}$ as in the previous proof, we have from \cref{eqn:ajkrs} that for a constant $C$, \begin{align*}\norm{T-\pi_{JK}T}_{\Ll_2(H^r,H^s)}^2 &=\sum_{j>J \text{ or } k>K} \abs{a_{jk}^{(r,s)}}^2  \\ &\leq C \sum_{j>J \text{ or } k>K} (1+j^{1/(d-1)})^{2(p-r)}(1+k^{1/(d-1)})^{2(s-q)}\abs{a_{jk}^{(p,q)}}^2 
		\end{align*} 
		
		Since $p\leq r$ and $q\geq s$, we see that
		\begin{align*}\sum_{j>J}\sum_k (1+j^{1/(d-1)})^{2(p-r)}(1+k^{1/(d-1)})^{2(s-q)}\abs{a_{jk}^{(p,q)}}^2 &\leq (1+J^{1/(d-1)})^{2(p-r)} \sum_{j>J}\sum_k \abs{a_{jk}^{(p,q)}}^2 \\ &\leq (1+J^{1/(d-1)})^{2(p-r)}\norm{T}_{\Ll_2(H^p,H^q)}^2. \end{align*}
		Arguing similarly for the sum over all $j$ and over $k>K$, we deduce that 
		\[\norm{T-\pi_{JK}T}_{\Ll_2(H^r,H^s)}^2 \leq 2C\norm{T}_{\Ll_2(H^p,H^q)}^2 \max\brackets[\big]{ (1+J^{1/(d-1)})^{2(p-r)},  (1+K^{1/(d-1)})^{2(s-q)}}. \]
		The result follows.\qedhere
	\end{proof}

	\section{Mapping properties of $\Lambda_\gamma$ and $\tilde{\Lambda}_\gamma$}\label{sec:MappingPropertiesOfLambda_gamma}
	In this \namecref{sec:MappingPropertiesOfLambda_gamma} we prove the following mapping properties of $\Lambda_\gamma$ and $\tilde{\Lambda}_\gamma$ which were used throughout the main body of the paper. The proofs rely on basic theory for elliptic PDEs, which can be found in the many books on the subject (we shall refer mostly to \cite{Lions1972}).
	\begin{lemma}
		\label{lem:DirichletProblemHasUniqueSolution}
		Let $s\in\RR$, let $m\in (0,1)$ and let $D'$ be a domain compactly contained in $D$. For $\gamma\in \Gamma_{m,D'}$ and $f\in H^{s+1}(\partial D)/\CC$, there is a unique weak solution $u_{\gamma,f}\in \Hh_s=\brackets[\big]{H^{\min\braces{1,s+3/2}}(D)\cap H^1_\loc (D)}/\CC$ to the Dirichlet problem \cref{eqn:DirichletProblem}. Moreover, if $u_{1,f}$ is the unique solution when $\gamma=1$, then for any other $\gamma\in\Gamma_{m,D'}$ bounded by $M$ on $D$, $u_{\gamma,f}-u_{1,f}$ lies in $H^1_0(D)/\CC$ and satisfies the estimate
		\begin{equation}\label{eqn:BoundOnNormW}
		\norm{u_{\gamma,f}-u_{1,f}}_{H^1(D)/\CC}\leq C{\tfrac{M}{m}}\norm{f}_{H^{s+1}(\partial D)/\CC},
		\end{equation}
		for some constant $C=C(D,D',s)$.
	\end{lemma}
	\cref{lem:DirichletProblemHasUniqueSolution} is neither novel nor necessarily sharp, but we require explicit bounds controlling how the constants depend on $\gamma$ for our proofs, which is why the result is given here.

	\begin{lemma}\label{lem:LambdaGammaTakesOneDerivative}\label{lem:tildeLambdaInfinitelySmoothing}
		For some $m\in (0,1)$ and some domain $D'$ compactly contained in $D$, let $\gamma\in\Gamma_{m,D'}$. For each $s\in\RR$, $\Lambda_\gamma$ is a continuous linear map from $H^{s+1}(\partial D)/\CC$ to $H^s_\diamond (\partial D)$, and it is continuously invertible. For each $s,t\in\RR$, the shifted operator $\tilde{\Lambda}_\gamma=\Lambda_\gamma-\Lambda_1$ is a continuous map from $H^s(\partial D)/\CC$ to $H^t_\diamond(\partial D)$.
		
		Moreover, if $\gamma$ also satisfies the bound $\norm{\gamma}_{\infty}\leq M$, then we have the explicit bounds  \begin{align}\label{eqn:NormLambdaGammaBound} \norm{\Lambda_\gamma}_{H^{s+1}(\partial D)/\CC\to H^s(\partial D)}\leq C_1 \frac{M}{m}, \\ \label{eqn:NormTildeLambdaGammaBound} \norm{\tilde{\Lambda}_\gamma}_{H^s(\partial D)/\CC \to H^t(\partial D)}\leq C_2 {\frac{M}{m}},\end{align} for constants $C_1=C(D,D',s)$ and $C_2=C_2(D,D',s,t)$.
	\end{lemma}
	
	Given \cref{lem:tildeLambdaInfinitelySmoothing}, the following is an immediate consequence of \cref{lem:ProjectionsApproximateHSOperatorsWell} (recall $\HH_r$ is defined in \cref{eqn:HHrDefinition}, and note that the proofs of the previous lemmas imply that $\tilde \Lambda_\gamma$ maps real-valued functions to real-valued valued functions). 
	\begin{lemma}\label{lem:tildeLambdaIsHilbertSchmidt} For any $r\in\RR$ and any $\gamma\in\Gamma_{m,D'}$, $\tilde{\Lambda}_\gamma\in\HH_r$.
	\end{lemma}

	A key to proving \cref{lem:tildeLambdaInfinitelySmoothing,lem:DirichletProblemHasUniqueSolution} is the following basic fact about harmonic functions. For convenience of the reader, we include a proof (following Lemma A.1 in \cite{Hanke2011}). Recall that a function $h$ is \textit{harmonic} on some domain $\Omega$ if $\Delta h =0$ on $\Omega$, where $\Delta$ denotes the Laplacian. Note that as we have assumed $\gamma=1$ on a neighbourhood $D\setminus D'$ of $\partial D$, our solutions $u_{\gamma,f}$ (and in particular, the \emph{differences} $u_{\gamma,f}-u_{\gamma_0,f}$ between solutions) are harmonic on this neighbourhood. 
	
	\begin{lemma}[Interior smoothness of harmonic functions] \label{lem:InteriorSmoothnessHarmonicFunctions} Let $U_0, U$ be  bounded domains such that $\bar{U}\subseteq U_0$. Then for any $s,t\in\RR$, there is a constant $C=C(s,t,U,U_0)$ such that for any harmonic function $v\in H^s(U)$, \[\norm{v}_{H^s(U)/\CC}\leq C\norm{v}_{H^t(U_0)/\CC}.\]
	\end{lemma}
	\begin{proof}
		By monotonicity of $H^t$ norms it suffices to prove the result for $s=t+k$ for $k\in\NN$. Let $v\in H^t(U_0)$ represent the equivalence class and choose a domain $U_1$ such that $\bar{U}\subseteq U_1\subseteq \bar{U}_1\subseteq U_0$. Let $\phi$ be a smooth cutoff function, identically one on $U_1$ and compactly supported in $U_0$. For $z\in\CC$ we observe that $\tilde{v}:=(v-z)\phi$ satisfies 
		\begin{align*}
		\begin{split}
		\Laplacian \tilde{v}&=F \quad \text{ in } U_0 \\
		\tilde{v} &=0 \quad \text { on } \partial U_0,
		\end{split}
		\end{align*}
		where $F=2\Grad\phi\cdot \Grad v +(v-z)\Laplacian \phi$. 
		Then 
		\[\norm{v}_{H^{t+1}(U_1)/\CC} \leq \norm{v-z}_{H^{t+1}(U_1)} \leq \norm{\tilde{v}}_{H^{t+1}(U_0)}\leq C\norm{F}_{H^{t-1}(U_0)},\]
		by standard elliptic boundary value regularity results (e.g.\ \cite{Lions1972} Chapter II Remark 7.2 on page 188, with $N=\{0\}$ there as we are considering the standard Laplace equation; in the case $t<1$ Remark 7.2 gives the result with a different norm in place of the $H^{t-1}(U_0)$ norm on the right, but the two norms are equivalent when restricted to functions of compact support in $U_0$). %$\Xi(\Omega)$ is defined in defn 6.1
		Note \[\norm{F}_{H^{t-1}(U_0)}\leq C(\phi)(\norm{v}_{H^t(U_0)/\CC} +\norm{v-z}_{H^{t-1}(U_0)}),\]
		and optimising across $z\in \CC$ yields
		\begin{equation}\label{eqn:HarmonicFunctions1Derivative} \norm{v}_{H^{t+1}(U_1)/\CC}\leq C\norm{v}_{H^t(U_0)/\CC}.\end{equation} Finally, we choose a finite sequence of domains $(U_j)_{1\leq j \leq k}$ such that $U_k=U$ and, for $1\leq j \leq k$, $\bar{U}_j\subset U_{j-1}$; applying \cref{eqn:HarmonicFunctions1Derivative} successively on each pair $(U_j,U_{j-1})$, we deduce the result. \qedhere
	\end{proof}
	
	\begin{proof}[Proof of \cref{lem:DirichletProblemHasUniqueSolution}]
		We adapt the proof of Theorem A.2 from Hanke et al.\ \cite{Hanke2011} to the Dirichlet setting here. From standard theory for the Laplacian, for $\gamma=1$ and $f\in H^{s+1}(\partial D)/\CC$ there exists a solution $u_{1,f}\in H^{s+3/2}(D)/\CC$ to the Dirichlet problem \cref{eqn:DirichletProblem}, and this solution satisfies \begin{equation}
		\label{eqn:BoundOnNormUWhenGamma=1}
		\norm{u_{1,f}}_{H^{s+3/2}(D)/\CC} \leq C \norm{f}_{H^{s+1}(\partial D)/\CC}
		\end{equation}
		for a constant $C=C(D,s)$ (e.g.\ see \cite{Lions1972} Chapter II, Remark 7.2 on page 188 with $N=\braces{0}$). %<- if $w\in N= \braces{\phi \in C^\infty_c (\bar{D})=C^\infty(\bar{D}) : \Laplacian \phi=0, B_j\phi=\phi|_{\partial D}=0}$ then $\int_D \abs{\Grad w}^2 \dx =\int_{\partial D} w \frac{\partial w}{\partial \nu} \dx - \int_D w\Laplacian w \dx =0$ May have $N=\CC$ for Neumann conditions (showing continuity of $\Lambda_\gamma^{-1}$). This isn't a problem since we quotient. 
		Also note that, as a harmonic function, $u_{1,f}\in H^1_\loc(D)/\CC$ by \cref{lem:InteriorSmoothnessHarmonicFunctions}.
		
		Define the sesquilinear operator $B_\gamma$ and the conjugate linear operator $A$, \[B_\gamma(w,v)=\int_D \gamma \Grad w \cdot \Grad \overline{v} , \quad \text{and} \quad A(v)=-\int_D \gamma \Grad u_{1,f}\cdot \Grad \overline{v} ,\] and consider the equation
		\begin{equation}\label{eqn:WeakFormulationForW}
		B_\gamma(w,v)=A(v) \quad \forall v\in H_0^1(D).
		\end{equation}
		Observe that if $w\in H_0^1(D)$ solves \cref{eqn:WeakFormulationForW}, then $u_{\gamma,f}=w+u_{1,f}$ solves the weak Dirichlet problem \cref{eqn:DirichletProblemWeakFormulation}; note that $u_{\gamma,f}$ so defined lies in $\Hh_s$. We will use Lax--Milgram theory to show the existence and uniqueness of such a $w$. 
		
		By definition any $\gamma\in\Gamma_{m,D'}$ is bounded; consistenly with the second half of the \namecref{lem:DirichletProblemHasUniqueSolution} let $M$ be an upper bound for $\norm{\gamma}_{\infty}$. Let $v,w\in H_0^1(D)$. By the Cauchy--Schwarz inequality, $\abs{B_\gamma(w,v)}\leq M\norm{w}_{H^1(D)}\norm{v}_{H^1(D)}$, and we also note $B_\gamma(v,v)\geq m \norm{\Grad v}_{L^2(D)}^2\geq cm \norm{v}_{H^1(D)}^2$, where the latter inequality, with constant $c=c(D)$, is the Poincar{\'e} inequality (e.g., Corollary 6.31 in \cite{AF03} applied to $v\in H^1_0(D)$). In other words, $B_\gamma$ is bounded and coercive. Next, since $\gamma=1$ on $D\setminus D'$, for $v \in H^1_0(D)$, an application of the divergence theorem yields
		\begin{align*}-\int_{D\setminus D'} \gamma \Grad u_{1,f} \cdot \Grad \overline{v} &= -\int_D \Grad u_{1,f} \cdot \Grad \overline{v} + \int_{D'} \Grad u_{1,f}\cdot \Grad \overline{v} \\
		&= \int_D \overline{v} \Delta u_{1,f} - \int_{\partial D} \overline{v} \frac{\partial u_{1,f}}{\partial \nu} +  \int_{D'} \Grad u_{1,f} \cdot \Grad \overline{v}=  \int_{D'} \Grad u_{1,f} \cdot \Grad \overline{v}\end{align*}
		It follows, since $\norm{1-\gamma}_\infty \leq \norm{\gamma}_\infty\leq M$, that \begin{align*} \abs{A(v)}&=\abs[\Big]{\int_{D'}-\gamma \Grad u_{1,f}\cdot \Grad{\overline{v}} - \int_{D\setminus D'}\gamma \Grad u_{1,f} \cdot \Grad \overline{v}} = \abs[\Big]{\int_{D'} (1-\gamma) \Grad u_{1,f}\cdot \Grad \overline{v}} \\ &\leq M\norm{v}_{H^1(D)}\norm{\Grad u_{1,f}}_{L^2(D')}.\end{align*} 
		By \cref{lem:InteriorSmoothnessHarmonicFunctions} and recalling \cref{eqn:BoundOnNormUWhenGamma=1}, there are constants $C$ and $C'$ depending only on $D,D'$ and $s$ such that \[\norm{\Grad u_{1,f}}_{L^2(D')}\leq \norm{u_{1,f}}_{H^1(D')/\CC} \leq C' \norm{u_{1,f}}_{H^{s+3/2}(D)/\CC} \leq C\norm{f}_{H^{s+1}(\partial D)/\CC}.\] 
		Thus, $\abs{A(v)}\leq CM\norm{f}_{H^{s+1}(\partial D)}\norm{v}_{H^1(D)}$. We deduce from the Lax--Milgram theorem (e.g., Theorem 1 in Section 6.2.1 of \cite{Evans1998}; note while the theorem there is stated for real scalars and bilinear maps, the same proof works for complex scalars and sesquilinear maps) that \cref{eqn:WeakFormulationForW} has a unique solution $w\in H^1_0(D)$. Moreover, the equation $B_\gamma(w,w)=A(w)$ shows that $\norm{w}_{H^1(D)}$ is upper bounded the operator norm of $A$ divided by the coercivity constant of $B_\gamma$, yielding \cref{eqn:BoundOnNormW}. 
		
		It remains to show that the (equivalence class of) function(s) $u_{\gamma,f}$ so constructed is the unique solution in $\Hh_s$ to \cref{eqn:DirichletProblemWeakFormulation}. Since we have shown uniqueness of $w$, it is enough to show that the difference $h$ between two $\Hh_s$ solutions lies in $H_0^1(D)$, as then it must be the zero function. (We are considering $h$ as a function, rather than an equivalence class of functions, which we can do by for example choosing a representative with average zero.) This is clear for $s \ge -1/2$ as then $\Hh_s=H^1(D)/\CC$, and can be shown also for $s<-1/2$ as in \cite{Hanke2011}, Theorem A.2.\qedhere
	\end{proof}
	
	\begin{proof}[Proof of \cref{lem:LambdaGammaTakesOneDerivative}]
		We first remark, for $\gamma\in\Gamma_{m,D'}$, that by the divergence theorem, \[\ip[\big]{\tfrac{\partial u}{\partial \nu},1}_{L^2(\partial D)} = \int_{\partial D} \gamma \tfrac{\partial u}{\partial \nu}  = \int_D \Div (\gamma \Grad u)=0\] for a solution $u$ to the Dirichlet problem \cref{eqn:DirichletProblem}, so that it suffices to prove \cref{eqn:NormLambdaGammaBound}, \cref{eqn:NormTildeLambdaGammaBound}, and the continuity of $\Lambda_\gamma^{-1} : H^s_\diamond(\partial D)\to H^{s+1}(\partial D)/\CC$. 
		
		We first prove \cref{eqn:NormTildeLambdaGammaBound}, by adapting the proof of Theorem A.3 from \cite{Hanke2011} and tracking the constants. Given $f\in H^s(\partial D)$ let $u_{\gamma,f}\in \Hh_{s-1}$ be the unique solution to the Dirichlet problem \cref{eqn:DirichletProblem} and let $w\in H^1_0$ be a representative of the function class $u_{\gamma,f}-u_{1,f}$.
		Choose a domain $\Omega$ with smooth boundary $\partial \Omega$, satisfying $\bar{D}'\subset \Omega\subset \bar{\Omega}\subset D.$ Choose also domains $U,U_0$ with smooth boundaries such that \[\partial \Omega \subseteq U \subset \bar{U}\subset U_0 \subset \bar{U}_0 \subset D\setminus D'.\]
		We can apply an appropriate trace theorem (\cite{Lions1972} Chapter I Theorem 9.4 for $t>0$, Chapter II Theorem 6.5 (and Remark 6.4) for $t\leq -3/2$, or Chapter II Theorem 7.3 for $-3/2<t<1/2$; note in the latter two cases we use that $w$ is harmonic on $D\setminus \bar{\Omega}$) to $w-z$ and optimise across $z\in\CC$ to see
		\begin{equation}\label{eqn:TraceTheoremBoundForpdWNu} \norm{\partial w/\partial \nu}_{H^t(\partial D)} \leq C \norm{w}_{H^{t+3/2}(D\setminus \bar{\Omega})/\CC}.\end{equation}
		Applying \cite{Lions1972} Chapter II Remark 7.2 on p.188 with $N=\braces{0}$ we see
		\[ \norm{w}_{H^{t+3/2}(D\setminus \bar{\Omega})/\CC} \leq C \brackets[\big]{\norm{\tr w}_{H^{t+1}(\partial D)/\CC}+\norm{\tr w}_{H^{t+1}(\partial \Omega)/\CC}}= C\norm{\tr w}_{H^{t+1}(\partial \Omega)/\CC}.\]
		Again applying an appropriate trace theorem, this time on a subset of $U$ bounded on one side by $\partial \Omega$, and applying \cref{lem:InteriorSmoothnessHarmonicFunctions}, we see
		\[ \norm{\tr w}_{H^{t+1}(\partial \Omega)/\CC} \leq C \norm{w}_{H^{t+3/2}(U)/\CC}\leq C'\norm{w}_{H^1(U_0)/\CC}.\]
		
		The constants in the above depend on $D$ and (via $\Omega,$ $U$ and $U_0$) on $D'$, but are otherwise independent of $\gamma$.
		Recalling \cref{eqn:BoundOnNormW}, which, because the smoothness of $f$ here is $s$, tells us $\norm{w}_{H^1(U_0)/\CC}\leq C\brackets[\big]{\frac{M}{m}}\norm{f}_{H^s(D)/\CC}$, we overall have
		$\norm{\partial w/\partial \nu}_{H^t(\partial D)}\leq C\brackets[\big]{\frac{M}{m}}\norm{f}_{H^s(\partial D)/\CC}$, so that we have proved \cref{eqn:NormTildeLambdaGammaBound}. 
		
		Now we prove \cref{eqn:NormLambdaGammaBound}; given \cref{eqn:NormTildeLambdaGammaBound}, it suffices to show $\norm{\Lambda_1}_{H^{s+1}(\partial D)/\CC\to H^s(\partial D)}\leq C\tfrac{M}{m}$ for an appropriate constant $C$. Since $u_{1,f}$ is harmonic on $D$, applying the trace theorems from \cite{Lions1972} as for \cref{eqn:TraceTheoremBoundForpdWNu} yields \[\norm{\partial u_{1,f}/\partial \nu}_{H^s(\partial D)}\leq C\norm{u_{1,f}}_{H^{s+3/2}(D)/\CC}\] for a constant $C=C(D,s)$. Then \cite{Lions1972} Chapter II Remark 7.2 with $N=\braces{0}$ yields \[\norm{u_{1,f}}_{H^{s+3/2}(D)/\CC}\leq C'\norm{f}_{H^{s+1}(\partial D)/\CC}\] for a constant $C'=C'(D,s)$, and \cref{eqn:NormLambdaGammaBound} follows.
		
		Finally we remark that the same arguments (see Theorem A.3 in \cite{Hanke2011}) applied to the inverse of $\Lambda_\gamma$, which is the \emph{Neumann-to-Dirichlet map}, show that this is continuous from $H^{s}_\diamond(\partial D)$ to $H^{s+1}(\partial D)/\CC$ as claimed. \qedhere
	\end{proof}

	\section{Asymptotic comparison of noise models}\label{sec:EquivalenceOfNoiseModels}
	In this \namecref{sec:EquivalenceOfNoiseModels} we formulate and prove the asymptotic comparison results discussed informally in \cref{sec:NoiseModel}. 
	\subsection{A brief overview of the Le~Cam distance}\label{sec:ABriefOverviewOfTheLeCamDistance} We first define the Le~Cam deficiency and the Le~Cam distance, providing the sense in which the models will be shown to be asymptotically related to each other. The concepts throughout this \namecref{sec:ABriefOverviewOfTheLeCamDistance} are drawn from Le~Cam's 1986 monograph \cite{LeCam1986}. We refer to the expository paper of Mariucci \cite{Mariucci2016} for a gentler introduction to the area.
	
	\begin{definitions*}
		\begin{description}
			\item[Statistical experiment]
			A \emph{statistical experiment}, or just \emph{experiment}, is the triple
			$(\Xx,\Ff,\braces{P_\theta}_{\theta\in\Theta})$, where for each $\theta$ in the parameter space $\Theta$, $P_\theta$ is a probability measure on the measurable space $(\Xx,\Ff)$.
			\item[Markov Kernel]
			For measurable spaces $(\Xx_i,\Ff_i)$, $i=1,2$, a \emph{Markov kernel with source $(\Xx_1,\Ff_1)$ and target $(\Xx_2,\Ff_2)$} is a map $T : \Xx_1\times \Ff_2 \to [0,1]$ such that $T(x,\cdot)$ is a probability measure for each $x\in\Xx_1$, and $T(\cdot,A)$ is measurable for each $A\in\Ff_2$.
			
			Given a (deterministic) measurable function $F: \Xx_1\to \Xx_2$ we will denote by $T_F$ the Markov kernel \begin{equation}\label{eqn:DeterministicMarkovKernelDefinition} T_F(x,A)=\II\braces{F(x)\in A}.\end{equation}
			\item[Le~Cam deficiency]
			The \emph{Le~Cam deficiency} between experiments $\Ee_1$ and $\Ee_2$, where $\Ee_i=(\Xx_i,\Ff_i,\braces{P_{i,\theta}}_{\theta\in\Theta})$ for $i=1,2$, for a common parameter space $\Theta$, is 
			\[ \delta(\Ee_1,\Ee_2)= \inf_{T} \sup_{\theta\in\Theta} \norm{TP_{1,\theta}-P_{2,\theta}}_{\textnormal{TV}},\] 
			where the infimum is over all Markov kernels with source $(\Xx_1,\Ff_1)$ and target $(\Xx_2,\Ff_2)$. The measure $TP_{1,\theta}$ is defined as 
			\[ TP_{1,\theta} (A)= \int_{\Xx_1} T(x,A) \dif P_{1,\theta}(x), \] and $\norm{\cdot}_{\textnormal{TV}}$ denotes the total variation norm on signed measures,
			\[ \norm{\nu}_{\textnormal{TV}}=\sup_A \abs{\nu(A)}.\]
			
			The Le~Cam deficiency satisfies the triangle inequality, but is not symmetric.
			\item[ Le~Cam distance] The \emph{Le~Cam distance} between experiments $\Ee_1$ and $\Ee_2$ on a common parameter space $\Theta$ is 
			\[ \Delta(\Ee_1,\Ee_2)=\max \brackets{\delta(\Ee_1,\Ee_2),\delta(\Ee_2,\Ee_1)}.\]
			
			If we identify experiments whose Le~Cam distance is zero, this defines a proper metric.
		\end{description}
	\end{definitions*}
\begin{remark*}
	Given any action set $A$, any loss function $L:\Theta\times A\to [0,1]$, and any decision rule $\rho_2: \Xx_2 \to A$, there exists a (possibly randomised) decision rule $\rho_1: \Xx_1 \to A$ such that, denoting the risk functions by $R_j(\rho_j,\theta)=E_{X\sim P_{j,\theta}} L(\theta,\rho_j(X)),$ $j=1,2$, we have (see \cite{Mariucci2016} Theorem 2.7)
	\begin{equation}\label{eqn:RiskDeficiencyTransport} R_1(\rho_1,\theta)\leq R_2(\rho_2,\theta)+\delta(\Ee_1,\Ee_2), \quad \forall \theta\in\Theta.\end{equation}We may further take the supremum over $\theta$ and then the infimum over $\rho_2$ (note that the infimum over decision rules $\rho_1$ associated as above to decision rules $\rho_2$ upper bounds the infimum over all decision rules $\rho_1$) %$\inf_{\rho_1}(R_1(\rho_1,\theta)) \leq \inf\braces{R_1(\rho,\theta): \rho\text{ is the rule associated to $\rho_2$ for some $\rho_2$}}$)
	to see that the minimax risks $R_{j,\textnormal{minimax}}=\inf_{\rho_j}\sup_\theta R_j(\rho_j,\theta)$, $j=1,2$ satisfy 
	\begin{equation}\label{eqn:MinimaxRiskDeficiencyTransport}R_{1,\textnormal{minimax}}\leq R_{2,\textnormal{minimax}}+\delta(\Ee_1,\Ee_2).\end{equation} These equations capture the intuitive definition that the Le~Cam deficiency is the worst-case error we incur when reconstructing a decision rule in $\Ee_2$ using data from $\Ee_1$ (see also \cref{thm:MinimaxityInAllNoiseModels} where this intuition is made concrete for the models considered here).
\end{remark*}
	
	We gather the key tools we will use to control Le~Cam deficiencies in the following \namecref{lem:LeCamDistanceLemmasCollected}. Recall that $K(p,q)$ denotes the Kullback--Leibler divergence between distributions with densities $p$ and $q$; in an abuse of notation we will in this section also write $K(P,Q)$ for the Kullback--Leibler divergence between distributions $P$ and $Q$. 
	\begin{lemma} \label{lem:LeCamDistanceLemmasCollected}
		
		Let $\Ee_1$ and $\Ee_2$ be experiments with a common parameter set $\Theta$: write $\Ee_j=(\Xx_j,\Ff_j,\braces{P_{j,\theta}}_{\theta\in\Theta})$.
		\begin{lemenum} 
			\item \label{lem:LeCamDistanceControlledByTV}
			Suppose further that the experiments are defined on a common probability space, i.e.\ that $\Xx_1=\Xx_2$ and $\Ff_1=\Ff_2$. Then
			\begin{equation}\label{eqn:LeCamDistanceControlledByTV}
			\Delta(\Ee_1,\Ee_2)\leq \sup_{\theta\in\Theta} \norm{P_{1,\theta}-P_{2,\theta}}_{\textnormal{TV}}\leq \sup_{\theta\in\Theta}\sqrt{K(P_{1,\theta},P_{2,\theta})/2}.
			\end{equation}
			\item \label{lem:LeCamDeficiencyControlledByImageMeasureTV}
			Let $F: \Xx_1\to \Xx_2$ 
			be any (deterministic) measurable map. Then
			\begin{equation}\label{eqn:LeCamDeficiencyControlledByImageMeasureTV}
			\delta(\Ee_1,\Ee_2)\leq \sup_{\theta\in\Theta} \norm{P_{1,\theta}\circ F^{-1}-P_{2,\theta}}_{\textnormal{TV}}.
			\end{equation} 
			\item \label{lem:LeCamDistanceSufficientStatistics}
			Let $F:\Xx_1\to \Xx_2$ be a measurable map. Suppose that $P_{1,\theta}\circ F^{-1}=P_{2,\theta}$ for each $\theta\in\Theta$ and suppose that $F(X)$ is a sufficient statistic for $X\sim P_{1,\theta}$. Then $\Delta(\Ee_1,\Ee_2)=0.$
		\end{lemenum}
	\end{lemma}
	\begin{proof}
		\begin{enumerate}[a.]
			\item The first inequality is immediate from the definition, since the Markov kernel $T_{\operatorname{Id}}$ corresponding to the identity map $\operatorname{Id} :\Xx_1\to \Xx_2= \Xx_1$ satisfies $T_{\operatorname{Id}}P=P$ for all probability measures $P$ on $(\Xx_1,\Ff_1)$. The second inequality is Pinsker's inequality (e.g.\ Proposition 6.1.7a in \cite{Gine2016}).
			\item
			Observe that $T_F P_{1,\theta}(A) = P_{1,\theta} (F(X)\in A)=P_{1,\theta}\circ F^{-1}(A).$ The result follows.
			\item See \cite{Mariucci2016}, Property 3.12.\qedhere
		\end{enumerate}
	\end{proof}

	\subsection{Le Cam deficiency bounds for noisy Calder{\'o}n problems} \label{sec:lecamdef}
	First we recall the noise models under consideration. The following displays are labelled with the notation we will use for experiments with the specified data, in each case taking the parameter space to be $\braces{\gamma\in \Gamma_{m,D'} : \norm{\gamma}_{\infty}\leq M}$ for some constants $m\in (0,1),M>1$ and some domain $D'$ compactly contained in $D$.
	
	For noise level $\eps>0$ and some $P\in\NN$, in the `electrode' model \cref{eqn:Model:HaarMeasurements} we are given data \begin{equation*}\label{eqn:Ee0} \Ee_0: \quad  Y_{p,q}= \ip{\tilde{\Lambda}_\gamma [\psi_p],\psi_q}_{L^2(\partial D)} +\eps g_{p,q}, \quad p,q\leq P, \quad g_{p,q}\iidsim N(0,1),\end{equation*} where $\psi_p=c_p\II_{I_p}$ for some disjoint (measurable) sets $I_p=I_{p,P}\subseteq \partial D$, with constants $c_p$ chosen such that the $\psi_p$ are $L^2(\partial D)$ orthonormal. For some $\eps>0,~r\in\RR,$ and $J,K\in\NN$, in the `discrete spectral' model \cref{eqn:Model:LaplaceBeltramiMeasurements} we are given data 
	\begin{equation*}\label{eqn:Ee1} \Ee_1^{(r)}: \quad Y_{j,k}= \ip{\tilde{\Lambda}_\gamma [ \phi_j^{(r)}],\phi_k^{(0)}}_{L^2(\partial D)} +\eps g_{j,k}, \quad j\leq J,k\leq K, \quad g_{j,k}\iidsim N(0,1),\end{equation*} for a (real-valued) Laplace--Beltrami basis $(\phi_k^{(r)}:{k\in\NN})$ of $H^r(\partial D)/\CC$. For $\eps>0$ and $r\in\RR$, in the `continuous' model \cref{eqn:Model:HilbertSpaceMeasurements} (see also the equivalent \cref{eqn:ContinuousModelDistributionalSense}) we are given data \begin{equation*} \label{eqn:Ee2}
	\Ee_2^{(r)}: \quad Y = \tilde{\Lambda}_\gamma +\eps \WW, \quad \text{$\WW$ a Gaussian white noise indexed by $\HH_r$.}
	\end{equation*}
	We start with the following result which shows how to approximate the discrete spectral model using the electrode model.
	\begin{theorem}\label{thm:LeCamOneWayDeficiency}
		Suppose $\cup_{p\leq P} I_p=\partial D$ and $\operatorname{diam}(I_p)\leq (A/P)^{1/(d-1)}$ for a constant $A$ independent of $P$, where $\operatorname{diam} S= \sup_{x,y\in S} \abs{x-y}$ denotes the Euclidean diameter of a subset $S\subset \RR^d$. Then the one-way Le~Cam deficiency $\delta(\Ee_0,\Ee_1^{(0)})$ satisfies 
		\[ \delta(\Ee_0,\Ee_1^{(0)}) \leq C\brackets[\big]{\max(J,K)^{(5d-2)/(2d-2)}+\eps^{-1}\max(J,K)^{3d/(2d-2)}}P^{-1/(d-1)}, \]  for some constant $C=C(A,D',D,M,m)$, and hence vanishes asymptotically if $P$ is large enough compared to $\eps^{-1},$ $J$ and $K$.
	\end{theorem}
	\begin{remarks*}
		\begin{enumerate}[i.]
			\item The conditions on $(I_p)_{p\leq P}$ are only used to prove that we can approximate any Laplace--Beltrami eigenfunction using a linear combination of the $(\psi_p)_{p\leq P}$ with $L^2(\partial D)$-norm approximation error proportional to 			$P^{-1/(d-1)}$% in $L^2(\partial D)$-norm 
			(\cref{lem:PhiPApproximatesPhi}). If $(\psi_p)_{p\leq P}$ are such that we can approximate Laplace--Beltrami eigenfunctions at a rate $f(P)$ then we achieve the result with $f(P)$ in place of $P^{-1/(d-1)}$ .
			\item The given conditions are naturally satisfied by `evenly spaced' sets $(I_p)_{p\leq P}$ partitioning the boundary $\partial D$, with a constant $A$ depending only on the domain $D$. This can be seen by considering the covering numbers $N(\partial D,d_{\partial D},\delta)$ (the smallest number of $d_{\partial D}$ balls of radius $\delta$ needed to cover $\partial D$) for $d_{\partial D}$ the geodesic distance.
			Theorem 4.5 in Geller \& Pensenson \cite{Geller2011} applied to the current setting yields \[N(\partial D,d_{\partial D},\delta) \leq A \delta^{-(d-1)}\] for a constant $A=A(D)$, for any $\delta>0$. Taking $\delta = 2(A/P)^{1/(d-1)}$ we deduce that there exist $P$ balls of ($d_{\partial D}$)-radius $\delta/2$ covering $D$. To construct $P$ disjoint subsets of diameter at most $\delta$, we simply assign each $x\in\partial D$ to exactly one of the balls containing it (note the Euclidean diameter is upper bounded by the geodesic diameter, because any geodesic path on the surface $\partial D$ is also a path in Euclidean space). %doing this is a measurable way, eg by choosing the one with the closest centre. There could be some trouble at boundaries but not if we take the Lebesgue measurable sets in place of the Borel ones
		\end{enumerate}
	\end{remarks*}
	
%	\smallskip
%	We now turn to proving \cref{thm:LeCamOneWayDeficiency}. The idea is to approximate Laplace--Beltrami eigenfunctions via linear combinations of the indicator functions. The following lemma allows us to control the error in this approximation.

	\begin{proof}%[Proof of \cref{thm:LeCamOneWayDeficiency}]
Let $(Y_{p,q})$ be the data from experiment $\Ee_0$. Let $\phi^P_j$ denote the $L^2$-orthogonal projection of $\phi_j^{(0)}$ onto $\Span\braces{\psi_p: p\leq P}$, and write $a_{jp}=\ip{\phi_j^{(0)},\psi_p}_{L^2(\partial D)}\in\RR$, so that $\phi^P_j=\sum_{p=1}^P a_{jp} \psi_p$. % Note that $a_{jp}\in\RR$ for all $j,p$ since $\phi_j^{(0)}$ and $\psi_p$ are real-valued.
Define $F: \RR^{P\times P}\to \RR^{J\times K}$ via \[ F((u_{pq})_{p,q\leq P})_{jk}=\sum_{p,q\leq P} a_{jp} a_{kq} u_{pq}.\] For parameter set $\braces{\gamma\in \Gamma_{m,D'} : \norm{\gamma}_{\infty}\leq M}$, let $\Ee_0'$ denote the experiment with data \begin{equation}\label{eqn:Model:ProjectedFourierWithCovaryingNoise} \Ee_0': \quad Y_{j,k}'=F((Y_{p,q})_{p,q\leq P})_{jk}=\ip{\tilde{\Lambda}_\gamma \phi_j^P,\phi_k^P}_{L^2(\partial D)}+ \eps g'_{j,k}, \end{equation}
where we define $g'_{j,k}=\sum_{p,q} a_{jp}a_{kq}g_{p,q},$ and let $\Ee_1'$ denote the experiment with data \cref{eqn:Model:ProjectedFourierWithCovaryingNoise} but for i.i.d.\ Gaussian noise. By \cref{lem:LeCamDeficiencyControlledByImageMeasureTV} we see that $\delta(\Ee_0,\Ee_0')=0$, so by the triangle inequality we deduce \[\delta(\Ee_0,\Ee_1^{(0)})\leq \delta(\Ee_0,\Ee_0')+\delta(\Ee_0',\Ee_1')+\delta(\Ee_1',\Ee_1^{(0)})\leq \Delta(\Ee_0',\Ee_1')+\Delta(\Ee_1',\Ee_1^{(0)}).\] We control the terms on the right.

\paragraph{$\Delta(\Ee_0',\Ee_1')$:}
The covariance of $(g_{j,k}')$ is given by
\[\Cov(g'_{j,k},g'_{l,m})=\ip{\phi^P_j,\phi^P_l}_{L^2(\partial D)}\ip{\phi^P_k,\phi^P_m}_{L^2(\partial D)}.\] Writing \begin{equation*}\label{eqn:phiPjDecomposition} \ip{\phi^P_j,\phi^P_l}_{L^2(\partial D)} = \ip{\phi_j^{(0)},\phi_l^{(0)}}_{L^2(\partial D)}+\ip{\phi_j^{(0)},\phi^P_l-\phi_l^{(0)}}_{L^2(\partial D)}+\ip{\phi^P_j-\phi_j^{(0)},\phi^P_l}_{L^2(\partial D)},
\end{equation*} and applying the Cauchy--Schwarz inequality (note also that $\norm{\phi_l^P}_{L^2(\partial D)}\leq \norm{\phi_l^{(0)}}_{L^2(\partial D)}=1$) we see that 
\begin{equation*}\label{eqn:phiPjDeltajL2control}
\abs{\ip{\phi_j^P,\phi_l^P}_{L^2(\partial D)} - \delta_{jl}}\leq  \norm{\phi_j^P-\phi_j^{(0)}}_{L^2(\partial D)}+\norm{\phi_l^P-\phi_l^{(0)}}_{L^2(\partial D)}.
\end{equation*}
A similar decomposition further yields 
\[\abs{\ip{\phi_j^P,\phi_l^P}_{L^2(\partial D)} \ip{\phi_k^P,\phi_m^P} -\delta_{jl}\delta_{km}}\leq 4\max_{i\leq \max (J,K)} \norm{\phi_i^P-\phi_i^{(0)}}_{L^2(\partial D)}.	\]
%\Cref{lem:PhiPApproximatesPhi} tells us that $\norm{\phi_j^{(0)}-\phi^P_j}_{L^2(\partial D)}^2\leq C\max\brackets{J,K}^{(2+d)/(d-1)}P^{-2/(d-1)}$, for some constant $C$, for all $j\leq \max(J,K)$, hence
\Cref{lem:PhiPApproximatesPhi} controls this $L^2$ approximation error, yielding that for a constant $C=C(A,D)$,
\[\abs{\Cov(g_{j,k}',g_{l,m}')-\delta_{jl}\delta_{km}}\leq C 
\max\brackets{J,K}^{(1+d/2)/(d-1)}P^{-1/(d-1)}.	\]

Thus, controlling the Le~Cam distance between Gaussian experiments with equal means by $\sqrt{2}$ times the Frobenius distance between the covariance matrices (e.g.\ as in the proof of Theorem 3.1 in \cite{Reiss2008}) yields, for constants $C',c'$,
\begin{align*}\Delta(\Ee_0',\Ee_1') &\leq \sqrt{2} \brackets[\Big]{\sum_{j,l\leq J,}\sum_{k,m\leq K } (\Cov(g'_{j,k},g'_{l,m})-\delta_{jl}\delta_{km})^2}^{1/2}\leq c' JK\max\brackets{J,K}^{(1+d/2)/(d-1)} P^{-1/(d-1)} \\ &\leq C'\max\brackets{J,K}^{(5d-2)/(2d-2)}P^{-1/(d-1)}.\end{align*} 
\paragraph{$\Delta(\Ee_1',\Ee_1^{(0)})$:}
Explicitly calculating the Kullback--Leibler divergence between multivariate normals with the same covariance matrix (cf.\ the similar calculation in \cref{lem:KLdistanceIsHHdistance}) and using \cref{lem:LeCamDistanceControlledByTV} yields \[\Delta(\Ee_1',\Ee_1^{(0)})\leq \tfrac{1}{2} \eps^{-1} \times \sup_{\gamma\in\Gamma_{m,D'} : \norm{\gamma}_{\infty}\leq M} \norm[\big]{\brackets[\big]{\ip{\tilde{\Lambda}_\gamma \phi_j^{(0)},\phi_k^{(0)}}_{L^2(\partial D)}-\ip{\tilde{\Lambda}_\gamma \phi_j^P,\phi_k^P}_{L^2(\partial D)}}_{j\leq J,k\leq K}}_{\RR^{J\times K}},\] where the norm on the right is the usual Frobenius or Hilbert--Schmidt norm on the space of $J\times K$ matrices.
By \cref{lem:tildeLambdaInfinitelySmoothing}, $\norm{\tilde{\Lambda}_\gamma}_{L^2(\partial D)\to L^2(\partial D)}$ is bounded by a constant $C=C(D,D',M,m)$, hence applying also \cref{lem:PhiPApproximatesPhi} (as well as the Cauchy--Schwarz inequality) we have for a different constant $C'=C'(A,D,D',M,m)$,
\begin{align*}
\brackets{\ip{\tilde{\Lambda}_\gamma \phi_j^{(0)},\phi_k^{(0)}}_{L^2(\partial D)}-\ip{\tilde{\Lambda}_\gamma \phi_j^P,\phi_k^P}_{L^2(\partial D)}}^2 & = \brackets{\ip{\tilde{\Lambda}_\gamma (\phi_j^{(0)}-\phi_j^P),\phi_k^{(0)}}_{L^2(\partial D)} + \ip{\tilde{\Lambda}_\gamma \phi_j^P,\phi_k^{(0)}-\phi_k^P}_{L^2(\partial D)}}^2 \\
&  \leq C^2\brackets{\norm{\phi_j^{(0)}-\phi_j^P}_{L^2(\partial D)}+\norm{\phi_k^{(0)}-\phi_k^P}_{L^2(\partial D)}}^2
\\ & \leq C'\max\brackets{J,K}^{(2+d)/(d-1)} P^{-2/(d-1)}.
\end{align*}
Summing over $j$ and $k$ we deduce $\Delta(\Ee_1',\Ee_1^{(0)})\leq C'\eps^{-1}\max(J,K)^{3d/(2d-2)}P^{-1/(d-1)},$ concluding the proof. \qedhere
	\end{proof}

\begin{lemma} \label{lem:PhiPApproximatesPhi} Under the hypotheses of \cref{thm:LeCamOneWayDeficiency}, let $\phi^P_j$ denote the $L^2$-orthogonal projection of $\phi_j^{(0)}$ onto $\Span\braces{\psi_p: p\leq P}$. Then there is a constant $C$ depending only on the constant $A$ of \cref{thm:LeCamOneWayDeficiency} and on $D$ such that, for $j\leq \max(J,K)$,
	\begin{equation} \label{eqn:normPhi-PhiP} \norm{\phi_j^{(0)}-\phi^P_j}_{L^2(\partial D)}^2\leq C\max\brackets{J,K}^{(2+d)/(d-1)}P^{-2/(d-1)}. \end{equation}
	
\end{lemma}
\begin{proof}
	Since $\phi^P_j$ as the $L^2$-orthogonal projection minimises the $L^2$ distance to $\phi_j^{(0)}$ of any function in $\Span\braces{\psi_{p}: p\leq P}$, for any points $x_p\in I_p$ we see 
	\begin{align*}
	\norm{\phi_j^{(0)}-\phi^P_j}_{L^2(\partial D)}^2 &\leq \norm{\phi_j^{(0)} - \sum_{p=1}^P \phi_j^{(0)}(x_p)\II_{I_p}}_{L^2(\partial D)}^2 \\
	& \leq \max_{p\leq P}\brackets{\operatorname{diam}(I_p)^2}\sum_{p=1}^P \int_{I_p} \frac{\abs{\phi_j^{(0)}(x)-\phi_j^{(0)}(x_p)}^2}{\abs{x-x_p}^2}\dx \\
	&\leq (A/P)^{2/(d-1)}\norm{\phi_j^{(0)} }_{\textnormal{Lip}}^2\operatorname{Area}(\partial D). 
	\end{align*}
	Using a Sobolev embedding for the compact manifold $\partial D$, we may estimating the Lipschitz constant of $\phi_j^{(0)}$ by a constant times $\norm{\phi_j^{(0)}}_{H^{\kappa}(\partial D)}$ for any $\kappa>1+(d-1)/2$. In particular, taking $\kappa=1+d/2$, we see that the final expression is bounded by $C\max\brackets{J,K}^{(2+d)/(d-1)}P^{-2/(d-1)}$ for some $C=C(A,D)$ by \cref{cor:ScalingOfEigValues}. \qedhere 
\end{proof}

The following theorem shows that the `discrete spectral'  and the continuous measurement models are very close to each other.
	
	\begin{theorem}\label{thm:LeCamAsymptoticEquivalence}
		For any $r\in\RR$ and any $\nu>0$ there is a constant  $C=C(\nu,r,D,D_0,M,m)$ such that the Le~Cam distance $\Delta(\Ee_1^{(r)},\Ee_2^{(r)})$ satisfies
		\[ \Delta(\Ee_1^{(r)},\Ee_2^{(r)})\leq C\eps^{-1} \min\brackets{J,K}^{-\nu}.\]
	\end{theorem}

	\begin{proof}
		We introduce the experiments $\Ee_i^{(r)}%=(\Omega_i, \Ff_i, \braces{P^{(i)}_\theta}_{\theta\in\Gamma})
		$, $i=3,4$ with parameter space $\braces{\gamma\in \Gamma_{m,D'} : \norm{\gamma}_{\infty}\leq M}$ corresponding to observations 
		\begin{align*}
		\Ee_3^{(r)}: \quad &\brackets[\big]{\pi_{JK} \tilde{\Lambda}_\gamma +\eps\WW}(U)_{U\in\HH_r}=\brackets[\big]{\ip{\pi_{JK}\tilde{\Lambda}_\gamma,U}_{\HH_r} + \eps \sum_{j,k} g_{jk} \ip{U\phi_j^{(r)},\phi_k^{(0)}}_{L^2(\partial D)} }_{U\in \HH_r}, \\
		\Ee_4^{(r)}: \quad & \brackets[\big]{\pi_{JK} \tilde{\Lambda}_\gamma +\eps\WW}(\pi_{JK} U)_{U\in \HH_r}=\brackets[\big]{\ip{\pi_{JK}\tilde{\Lambda}_\gamma,\pi_{JK}U}_{\HH_r} + \eps \sum_{\substack{j\leq J,\\k\leq K}} g_{jk} \ip{U\phi_j^{(r)},\phi_k^{(0)}}_{L^2(\partial D)} }_{U\in \HH_r},
		\end{align*}
		where we recall the projection $\pi_{JK}$ was defined in \cref{eqn:def:PiJK}, and $g_{jk}=\WW(b_{jk}^{(r)})\iidsim N(0,1)$.
		By the triangle inequality, we decompose $\Delta(\Ee_1^{(r)},\Ee_2^{(r)})\leq \Delta(\Ee_2^{(r)}, \Ee_3^{(r)})+\Delta(\Ee_3^{(r)},\Ee_4^{(r)})+\Delta(\Ee_4^{(r)},\Ee_1^{(r)})$. We control each of the terms on the right.
		
		\paragraph{$\Delta(\Ee_2^{(r)},\Ee_3^{(r)})$:}
		Lemmas~\ref{lem:LeCamDistanceControlledByTV}, \ref{lem:ProjectionErrorDecaysAsMinJK^-nu}, and the proof of \cref{lem:KLdistanceIsHHdistance} yield
		\[ \Delta(\Ee_2^{(r)},\Ee_3^{(r)})\leq  \tfrac{1}{2}\eps^{-1} \times \sup_{\gamma\in\Gamma_{m,D'} : \norm{\gamma}_{\infty}\leq M}\norm{\tilde{\Lambda}_\gamma-\pi_{JK} \tilde{\Lambda}_\gamma}_{\HH_r}\leq C\eps^{-1} \min\brackets{J,K}^{-\nu},\] for a constant $C=C(\nu,r,D,D',M,m).$ 
		\paragraph{$\Delta(\Ee_3^{(r)},\Ee_4^{(r)})$:} We note that 
		$\brackets[\big]{\pi_{JK} \tilde{\Lambda}_\gamma +\eps\WW}(\pi_{JK} U)_{U\in \HH_r}$ is a sufficient statistic for $\brackets[\big]{\pi_{JK} \tilde{\Lambda}_\gamma +\eps\WW}(U)_{U\in \HH_r}$ by independence of $(g_{jk})_{j\leq J, k\leq K}$ from $(g_{jk} : j>J \textnormal{ or } k>K)$, so that $\Delta(\Ee_3^{(r)},\Ee_4^{(r)})=0$ by \cref{lem:LeCamDistanceSufficientStatistics}.
		\paragraph{$\Delta(\Ee_4^{(r)},\Ee_1^{(r)})$:} Using \cref{lem:LeCamDeficiencyControlledByImageMeasureTV} as in the proof of \cref{thm:LeCamOneWayDeficiency}, it is clear that the experiment $\Ee_1^{(r)}$ is equivalent to observing 
		\[ \brackets[\big]{{\textstyle\sum_{j\leq J,k\leq K}} \brackets[\big]{\ip{\tilde{\Lambda}_\gamma \phi_j^{(r)},\phi_k^{(0)}}_{L^2(\partial D)}{u}_{jk} +\eps g_{j,k}{u}_{jk}}}_{u\in \ell^2},
		\] where $\ell_2$ denotes the space of square-summable real sequences. Since $(g_{jk}):=(\WW(b_{jk}^{(r)}))\overset{d}=(g_{j,k})$ (the latter being the noise in $\Ee_1^{(r)}$) 
		and, for $U=\sum_{j,k} u_{jk}b_{jk}^{(r)}$, \[{\textstyle\sum_{j\leq J, k\leq K}} \ip{\tilde{\Lambda}_\gamma \phi_j^{(r)},\phi_k^{(0)}}_{L^2(\partial D)}{u}_{jk}= \ip{\pi_{JK} \tilde{\Lambda}_\gamma,\pi_{JK} U}_{\HH_r},\] we deduce $\Delta(\Ee_4^{(r)},\Ee_1^{(r)})=0$.\qedhere
	\end{proof}
Complementing the two previous theorems, let us remark that in fact $\delta(\Ee_2^{(r)},\Ee_1^{(r)})=0$ and similarly $\delta(\Ee_2^{(0)},\Ee_0)=0$, as is shown in the proof of the lower bound in the following theorem.

\smallskip

In our final result we prove that the rates obtained in \cref{thm:ExistenceOfEstimator,thm:MinimaxLowerBound} hold with $P_\eps^\gamma$ replaced by the law $P_{\eps,P}^\gamma$ of the data $Y=(Y_{p,q})_{p,q\leq P}$ in the electrode model \cref{eqn:Model:HaarMeasurements}, for an appropriate number of electrodes $P$ and for appropriate sets $(I_p)_{p\leq P}$. The same conclusion holds as well in the `discrete spectral' model \cref{eqn:Model:LaplaceBeltramiMeasurements} -- the proof of this fact is similar (in fact simpler) and omitted. 

\begin{theorem}\label{thm:MinimaxityInAllNoiseModels} 
Let $P_{\eps,P}^\gamma$ denote the law of the data $Y=(Y_{p,q})_{p,q\leq P}$ in the electrode model \cref{eqn:Model:HaarMeasurements}. Recall the parameter set $\Gamma_{m_0,D_0}^\alpha(M)$ defined in \cref{eqn:DefinitionOfThetaMalpha}.
	\begin{enumerate}[A.]
		\item Suppose $m_0,D_0,\alpha,M$ satisfy the conditions of \cref{thm:ExistenceOfEstimator}. Suppose that the sets $(I_p)_{p\leq P}$ satisfy the conditions of \cref{thm:LeCamOneWayDeficiency} and that $P\geq \eps^{-\mu(d-1)}$ for some $\mu>1$. Then there exists a measurable function $\hat{\gamma}'$ of the data $Y\sim P_{\eps,P}^\gamma$ such that, for $C,\delta$ the same constants as in \cref{thm:ExistenceOfEstimator},
		\begin{equation}\label{eqn:UpperBound}\sup_{\gamma\in \Gamma^\alpha_{ m_0, D_0}(M)} P_{\eps,P}^\gamma(\norm{\hat{\gamma}'-\gamma}_{\infty}> C\log(1/\eps)^{-\delta})\to 0,~\text{as } \eps\to 0.\end{equation} 
		\item  Suppose $D_0,D,m_0,\alpha,M, c, \delta'$ are as in \cref{thm:MinimaxLowerBound}.	
		Let $P\in \NN$ be arbitrary, and let $(I_p)_{p\leq P}$ be arbitrary disjoint measurable sets. %(which need not satisfy the conditions of \cref{thm:LeCamOneWayDeficiency}). 
		Then for all $\eps$ small enough,
		\begin{equation}\label{eqn:LowerBound} \inf_{\tilde{\gamma}'} \sup_{\gamma \in \Gamma^\alpha_{ m_0, D_0}(M)} P_{\eps,P}^\gamma(\norm{\tilde{\gamma}'-\gamma}_{\infty}> c\log(1/\eps)^{-\delta'})>1/4,\end{equation} where the infimum extends over all measurable functions $\tilde{\gamma}'=\tilde{\gamma}'(Y)$ of the data $Y\sim P_{\eps,P}^\gamma$. 
	\end{enumerate}
\end{theorem}

\begin{proof} 
We consider statistical experiments denoted $\Ee_0$, $\Ee_1^{(r)}$, $\Ee_2^{(r)}$ defined as at the start of \cref{sec:lecamdef}, but here for the smaller parameter space $\Gamma^\alpha_{m_0,D_0}(M)$. The definition of the Le~Cam deficiency ensures that it cannot increase upon considering a smaller parameter space, so that \cref{thm:LeCamOneWayDeficiency,thm:LeCamAsymptoticEquivalence} continue to hold for these experiments.
\paragraph{A:}
	Consider the loss function $L(\gamma,\rho)=\II\braces{\norm{\gamma-\rho}_\infty > C\log(1/\eps)^{-\delta}}.$ The associated risk $R$ of the decision rule $\hat{\gamma}$ in the continuous model \cref{eqn:Model:HilbertSpaceMeasurements} is $P_\eps^\gamma(\norm{\hat{\gamma}-\gamma}_\infty>C\log(1/\eps)^{-\delta})$, so that $\sup_\gamma R\to 0$ as $\eps\to 0$, for any $r\in\RR$, by \cref{thm:ExistenceOfEstimator}. 	
	
	In view of \cref{eqn:RiskDeficiencyTransport}, there exists a measurable function $\hat{\gamma}'$ of the data in model \cref{eqn:Model:HaarMeasurements} whose risk $R'$ satisfies $R'\leq R+\delta(\Ee_0,\Ee_2^{(0)}).$ [In fact, as the proof shows, all Markov kernels involved in bounding $\delta(\Ee_0,\Ee_2^{(0)})$ arise from deterministic maps, and hence $\hat \gamma'$ can be taken to be non-randomised.] The limit \cref{eqn:UpperBound} (which is $\sup_\gamma R' \to 0$) will follow from showing that $\delta(\Ee_0,\Ee_2^{(0)})\to 0$. The triangle inequality gives that $\delta(\Ee_0,\Ee_2^{(0)})\leq \delta(\Ee_0,\Ee_1^{(0)})+\delta(\Ee_1^{(0)},\Ee_2^{(0)})$, where we may freely choose the parameters $J,K$ of the intermediate model.	For some positive constant $\mu'<\min\braces{(2d-2)(\mu-1)/(3d),(2d-2)\mu/(5d-2)},$ take $\nu>1/\mu'$ and choose $J=K$ of order $\eps^{-\mu'}$. Then, by \cref{thm:LeCamOneWayDeficiency,thm:LeCamAsymptoticEquivalence} and the triangle inequality, we have for a constant $C$	\[\delta(\Ee_0,\Ee_2^{(0)})\leq C(\eps^{\mu-1} J^{3d/(2d-2)} +J^{(5d-2)/(2d-2)} \eps^\mu + \eps^{-1} J^{-\nu})\to 0,\] as required. 
\paragraph{B:} 	Consider the loss function $\tilde{L}(\gamma,\rho)=\II\braces{\norm{\gamma-\rho}_\infty > c\log(1/\eps)^{-\delta'}}.$ In view of \cref{thm:MinimaxLowerBound}, the associated minimax risk in the continuous model, $\tilde{R}_{\textnormal{minimax}}=\inf_{\rho}\sup_\gamma P_\eps^\gamma (\norm{\gamma-\rho}_\infty >c\log(1/\eps)^{-\delta'})$ is greater than 1/4, at least for $\eps$ small enough. Then \cref{eqn:MinimaxRiskDeficiencyTransport} implies that the minimax risk $\tilde{R}_{\textnormal{minimax}}'$ in model \cref{eqn:Model:HaarMeasurements} satisfies $\tilde{R}_{\textnormal{minimax}}'> 1/4- \delta(\Ee_2^{(r)},\Ee_0).$ The lower bound \cref{eqn:LowerBound} will follow from showing that $\delta(\Ee_2^{(0)},\Ee_0)=0$. By \cref{lem:LeCamDeficiencyControlledByImageMeasureTV}, it suffices to show that the data in model \cref{eqn:Model:HaarMeasurements} has law matching that of some subset of the data observed in model \cref{eqn:Model:HilbertSpaceMeasurements} with $r=0$. Such a subset, recalling the concrete interpretation \cref{eqn:ContinuousModelDistributionalSense} of model \cref{eqn:Model:HilbertSpaceMeasurements}, is given by $(Y(T_{pq}))_{p,q\leq P}$, where $T_{pq}=\psi_p\otimes \psi_q = \ip{\cdot, \psi_p}_{L^2(\partial D)} \psi_q$. To see this, observe that $\ip{T_{pq}(\phi_j^{(0)}),\phi_k^{(0)}}_{L^2(\partial D)} = \ip{\phi_j^{(0)},\psi_p}_{L^2(\partial D)} \ip{\psi_q,\phi_k^{(0)}}_{L^2(\partial D)}$, hence
%\begin{align*}
%	\ip{\tilde{\Lambda}_\gamma,T_{pq}}_{\HH_0}&=\sum_j \ip{\tilde{\Lambda}_\gamma \phi_j^{(0)},\ip{\phi_j^{(0)},\psi_p}_{L^2(\partial D)}\psi_q}_{L^2(\partial D)} 
%	=\sum_j \ip[\big]{\ip{\psi_p,\phi_j^{(0)}}_{L^2(\partial D)} \tilde{\Lambda}_\gamma \phi_j^{(0)},\psi_q}_{L^2(\partial D)}  \\ &= \ip{\tilde{\Lambda}_\gamma \psi_p,\psi_q}_{L^2(\partial D)}, \\
%	\ip{\WW,T_{pq}}_{\HH_0}&=\sum_{j,k} g_{jk}{\ip[\big]{\ip{\phi_j^{(0)},\psi_p}_{L^2(\partial D)}\psi_q,\phi_k^{(0)}}}_{L^2(\partial D)} = \sum_{j,k} g_{jk} \ip{\psi_p,\phi_j^{(0)}}_{L^2(\partial D)} {\ip{\psi_q,\phi_k^{(0)}}}_{L^2(\partial D)}.
%	\end{align*}
\begin{align*}
\ip{\tilde{\Lambda}_\gamma,T_{pq}}_{\HH_0}&=\sum_{j,k} \ip{\tilde{\Lambda}_\gamma \phi_j^{(0)},\phi_k^{(0)}}_{L^2(\partial D)}\ip{\phi_j^{(0)},\psi_p}_{L^2(\partial D)} \ip{\psi_q,\phi_k^{(0)}}_{L^2(\partial D)} 
= \ip{\tilde{\Lambda}_\gamma \psi_p,\psi_q}_{L^2(\partial D)}, \\
\WW(T_{pq})&=\sum_{j,k} g_{jk}\ip{\phi_j^{(0)},\psi_p}_{L^2(\partial D)}\ip{\psi_q,\phi_k^{(0)}}_{L^2(\partial D)}.
\end{align*}

	The noise variables $g'_{p,q}:=\ip{\WW,T_{pq}}_{\HH_0}$ are jointly normally distributed with mean zero and covariances \begin{equation*}\begin{split}\Cov(g'_{p,q},{g}_{l,m}')&= \sum_{j,k} \ip{\phi_j^{(0)},\psi_p}_{L^2(\partial D)}{\ip{\phi_j^{(0)},\psi_l}}_{L^2(\partial D)}{\ip{\psi_q,\phi_k^{(0)}}}_{L^2(\partial D)}\ip{\psi_m,\phi_k^{(0)}}_{L^2(\partial D)} \\ &=\ip{\psi_p,\psi_l}_{L^2(\partial D)}{\ip{\psi_q,\psi_m}}_{L^2(\partial D)}=\delta_{pl}\delta_{qm},\end{split}\end{equation*} so that indeed $(Y(T_{pq}))_{p,q\leq P}\overset{d}{=} (Y_{p,q})_{p,q\leq P}$ as claimed.\qedhere	
\end{proof}

	{\textbf{Acknowledgements.} The authors would like to thank Tapio Helin for helpful discussions at the outset of this project. We are further grateful to five anonymous referees and an associate editor for their valuable remarks and suggestions. We would particularly like to acknowledge one referee for pointing out the reference \cite{NS10} which allowed us to generalise our results to include $d=2$. RN was supported by ERC grant No.~647812. KA was supported by the UK EPSRC grant EP/L016516/1.
		
		\section{Notation index}\label{sec:Notation}
		\footnotesize
		
		\begin{description}
			\itemsep -0.1em 
			\item $D\subseteq \RR^d$, $d\geq 2$, a bounded domain, which is taken to mean a connected open set with smooth boundary $\partial D$.
			\item  $\Gamma_{m,D'}=\braces{\gamma\in C(D) : \inf_{x\in D} \gamma(x) \geq m,~\gamma=1 \text{ on } D\setminus D'}$, some $m\in(0,1)$ and some domain $D'$ compactly contained in $D$ (i.e.\ the closure $\bar{D}'$ is a subset of $D$). $C(D)$ denotes the real-valued bounded continuous functions from $D$ to $\RR$. $C_u(D)$ denotes the real-valued uniformly continuous functions on $D$.
			\item $\Gamma^\alpha_{m,D'}(M)=\braces{\gamma \in \Gamma_{m,D'} : ~\norm{\gamma}_{H^\alpha(D)}\leq M}$.
			\item $\gamma\in \Gamma_{m,D'}$ a conductivity function, $\gamma_0$ its `true' value for some statistical theorems. 
			\item $\theta_0=\Phi^{-1}\circ \gamma_0$ for $\Phi$ described in \cref{sec:PriorConstruction}. $\theta\in C_u(D)$ a function of the form $\Phi^{-1}\circ\gamma$ for $\gamma\in \Gamma_{m,D'}$ (usually denoting a generic draw from the prior $\Pi$ of \cref{eqn:tampering}).
			\item $m_0,D_0$ a lower bound and support set for the `true' $\gamma_0$; $m_1,D_1$ a lower bound and support set for any draw $\gamma=\Phi\circ\theta$ from the prior $\Pi$ of \cref{sec:PriorConstruction}.
			\item $\norm{\cdot}_{\infty}$ the usual supremum norm on $C(D)$ or $C(\RR)$.
			\item $\overline{v}$ the usual complex conjugate of a number or function $(\overline{v}(x):=\overline{v(x)})$.
			\item $u_{\gamma,f}$ the (weak) solution to the Dirichlet problem \cref{eqn:DirichletProblem} ($\Div(\gamma \Grad u)=0$ in $D$, $u=f$ on $\partial D$).
			\item $H^s$ an $L^2$-Sobolev space of complex-valued functions (carefully defined in \cref{sec:LaplaceBeltramiEigenfunctionsAndSobolevSpaces}); $H^1_0(D)$ the traceless subset of $H^1(D)$. Inner products here are linear in the first argument and conjugate-linear in the second. 
			\item $\Hh_s=\brackets[\big]{H^{\min\braces{1,s+3/2}}(D)\cap H^1_\loc (D)}/\CC$ ($H^1_\loc$ is defined in \cref{sec:LaplaceBeltramiEigenfunctionsAndSobolevSpaces}, and $/\CC$ means we identify functions $f,f+c$ which are equal up to a scalar $c\in\CC$).
			\item $H^s_\diamond(\partial D)= \braces{g \in H^{s}(\partial D) : \ip{g,1}_{L^2(\partial D)}=0}$, $L^2_\diamond(\partial D)=H^0_\diamond(\partial D)$.
			\item $(\phi_k^{(r)})_{k\in\NN\cup\braces{0}}$ an orthonormal basis of $H^r(\partial D)$ consisting of real-valued eigenfunctions of the Laplace--Beltrami operator on $\partial D$, with corresponding eigenvalues (sorted so they increase with $k$) $\lambda_k\geq 0$. More details in \cref{sec:LaplaceBeltramiEigenfunctionsAndSobolevSpaces}.
			\item $\psi_p=c_p\II_{I_p}$ indicator functions of some disjoint measureable subsets $(I_p)_{p\leq P}$ of $\partial D$, scaled to be orthonormal.
			\item $\pd{}{\nu}$ the outward normal derivative at the boundary of a domain (i.e.\ usually on $\partial D$) defined in a trace sense. $\tr$ denotes the usual trace map, taking $V : D\to \CC$ to its `boundary values' $\tr V=V|_{\partial D} : \partial D \to \CC$.
			\item $\Lambda_\gamma : H^{s+1}(\partial D)/\CC \to H^{s}_\diamond (\partial D)$ the Dirichlet-to-Neumann map, taking $f$ to $\gamma  \pd{u_{\gamma,f}}{\nu}|_{\partial D}.$
			\item  $\tilde{\Lambda}_\gamma = \Lambda_\gamma - \Lambda_1$.
			\item $\norm{\cdot}_*=\norm{\cdot}_{H^{1/2}(\partial D)/\CC\to H^{-1/2}(\partial D)}$, where $\norm{\cdot}_{A\to B}$ denotes the operator norm between Banach spaces $A$ and $B$.
			\item $\Ll(A,B)=\braces{T: A\to B \text{ linear s.t. }\norm{T}_{A\to B}<\infty}.$
			\item $\Ll_2(A,B)=\braces{ T\in \Ll(A,B) : \norm{T}_{\Ll_2(A,B)}^2:=\sum_k \norm{T e_k^{(A)}}_{B}^2 <\infty}$ the space of Hilbert--Schmidt operators from $A$ to $B$ for separable Hilbert spaces $A$ and $B$, with $(e_k^{(A)})_{k\in\NN}$ an orthonormal basis of $A$.
			\item $b_{jk}^{(r)}(f)\equiv (\phi_j^{(r)}) \otimes \phi_k^{(0)}(f)=\ip{f,\phi_j^{(r)}}_{H^r(\partial D)} \phi_k^{(0)}$, $j,k\in\NN$. These form an orthonormal basis of $\HH_r$ (considering real linear combinations) and of $\Ll_2(H^r(\partial D)/\CC,L^2_\diamond(\partial D))$ (considering complex linear combinations).
			\item $\HH_r =\braces[\big]{T:H^r(\partial D) \to L^2(\partial D),~ T=\sum_{j,k=1}^\infty t_{jk} b_{jk}^{(r)} : t_{jk}\in \RR, \sum_{j,k=1}^\infty t_{jk}^2<\infty }, %(t_{jk})\in \ell_2(\RR)}
			$ a (real) Hilbert space for the inner product
			$\ip{S,T}_{\HH_r}=\sum_{j,k=1}^\infty s_{jk}t_{jk}=\sum_{j,k=1}^\infty \ip{S\phi_j^{(r)},\phi_k^{(0)}}_{L^2(\partial D)} \ip{T\phi_j^{(r)},\phi_k^{(0)}}_{L^2(\partial D)} .$ We can also view $\HH_r$ as the subset (closed under vector addition and \emph{real} scalar multiplication) of $\Ll_2(H^r(\partial D)/\CC,L^2_\diamond(\partial D))$ comprising those operators which map real-valued functions to real-valued functions.
			\item $Y=\tilde{\Lambda}_\gamma +\eps \WW$ the observed data in model \cref{eqn:Model:HilbertSpaceMeasurements}, where $\WW$ is a Gaussian white noise indexed by the Hilbert space ${\HH_r}$, and $\eps$ is a noise level which tends to zero for our asymptotic results. Concrete interpretation of this model given in \cref{eqn:ContinuousModelDistributionalSense}.
			\item $P^\gamma_\eps=P^\gamma_{\eps,r}$ the law of $Y$, $E^\gamma_\eps$ the corresponding expectation operator, $\Var_{\gamma}$ the corresponding variance operator.
			\item $p^\gamma_\eps(Y)=\exp\brackets[\big]{\frac{1}{\eps^2}\ip{Y,\tilde{\Lambda}_\gamma}_{\HH_r} - \frac{1}{2\eps^2}\norm{\tilde{\Lambda}_\gamma}_{\HH_r}^2}$ the probability density of the law of $Y$ w.r.t.\ the law of $\eps\WW$.
			\item $\ell(\gamma)=\log p^\gamma_\eps$ the log-likelihood function.
			\item $\xi_{\eps,\delta}=\brackets{\log(\eps^{-1})}^{-\delta}$ for $\eps,\delta>0$.
			\item $\Pi$ a prior for $\theta\in C_u(D)$ described in \cref{sec:PriorConstruction}. $\Pi$ also denotes the induced prior on $\gamma=\Phi\circ \theta \in \Gamma_{m_1,D_1}$ and the induced prior on $\Lambda_\gamma$.
			\item $\Pi(\mathrel{\;\cdot\;} \mid Y)$ the corresponding posterior. $E^{\Pi}\sqbrackets{\mathrel{\;\cdot\;} \mid Y}$ the posterior expectation. \item $(\Hh,\norm{\cdot}_\Hh)\subseteq (H^\alpha(D),\norm{\cdot}_{H^\alpha(D)})$ the RKHS of a base prior $\Pi'$ from which $\Pi$ is constructed.
			\item $\Phi$ a `regular link function' used in the prior construction (see \cref{sec:PriorConstruction}). 
			\item $\zeta: D\to [0,1]$ a smooth cutoff function used in the prior construction.
			\item $\pi_{JK}$ the ${\HH_r}$-orthogonal projection map onto $\Span\braces{b_{jk}^{(r)} : j\leq J, k\leq K}$ (see \cref{eqn:def:PiJK}).   
			\item $K(P,Q)=K(p,q)=E_{X\sim p} \log((p/q)(X))$ for distributions $P,Q$ with densities $p,q$ (the \emph{Kullback--Leibler}, or just KL, divergence)
			\item $B_{KL}^\eps(\eta) = \braces{\theta\in C_u(D) : K(p^{\theta_0}_\eps,p^{\Phi\circ \theta}_\eps)\leq (\eta/\eps)^2, \Var_{\gamma_0}(\log(p^{\theta_0}_\eps/p^{\Phi\circ\theta}_\eps))\leq (\eta/\eps)^2}.$
			\item $N(S,\rho,\delta)$ the covering numbers of the set $S$ for metric $\rho$, i.e.\ the smallest number of $\rho$-balls of radius $\delta$ needed to cover $S$.
		\end{description}

		\hypersetup{allcolors=}
		\printbibliography %biblatex

@misc{Monard2019,
       author = {{Monard}, Fran{\c{c}}ois and {Nickl}, Richard and
         {Paternain}, Gabriel P.},
        title = "{Consistent Inversion of Noisy Non-Abelian X-Ray Transforms}",
      journal = {},
          eid = {arXiv:1905.00860},
  archivePrefix = {arXiv},
       eprint = {1905.00860},
             YEAR = {2019},
       adsurl = {https://ui.adsabs.harvard.edu/abs/2019arXiv190500860M}
}

@article {N88,
    AUTHOR = {Novikov, R. G.},
     TITLE = {A multidimensional inverse spectral problem for the equation
              {$-\Delta\psi +(v(x)-Eu(x))\psi=0$}},
   JOURNAL = {Funktsional. Anal. i Prilozhen.},
  FJOURNAL = {Akademiya Nauk SSSR. Funktsional\cprime ny\u{\i} Analiz i ego
              Prilozheniya},
    VOLUME = {22},
      YEAR = {1988},
    NUMBER = {4},
     PAGES = {11--22, 96},
      ISSN = {0374-1990},
   MRCLASS = {35R30 (35J05 35P05)},
  MRNUMBER = {976992},
MRREVIEWER = {A. G. Ramm},
       DOI = {10.1007/BF01077418},
       URL = {https://doi-org.ezp.lib.cam.ac.uk/10.1007/BF01077418},
}

@article {AV05,
    AUTHOR = {Alessandrini, Giovanni and Vessella, Sergio},
     TITLE = {Lipschitz stability for the inverse conductivity problem},
   JOURNAL = {Adv. in Appl. Math.},
  FJOURNAL = {Advances in Applied Mathematics},
    VOLUME = {35},
      YEAR = {2005},
    NUMBER = {2},
     PAGES = {207--241},
      ISSN = {0196-8858},
   MRCLASS = {35R30 (35B35 35J25 78A99)},
  MRNUMBER = {2152888},
MRREVIEWER = {Paul Andrew Martin},
       DOI = {10.1016/j.aam.2004.12.002},
       URL = {https://doi.org/10.1016/j.aam.2004.12.002},
}

@misc{GN19,
       author = {{Giordano}, Matteo and {Nickl}, Richard},
        title = "{Consistency of Bayesian inference with Gaussian process priors in an elliptic inverse problem}",
      journal = {},
          eid = {arXiv:1910.07343},
  archivePrefix = {arXiv},
       eprint = {1905.00860},
             YEAR = {2019},
      }

@article {N11,
    AUTHOR = {Novikov, R. G.},
     TITLE = {New global stability estimates for the {G}el'fand-{C}alderon
              inverse problem},
   JOURNAL = {Inverse Problems},
  FJOURNAL = {Inverse Problems. An International Journal on the Theory and
              Practice of Inverse Problems, Inverse Methods and Computerized
              Inversion of Data},
    VOLUME = {27},
      YEAR = {2011},
    NUMBER = {1},
     PAGES = {015001, 21},
      ISSN = {0266-5611},
   MRCLASS = {65N21 (35J25 35R30)},
  MRNUMBER = {2746404},
MRREVIEWER = {Barbara Kaltenbacher},
       DOI = {10.1088/0266-5611/27/1/015001},
       URL = {https://doi.org/10.1088/0266-5611/27/1/015001},
}

@article {CG17,
    AUTHOR = {Caro, Pedro and Garcia, Andoni},
     TITLE = {The {C}alder\'{o}n problem with corrupted data},
   JOURNAL = {Inverse Problems},
  FJOURNAL = {Inverse Problems. An International Journal on the Theory and
              Practice of Inverse Problems, Inverse Methods and Computerized
              Inversion of Data},
    VOLUME = {33},
      YEAR = {2017},
    NUMBER = {8},
     PAGES = {085001, 17},
      ISSN = {0266-5611},
   MRCLASS = {78A48 (35J25 35R30 65N21)},
  MRNUMBER = {3663121},
MRREVIEWER = {Hideo Soga},
       DOI = {10.1088/1361-6420/aa7425},
       URL = {https://doi.org/10.1088/1361-6420/aa7425},
}

@article {KV87,
    AUTHOR = {Kohn, Robert V. and Vogelius, Michael},
     TITLE = {Relaxation of a variational method for impedance computed
              tomography},
   JOURNAL = {Comm. Pure Appl. Math.},
  FJOURNAL = {Communications on Pure and Applied Mathematics},
    VOLUME = {40},
      YEAR = {1987},
    NUMBER = {6},
     PAGES = {745--777},
      ISSN = {0010-3640},
   MRCLASS = {35R30 (65P05 92A07 94C05)},
  MRNUMBER = {910952},
MRREVIEWER = {Gunther A. Uhlmann},
       DOI = {10.1002/cpa.3160400605},
       URL = {https://doi.org/10.1002/cpa.3160400605},
}

@article {HKQ18,
    AUTHOR = {Hinze, Michael and Kaltenbacher, Barbara and Quyen, Tran Nhan
              Tam},
     TITLE = {Identifying conductivity in electrical impedance tomography
              with total variation regularization},
   JOURNAL = {Numer. Math.},
  FJOURNAL = {Numerische Mathematik},
    VOLUME = {138},
      YEAR = {2018},
    NUMBER = {3},
     PAGES = {723--765},
      ISSN = {0029-599X},
   MRCLASS = {65N21 (35J25 35R30 65N12 78A46 94A08)},
  MRNUMBER = {3767699},
MRREVIEWER = {Pedro Serranho},
       DOI = {10.1007/s00211-017-0920-8},
       URL = {https://doi.org/10.1007/s00211-017-0920-8},
}

@article {NS10,
    AUTHOR = {Novikov, Roman and Santacesaria, Matteo},
     TITLE = {A global stability estimate for the {G}elfand-{C}alder\'{o}n
              inverse problem in two dimensions},
   JOURNAL = {J. Inverse Ill-Posed Probl.},
  FJOURNAL = {Journal of Inverse and Ill-Posed Problems},
    VOLUME = {18},
      YEAR = {2010},
    NUMBER = {7},
     PAGES = {765--785},
      ISSN = {0928-0219},
   MRCLASS = {35R30 (35B35 35J25)},
  MRNUMBER = {2772570},
MRREVIEWER = {Bruno Volzone},
       DOI = {10.1515/JIIP.2011.003},
       URL = {https://doi.org/10.1515/JIIP.2011.003},
}

@article {N96,
    AUTHOR = {Nachman, Adrian I.},
     TITLE = {Global uniqueness for a two-dimensional inverse boundary value
              problem},
   JOURNAL = {Ann. of Math. (2)},
  FJOURNAL = {Annals of Mathematics. Second Series},
    VOLUME = {143},
      YEAR = {1996},
    NUMBER = {1},
     PAGES = {71--96},
      ISSN = {0003-486X},
   MRCLASS = {35R30 (35J15)},
  MRNUMBER = {1370758},
MRREVIEWER = {Yongzhi Xu},
       DOI = {10.2307/2118653},
       URL = {https://doi.org/10.2307/2118653},
}

@article {NS17b,
 author = {{Nickl}, Richard and {S\"ohl}, Jakob},
  title = "{Bernstein-von Mises theorems for statistical inverse problems II: compound Poisson processes}",
   JOURNAL = {Electronic J. Statist.},
    VOLUME = {13},
      YEAR = {2019},
       PAGES = {3513--3571},
}

@article {GN11,
    AUTHOR = {Gin\'{e}, Evarist and Nickl, Richard},
     TITLE = {Rates of contraction for posterior distributions in
              {$L^r$}-metrics, {$1\leq r\leq\infty$}},
   JOURNAL = {Ann. Statist.},
  FJOURNAL = {The Annals of Statistics},
    VOLUME = {39},
      YEAR = {2011},
    NUMBER = {6},
     PAGES = {2883--2911},
      ISSN = {0090-5364},
   MRCLASS = {62G20 (62F15 62G07 62G08)},
  MRNUMBER = {3012395},
MRREVIEWER = {Thomas R. Boucher},
       DOI = {10.1214/11-AOS924},
       URL = {https://doi.org/10.1214/11-AOS924},
}

@article {KKSV00,
    AUTHOR = {Kaipio, Jari P. and Kolehmainen, Ville and Somersalo, Erkki
              and Vauhkonen, Marko},
     TITLE = {Statistical inversion and {M}onte {C}arlo sampling methods in
              electrical impedance tomography},
   JOURNAL = {Inverse Problems},
  FJOURNAL = {Inverse Problems. An International Journal on the Theory and
              Practice of Inverse Problems, Inverse Methods and Computerized
              Inversion of Data},
    VOLUME = {16},
      YEAR = {2000},
    NUMBER = {5},
     PAGES = {1487--1522},
      ISSN = {0266-5611},
   MRCLASS = {78A99 (65C05 78M25 92C55)},
  MRNUMBER = {1800606},
       DOI = {10.1088/0266-5611/16/5/321},
       URL = {https://doi.org/10.1088/0266-5611/16/5/321},
}

@article {CRSW13,
    AUTHOR = {Cotter, S. L. and Roberts, G. O. and Stuart, A. M. and White,
              D.},
     TITLE = {M{CMC} methods for functions: modifying old algorithms to make
              them faster},
   JOURNAL = {Statist. Sci.},
  FJOURNAL = {Statistical Science. A Review Journal of the Institute of
              Mathematical Statistics},
    VOLUME = {28},
      YEAR = {2013},
    NUMBER = {3},
     PAGES = {424--446},
      ISSN = {0883-4237},
   MRCLASS = {62F15 (60G60 60J22 62G05 62G07 65C05)},
  MRNUMBER = {3135540},
MRREVIEWER = {Sreenivasan Ravi},
       DOI = {10.1214/13-STS421},
       URL = {https://doi.org/10.1214/13-STS421},
}

@article {HSV14,
    AUTHOR = {Hairer, Martin and Stuart, Andrew M. and Vollmer, Sebastian
              J.},
     TITLE = {Spectral gaps for a {M}etropolis-{H}astings algorithm in
              infinite dimensions},
   JOURNAL = {Ann. Appl. Probab.},
  FJOURNAL = {The Annals of Applied Probability},
    VOLUME = {24},
      YEAR = {2014},
    NUMBER = {6},
     PAGES = {2455--2490},
      ISSN = {1050-5164},
   MRCLASS = {60J22 (37A30 60B12 60J05 65C05 65C40)},
  MRNUMBER = {3262508},
MRREVIEWER = {Laurent Miclo},
       DOI = {10.1214/13-AAP982},
       URL = {https://doi.org/10.1214/13-AAP982},
}

@article {vdVvZ08,
    AUTHOR = {van der Vaart, A. W. and van Zanten, J. H.},
     TITLE = {Rates of contraction of posterior distributions based on
              {G}aussian process priors},
   JOURNAL = {Ann. Statist.},
  FJOURNAL = {The Annals of Statistics},
    VOLUME = {36},
      YEAR = {2008},
    NUMBER = {3},
     PAGES = {1435--1463},
      ISSN = {0090-5364},
   MRCLASS = {62G05 (60G15)},
  MRNUMBER = {2418663},
MRREVIEWER = {Theofanis Sapatinas},
       DOI = {10.1214/009053607000000613},
       URL = {https://doi.org/10.1214/009053607000000613},
}

@article {BM12,
    AUTHOR = {Jin, Bangti and Maass, Peter},
     TITLE = {An analysis of electrical impedance tomography with
              applications to {T}ikhonov regularization},
   JOURNAL = {ESAIM Control Optim. Calc. Var.},
  FJOURNAL = {ESAIM. Control, Optimisation and Calculus of Variations},
    VOLUME = {18},
      YEAR = {2012},
    NUMBER = {4},
     PAGES = {1027--1048},
      ISSN = {1292-8119},
   MRCLASS = {49N45 (65N21)},
  MRNUMBER = {3019471},
}

@article {NS17,
    AUTHOR = {Nickl, Richard and S\"{o}hl, Jakob},
     TITLE = {Nonparametric {B}ayesian posterior contraction rates for
              discretely observed scalar diffusions},
   JOURNAL = {Ann. Statist.},
  FJOURNAL = {The Annals of Statistics},
    VOLUME = {45},
      YEAR = {2017},
    NUMBER = {4},
     PAGES = {1664--1693},
      ISSN = {0090-5364},
   MRCLASS = {62G05 (60J60 62F15 62G20)},
  MRNUMBER = {3670192},
       DOI = {10.1214/16-AOS1504},
       URL = {https://doi.org/10.1214/16-AOS1504},
}

@article {Dunlop2016,
    AUTHOR = {Dunlop, Matthew M. and Stuart, Andrew M.},
     TITLE = {The {B}ayesian formulation of {EIT}: analysis and algorithms},
   JOURNAL = {Inverse Probl. Imaging},
  FJOURNAL = {Inverse Problems and Imaging},
    VOLUME = {10},
      YEAR = {2016},
    NUMBER = {4},
     PAGES = {1007--1036},
      ISSN = {1930-8337},
   MRCLASS = {94A08},
  MRNUMBER = {3610749},
       DOI = {10.3934/ipi.2016030},
       URL = {https://doi.org/10.3934/ipi.2016030},
}

@article {Roininen2014,
    AUTHOR = {Roininen, Lassi and Huttunen, Janne M. J. and Lasanen, Sari},
     TITLE = {Whittle-{M}at\'{e}rn priors for {B}ayesian statistical inversion
              with applications in electrical impedance tomography},
   JOURNAL = {Inverse Probl. Imaging},
  FJOURNAL = {Inverse Problems and Imaging},
    VOLUME = {8},
      YEAR = {2014},
    NUMBER = {2},
     PAGES = {561--586},
      ISSN = {1930-8337},
   MRCLASS = {65N21 (62F15)},
  MRNUMBER = {3209311},
       DOI = {10.3934/ipi.2014.8.561},
       URL = {https://doi.org/10.3934/ipi.2014.8.561},
}

@article {Kolehmainen2013,
    AUTHOR = {Kolehmainen, Ville and Lassas, Matti and Ola, Petri and
              Siltanen, Samuli},
     TITLE = {Recovering boundary shape and conductivity in electrical
              impedance tomography},
   JOURNAL = {Inverse Probl. Imaging},
  FJOURNAL = {Inverse Problems and Imaging},
    VOLUME = {7},
      YEAR = {2013},
    NUMBER = {1},
     PAGES = {217--242},
      ISSN = {1930-8337},
   MRCLASS = {35R30 (35J25 65N21)},
  MRNUMBER = {3031845},
MRREVIEWER = {Akhtar A. Khan},
       DOI = {10.3934/ipi.2013.7.217},
       URL = {https://doi.org/10.3934/ipi.2013.7.217},
}

@article {Stuart2010,
    AUTHOR = {Stuart, A. M.},
     TITLE = {Inverse problems: a {B}ayesian perspective},
   JOURNAL = {Acta Numer.},
  FJOURNAL = {Acta Numerica},
    VOLUME = {19},
      YEAR = {2010},
     PAGES = {451--559},
      ISSN = {0962-4929},
   MRCLASS = {65J22 (35R25 35R30 62C10 65J20)},
  MRNUMBER = {2652785},
MRREVIEWER = {Ruben D. Spies},
       DOI = {10.1017/S0962492910000061},
       URL = {https://doi.org/10.1017/S0962492910000061},
}

@article {Uhlmann2009,
    AUTHOR = {Uhlmann, G.},
     TITLE = {Electrical impedance tomography and {C}alder\'{o}n's problem},
   JOURNAL = {Inverse Problems},
  FJOURNAL = {Inverse Problems. An International Journal on the Theory and
              Practice of Inverse Problems, Inverse Methods and Computerized
              Inversion of Data},
    VOLUME = {25},
      YEAR = {2009},
    NUMBER = {12},
     PAGES = {123011, 39},
      ISSN = {0266-5611},
   MRCLASS = {78A48 (35-02 35J25 35R30)},
  MRNUMBER = {3460047},
MRREVIEWER = {Sergey G. Pyatkov},
       DOI = {10.1088/0266-5611/25/12/123011},
       URL = {https://doi.org/10.1088/0266-5611/25/12/123011},
}

@article {Nachman1988,
    AUTHOR = {Nachman, Adrian I.},
     TITLE = {Reconstructions from boundary measurements},
   JOURNAL = {Ann. of Math. (2)},
  FJOURNAL = {Annals of Mathematics. Second Series},
    VOLUME = {128},
      YEAR = {1988},
    NUMBER = {3},
     PAGES = {531--576},
      ISSN = {0003-486X},
   MRCLASS = {35R30 (35J05 35P25)},
  MRNUMBER = {970610},
MRREVIEWER = {Gunther A. Uhlmann},
       DOI = {10.2307/1971435},
       URL = {https://doi.org/10.2307/1971435},
}

@article {Sylvester1987,
    AUTHOR = {Sylvester, John and Uhlmann, Gunther},
     TITLE = {A global uniqueness theorem for an inverse boundary value
              problem},
   JOURNAL = {Ann. of Math. (2)},
  FJOURNAL = {Annals of Mathematics. Second Series},
    VOLUME = {125},
      YEAR = {1987},
    NUMBER = {1},
     PAGES = {153--169},
      ISSN = {0003-486X},
   MRCLASS = {35R30 (86A20)},
  MRNUMBER = {873380},
MRREVIEWER = {P. Szeptycki},
       DOI = {10.2307/1971291},
       URL = {https://doi.org/10.2307/1971291},
}

@incollection {Calderon1980,
    AUTHOR = {Calder\'{o}n, Alberto-P.},
     TITLE = {On an inverse boundary value problem},
 BOOKTITLE = {Seminar on {N}umerical {A}nalysis and its {A}pplications to
              {C}ontinuum {P}hysics ({R}io de {J}aneiro, 1980)},
     PAGES = {65--73},
 PUBLISHER = {Soc. Brasil. Mat., Rio de Janeiro},
      YEAR = {1980},
   MRCLASS = {35R30 (35K60)},
  MRNUMBER = {590275},
MRREVIEWER = {J. R. Cannon},
}

@inproceedings{Diaconis1988,
author = {Diaconis, Persi},
booktitle = {Statistical Decision Theory and Related Topics IV},
editor = {Berger, J. and Gupta, S.},
pages = {163--175},
publisher = {Springer, New York},
title = {{Bayesian Numerical Analysis}},
url = {https://statweb.stanford.edu/{~}cgates/PERSI/papers/Bayes{\_}numerical88.pdf},
year = {1988}
}

@book{S08,
author = {Salo, Mikko},
title = {Calder{\'o}n problem, lecture notes},
year = {2008},
note = {\url{http://users.jyu.fi/~salomi/lecturenotes/calderon_lectures.pdf}}
}

@article {V13,
    AUTHOR = {Vollmer, Sebastian J.},
     TITLE = {Posterior consistency for {B}ayesian inverse problems through
              stability and regression results},
   JOURNAL = {Inverse Problems},
  FJOURNAL = {Inverse Problems. An International Journal on the Theory and
              Practice of Inverse Problems, Inverse Methods and Computerized
              Inversion of Data},
    VOLUME = {29},
      YEAR = {2013},
    NUMBER = {12},
     PAGES = {125011, 32},
      ISSN = {0266-5611},
   MRCLASS = {35R30 (62F15 62G08 65J22 65N21)},
  MRNUMBER = {3141858},
MRREVIEWER = {Vyacheslav F. Gubarev},
       DOI = {10.1088/0266-5611/29/12/125011},
       URL = {https://doi.org/10.1088/0266-5611/29/12/125011},
}

@book{LeCam1986,
author = {Le Cam, Lucien},
title = {Asymptotic Methods in Statistical Decision Theory},
year = {1986},
month = {01},
pages = {},
publisher = {Springer, New York}
}

@article{Alessandrini1988,
author = { Giovanni   Alessandrini },
title = {Stable determination of conductivity by boundary measurements},
journal = {Applicable Analysis},
volume = {27},
number = {1-3},
pages = {153-172},
year  = {1988},
publisher = {Taylor & Francis},
doi = {10.1080/00036818808839730},

URL = { 
        https://doi.org/10.1080/00036818808839730
    
},
eprint = { 
        https://doi.org/10.1080/00036818808839730
    
}

}

@Book{P1896,
  Title                    = {{Calcul des probabilit{\'e}s}},
  Author                   = {Poincar{\'e}, Henri},
  Publisher                = {Gautier-Villars, Paris},
  Year                     = {1912},
}

@Book{Chavel1984,
  Title                    = {{Eigenvalues in Riemannian Geometry}},
  Author                   = {Chavel, Isaac},
   Publisher                = {Elsevier},
  Year                     = {1984},
}

@Book{Evans1998,
  Title                    = {{Partial Differential Equations}},
  Author                   = {Evans, L C},
  Publisher                = {American Mathematical Society},
  Year                     = {1998},
  ISBN                     = {9780821807729},
  Url                      = {https://books.google.co.uk/books?id=5Pv4LVB{\_}m8AC}
}

@Article{Gehre2014,
  Title                    = {{Expectation propagation for nonlinear inverse problems -- with an application to electrical impedance tomography}},
  Author                   = {Gehre, Matthias and Jin, Bangti},
  Journal                  = {Journal of Comp. Phys.},
  Year                     = {2014},
  Pages                    = {513--535},
  Volume                   = {259},
  Doi                      = {https://doi.org/10.1016/j.jcp.2013.12.010},
  ISSN                     = {0021-9991},
  Url                      = {http://www.sciencedirect.com/science/article/pii/S0021999113008097}
}

@Article{Geller2011,
  Title                    = {{Band-Limited Localized Parseval Frames and Besov Spaces on Compact Homogeneous Manifolds}},
  Author                   = {Geller, Daryl and Pesenson, Isaac Z},
  Journal                  = {Journal of Geometric Analysis},
  Year                     = {2011},

  Month                    = {apr},
  Number                   = {2},
  Pages                    = {334--371},
  Volume                   = {21},
  Doi                      = {10.1007/s12220-010-9150-3},
  ISSN                     = {1559-002X},
  Url                      = {https://doi.org/10.1007/s12220-010-9150-3}
}

@book {T09,
    AUTHOR = {Tsybakov, Alexandre B.},
     TITLE = {Introduction to nonparametric estimation},
      PUBLISHER = {Springer, New York},
      YEAR = {2009},
      ISBN = {978-0-387-79051-0},
   MRCLASS = {62-01 (62G05 62G07 62G08 62G20)},
  MRNUMBER = {2724359},
       DOI = {10.1007/b13794},
       URL = {https://doi.org/10.1007/b13794},
}

@Book{Ghosal2017,
  Title                    = {{Fundamentals of nonparametric Bayesian inference}},
  Author                   = {Ghosal, Subhashis and van der Vaart, Aad},
  Publisher                = {Cambridge University Press},
  Year                     = {2017},

  Address                  = {Cambridge},
  ISBN                     = {9780521878265}
}

@Book{Gine2016,
  Title                    = {Mathematical Foundations of Infinite-Dimensional Statistical Models},
  Author                   = {Gin\'{e}, Evarist and Nickl, Richard},
  Publisher                = {Cambridge University Press},
  Year                     = {2016},

  Address                  = {Cambridge},

  Doi                      = {10.1017/CBO9781107337862},
  ISBN                     = {9781107337862},
  Url                      = {http://ebooks.cambridge.org/ref/id/CBO9781107337862}
}

@Article{Hanke2011,
  Title                    = {{Convex backscattering support in electric impedance tomography}},
  Author                   = {Hanke, Martin and Hyv{\"{o}}nen, Nuutti and Reusswig, Stefanie},
  Journal                  = {Numerische Mathematik},
  Year                     = {2011},

  Month                    = {feb},
  Number                   = {2},
  Pages                    = {373--396},
  Volume                   = {117},
  Doi                      = {10.1007/s00211-010-0320-9},
  ISSN                     = {0945-3245},
  Url                      = {https://doi.org/10.1007/s00211-010-0320-9}
}

@Article{Li1999,
  Title                    = {{Approximation, Metric Entropy and Small Ball Estimates for Gaussian Measures}},
  Author                   = {Li, Wenbo V and Linde, Werner},
  Journal                  = {Ann. Probab.},
  Year                     = {1999},
  Number                   = {3},
  Pages                    = {1556--1578},
  Volume                   = {27},

  Doi                      = {10.1214/aop/1022677459},
  Publisher                = {The Institute of Mathematical Statistics},
  Url                      = {https://doi.org/10.1214/aop/1022677459}
}

@Book{Lions1972,
  Title                    = {{Non-Homogeneous Boundary Value Problems and Applications}},
  Author                   = {Lions, J. L. and Magenes, E.},
  Publisher                = {Springer Berlin Heidelberg},
  Year                     = {1972},

  Address                  = {Berlin, Heidelberg},

  Doi                      = {10.1007/978-3-642-65161-8},
  ISBN                     = {978-3-642-65163-2},
  Url                      = {http://link.springer.com/10.1007/978-3-642-65161-8}
}

@Article{Mandache2001,
  Title                    = {{Exponential instability in an inverse problem for the Schr{\"{o}}dinger equation}},
  Author                   = {Mandache, Niculae},
  Journal                  = {Inverse Problems},
  Year                     = {2001},

  Month                    = {aug},
  Number                   = {5},
  Pages                    = {1435--1444},
  Volume                   = {17},
  Doi                      = {10.1088/0266-5611/17/5/313},
  Publisher                = {{\{}IOP{\}} Publishing},
  Url                      = {https://doi.org/10.1088{\%}2F0266-5611{\%}2F17{\%}2F5{\%}2F313}
}

@Article{Mariucci2016,
    AUTHOR = {Mariucci, Ester},
     TITLE = {Le {C}am theory on the comparison of statistical models},
   JOURNAL = {Grad. J. Math.},
  FJOURNAL = {The Graduate Journal of Mathematics},
    VOLUME = {1},
      YEAR = {2016},
    NUMBER = {2},
     PAGES = {81--91},
      ISSN = {1737-0299},
   MRCLASS = {62B15 (62G07 62G20)},
  MRNUMBER = {3850766},
}

@Article{Nickl2017b,
  Title                    = {{Bernstein--von Mises theorems for statistical inverse problems I: Schr{\"{o}}dinger equation}},
  Author                   = {Nickl, Richard},
  Journal                  = {Journal of the European Mathematical Society},
  Year                   = {to appear},
  Archiveprefix            = {arXiv},
  Arxivid                  = {math.ST/1707.01764},
  Eprint                   = {1707.01764},
  Primaryclass             = {math.ST},
  Url                      = {https://arxiv.org/abs/1707.01764}
}

@Article{Nickl2018,
  Title                    = {Convergence rates for Penalised Least Squares Estimators in PDE- constrained regression problems},
  Author                   = {Nickl, Richard and van de Geer, Sara and Wang, Sven},
  Journal                  = {SIAM/ASA Journal of Uncertainty Quantification},
  Year                     = {to appear},
  Archiveprefix            = {arXiv},
  Arxivid                  = {math.ST/1809.08818},
  Eprint                   = {1809.08818},
  Primaryclass             = {math.ST}
}

@Article{Reiss2008,
  Title                    = {{Asymptotic Equivalence for Nonparametric Regression with Multivariate and Random Design}},
  Author                   = {Rei{\ss}, Markus},
  Journal                  = {The Annals of Statistics},
  Year                     = {2008},
  Number                   = {4},
  Pages                    = {1957--1982},
  Volume                   = {36},
  ISSN                     = {00905364},
  Publisher                = {Institute of Mathematical Statistics},
  Url                      = {http://www.jstor.org/stable/25464697}
}

@book {AF03,
    AUTHOR = {Adams, Robert A. and Fournier, John J. F.},
     TITLE = {Sobolev spaces},
 PUBLISHER = {Elsevier, Amsterdam},
      YEAR = {2003},
      ISBN = {0-12-044143-8},
   MRCLASS = {46E35 (46-01 46-02 46B70 46Exx)},
  MRNUMBER = {2424078},
}

@Book{Aubin2011,
  Title                    = {Applied Functional Analysis},
  Author               ={Aubin, J.P.},
  Publisher                = {John Wiley {\&} Sons, Ltd},
  Year                     = {2011},
  Booktitle                = {Applied Functional Analysis},
  Doi                      = {10.1002/9781118032725.ch12},
  ISBN                     = {9781118032725},
  Url                      = {https://onlinelibrary.wiley.com/doi/abs/10.1002/9781118032725.ch12}
}
		%\bibliography{./References/CalderonReferences} %natbib
		
		\noindent \textsc{D\'epartment de Math\'ematiques d'Orsay}\\
		\textsc{Universit\'e Paris-Sud}\\ 
		\textsc{91405 Orsay Cedex, France} \\
		\text{email: kweku.abraham@universite-paris-saclay.fr}
		
		\smallskip
		
		\noindent \textsc{Department of Pure Mathematics and Mathematical Statistics}\\
		\textsc{University of Cambridge}\\ 
		\textsc{Wilberforce Road, CB3 0WB Cambridge, UK} \\
		\text{email: r.nickl@statslab.cam.ac.uk}

	\end{document}